\documentclass[10pt]{amsproc}

\usepackage{amscd,amssymb,graphics,a4wide,color,hyperref,fancyhdr}
\usepackage{mathrsfs}
\usepackage{upgreek}

\input xy
\xyoption{all}

\numberwithin{equation}{section}

\newtheorem{theorem}{Theorem}[section]
\newtheorem{lemma}[theorem]{Lemma}
\newtheorem{proposition}[theorem]{Proposition}
\newtheorem{corollary}[theorem]{Corollary}

\theoremstyle{definition}
\newtheorem{definition}[theorem]{Definition}

\newtheorem{example}[theorem]{Example}

\theoremstyle{remark}
\newtheorem{remark}[theorem]{Remark}

\newcommand{\ra}{\to}

\newcommand{\lra}{\longrightarrow}
\newcommand{\sra}{\twoheadrightarrow}
\newcommand{\hra}{\hookrightarrow}

\newcommand{\IFF}{\Leftrightarrow}

\newcommand{\shC}{{\mathcal{C}}}
\newcommand{\shF}{{\mathcal{F}}}

\newcommand{\shHom}{{\mathcal{H}om}}
\newcommand{\shI}{{\mathcal{I}}}
\newcommand{\shJ}{{\mathcal{J}}}
\newcommand{\shK}{{\mathcal{K}}}
\newcommand{\shL}{{\mathcal{L}}}

\newcommand{\shO}{{\mathcal{O}}}
\newcommand{\shQ}{{\mathcal{Q}}}
\newcommand{\shS}{{\mathcal{S}}}
\newcommand{\shT}{{\mathcal{T}}}

\newcommand{\shW}{{\mathcal{W}}}

\newcommand{\SC}{{\mathscr C}}

\newcommand{\op}[1]{{\operatorname{#1}}}

\newcommand{\CC}{{\mathbb{C}}}
\newcommand{\HH}{\mathbb{H}}
\newcommand{\NN}{\mathbb{N}}
\newcommand{\PP}{\mathbb{P}}

\newcommand{\RR}{\mathbb{R}}
\newcommand{\ZZ}{\mathbb{Z}}

\newcommand  {\aff}{{\op{aff}}}
\newcommand  {\Aut}{\op{Aut}\,}
\renewcommand{\atop}[2]{{\genfrac{}{}{0pt}{}{#1}{#2}}}
\newcommand  {\bct}{{\op{bct}}}
\newcommand  {\bLambda}{{\bf\Lambda}}

\newcommand  {\codim}{\op{codim}\,}
\newcommand  {\coker}{\op{coker}\,}
\newcommand  {\core}[1]{{\backslash #1 /}}
\newcommand  {\dual}{\vee}

\newcommand  {\eps}{\epsilon}
\newcommand  {\Hom}{\op{Hom}\,}
\newcommand  {\id}{\op{id}\,}
\newcommand  {\im}{\op{im}\,}
\renewcommand{\k}{{\Bbbk}}
\newcommand  {\kk}{\Bbbk}
\renewcommand{\ker}{\op{kern}\,}
\newcommand  {\loc}{{\tiny \op{loc}}}
\renewcommand{\log}{\dagger}
\newcommand  {\llog}{{\op{log}}}
\newcommand  {\Newton}{\op{Newton}\,}
\newcommand  {\ocore}[1]{{\backslash \overset{\circ}{#1} /}}
\renewcommand{\P}{\mathscr{P}}

\newcommand  {\Quot}{\op{Quot}\,}
\newcommand  {\rank}{\op{rank}\,}
\newcommand  {\relint}{\op{relint}\,}
\newcommand  {\red}{\op{red}\,}
\newcommand  {\res}{\op{res}\,}
\newcommand  {\rk}{\op{rk}\,}
\renewcommand{\span}[1]{\op{span}_{#1}\,}
\newcommand  {\Sing}{\op{Sing}\,}
\newcommand  {\Spec}{\op{Spec}\,}
\newcommand  {\st}{{\op{st}}}
\renewcommand{\top}{{\op{top}}}
\newcommand  {\Tors}{{\shT\!ors}}
\newcommand  {\Tot}{{\op{Tot}}}

\newcommand  {\Ver}{\op{Vert}\,}
\newcommand  {\Int}{\op{Int}\,}
\newcommand  {\X}{\mathcal{X}}
\newcommand  {\Z}{\mathcal{Z}}

\begin{document}

\title[Log Hodge groups on a toric Calabi-Yau degeneration]
{Log Hodge groups on a toric\\ Calabi-Yau degeneration}
\author{Helge Ruddat} 
\address{Mathematisches Institut,
Universit\"at Freiburg, Eckerstr.1, 79104 Freiburg, Germany}
\email{helge.ruddat@math.uni-freiburg.de}

\date{June 14, 2009}
 		

\begin{abstract}
We give a spectral sequence to compute the logarithmic Hodge groups 
on a hypersurface type toric log Calabi-Yau space $X$, compute its $E_1$ term 
explicitly in terms of tropical degeneration data and Jacobian rings 
and prove its degeneration at $E_2$ under mild assumptions. 
We prove the basechange of the affine Hodge groups and deduce it
for the logarithmic Hodge groups in low dimensions.
As an application, we prove a mirror symmetry duality in dimension two and 
four involving the ordinary Hodge numbers, the stringy Hodge numbers and the affine Hodge numbers.
\end{abstract}

\maketitle
\tableofcontents
\bigskip


\section*{Introduction} \label{introduction}

Hodge theory implies that Hodge numbers stay constant in smooth,
proper families \cite{deligne0}. By using logarithmic differential forms
Steenbrink extended this result to normal crossing degenerations 
\cite{steenbrink1}. Later it was observed 
\cite{kawamatanamikawa}, \cite{illusie} that the notion of
log smoothness in abstract log geometry \cite{fkato},\cite{kkato} provides the right
framework for this kind of result.

In \cite{grosie1} and \cite{grosie3} Gross and Siebert provide a framework for a
comprehensive understanding of mirror symmetry via maximal
degenerations $\X\to S$, using the technique of log geometry. The
central fibre of their maximal degenerations are unions of complete
toric varieties, and they allow an essentially combinatorial
(``tropical'') description via an integral affine manifold $B$ with
certain singularities along with a compatible decomposition into
lattice polyhedra. While a maximal degeneration does not literally
define a log smooth morphism, it is shown in \cite{grosie2} that in many
cases there is enough log smoothness to compute the Hodge numbers of
the general fibers from the log Hodge numbers of the central fiber. 
The latter can in turn be computed on $B$:
\[
h^{p,q}(X_t)= h^{p,q}_\aff (B):= h^q(B, \bigwedge^p i_*\check\Lambda\otimes_\ZZ\RR).
\]
Here $X_t$ is a general fibre of $\X\to S$, $\Lambda$ is the sheaf of integral 
tangent vectors on the complement $B\setminus \Delta$ of the singular locus $\Delta$
of the affine structure and $i: B\setminus \Delta\to B$ is the inclusion.

Starting from dimension four this result can not always hold for one
expects stringy Hodge numbers to replace ordinary Hodge numbers
\cite{batyrev0}, \cite{batyrev3}. 
In fact, the authors of \cite{grosie2} impose the subtle condition that certain
lattice polytopes encoding the local affine monodromy are
standard simplices rather than elementary simplices. By definition a
lattice simplex is elementary if it does not contain any interior
integral points. 

In general, the stringy Hodge numbers $h_\st^{p,q}$ 
are greater than or equal to the ordinary Hodge numbers. For a not necessarily
maximal degeneration, we supplement this to
$$h^{p,q}_\aff (B)\ \le\ h^{p,q}(X_t)\ \le\ h_\st^{p,q}(X_t).$$

Moreover, we observe that mirror symmetry interchanges
differences: Let $\X\to S$ be a maximal degeneration with $n=\dim X_t=4$.
We can recover the difference of the stringy to the ordinary 
Hodge numbers on the mirror dual degeneration $\check \X\to S$ as the
difference of the ordinary to the affine Hodge numbers, i.e.,

$$h^{p,q}(X_t) - h^{p,q}_\aff (B)\ =\ h^{n-p,q}_\st (\check X_t) - h^{n-p,q}(\check X_t).$$

Note that a zero difference on the left hand side under the standard simplex 
condition of \cite{grosie2} is reflected by a smoothness condition on the right hand side 
because each lattice polytope locally describes
a toric singularity of $\check X_t$ and smoothness corresponds to a standard simplex.
Note that mirror symmetry of stringy Hodge numbers for complete intersections in
toric varieties was shown in \cite{batybor}.

More generally, we investigate the Hodge groups for non-maximal degenerations by defining
the new degeneration space classes \emph{hypersurface type (h.t.)} and 
\emph{complete intersection type (c.i.t.)}.
For instance, an anticanonically embedded general hypersurface in a Fano
toric variety yields a h.t. degeneration. To refine this to a maximal degeneration
one would have to form its MPCP resolution and possibly even blow this further up.

We relate the $h^{p,q}(X_t)$ to the logarithmic Hodge numbers of the 
central fibre $h^{p,q}_\llog(\X_0)$ and
derive a recipe to compute $h^{p,q}_\llog(\X_0)$ in terms of
$h^{p,q}_\aff (B)$ and additional contributions which we call \emph{log twisted sectors}.
The latter depend on the affine data of $B$ (the monodromy polytopes) as well as a 
continuous parameter $Z$ which is the locus of the logarithmic singularities of $\X_0$. 
Our result was inspired by \cite{bormav}, where Borisov and Mavlyutov give 
a conjectural definition of string cohomology for a hypersurface Calabi-Yau in 
a toric variety. They use toric Jacobian rings which come up in our setting as well.
More recently, Helm and Katz \cite{helmkatz} have related the cohomology of a subvariety
of a torus to the topology of the tropical variety obtained from a 
normal crossing degeneration thereof.

If $f$ is a local equation for an irreducible component $Z_\omega$ of $Z$, 
$\check\Delta_\omega$ the corresponding monodromy polytope 
and $C(\check\Delta_\omega)$ the cone over the 
polytope then the graded dimensions of the toric 
Jacobian rings $R_0(C(\check\Delta_\omega),f)$ and 
$R_1(C(\check\Delta_\omega),f)$ in the notation of \cite{bormav}
play a central role in the computation of the log twisted sectors. 
For the moment, let $X=X_\omega$ denote 
the smallest stratum of $\X_0$ containing $Z_\omega$.
For the canonical linear system $V\subseteq \Gamma(X,\shO_X(Z_\omega))$
given by $f$ and its logarithmic derivatives, we set
$R(Z_\omega)_n:=\coker\big(V\otimes \Gamma(X,\shO_X((n-1)Z_\omega))\ra \Gamma(X,\shO_X(nZ_\omega)) \big).$
If $\check\Delta_\omega$ is a simplex, we have
$V=\Gamma(X,\shO_X(Z_\omega))$, $R(Z_\omega)\cong R_0(C(\check\Delta_\omega),f)$ and
\[\dim R(Z_\omega)_n = \#\left\{ \atop{\hbox{lattice points}\hbox{ of }n\cdot \check\Delta_\omega
\hbox{ which are {\it not} a sum of a}}{\hbox{lattice point of }(n-1)\cdot\check\Delta_\omega
\hbox{ and a vertex of }\check\Delta_\omega} \right\}\]

We prove the following result for the 
logarithmic Hodge groups $H^q(X_0,\Omega_\llog^p)$ 
of the central fibre $X_0$ of a toric degeneration.

\begin{theorem} \label{intro1}
Let $X_0$ be a hypersurface type (h.t.) toric log Calabi-Yau space. 
\begin{enumerate}
\item[a)]
  For each $p$, there is a spectral sequence which computes the logarithmic Hodge groups 
  $H^{q}(\X_0,\Omega_\llog^p)$ whose
  $E_1$ term can be given explicitly in terms of
  $i_*\bigwedge^r\check\Lambda\otimes_\ZZ\k$, 
  $\bigwedge^r i_*\check\Lambda\otimes_\ZZ\k$ and 
  $R(Z_\omega\cap X_\tau)$ for various $\omega, \tau$.
\item[b)] If every $\check\Delta_\omega$ is a simplex, 
the spectral sequence is degenerate at $E_2$ and 
$$\begin{array}{rll}\dim E_2^{q,0}(\Omega_\llog^p)&=& h^{p,q}_\aff (B)\\
\sum_{k>0}\dim E_2^{q-k,k}(\Omega_\llog^p)&=& \hbox{log twisted sectors}
\end{array}$$
\end{enumerate}
\end{theorem}

To relate this to the ordinary Hodge groups of the general fibre, 
the authors of \cite{grosie2} have shown for maximal degenerations that 
\begin{equation}
H^{q}(\X_0,\Omega_\llog^p)\cong H^{q}(X_t,\Omega_\llog^p)
=H^{q}(X_t,\Omega^p_{X_t}) \label{eq_genspec}
\end{equation}
by means of a base change result for $\HH^{k}(\X,\Omega_\llog^\bullet)$.
The base change easily generalizes to c.i.t. degenerations, so that 
generalizing (\ref{eq_genspec}) is equivalent to showing that
the first hypercohomology spectral sequence computing $\HH^{k}(X_0,\Omega_\llog^\bullet)$
degenerates at $E_1$. The classical way would be to show 
that $\Omega_\llog^\bullet$ carries the structure of a cohomological 
mixed Hodge complex (\cite{deligne3},~8.1.9).
This requires the topological result that $\HH^{k}(X_0,\Omega_\llog^\bullet)$ computes the cohomology
of the Kato-Nakayama space corresponding to $X_0$ which we leave for future work. 
Instead, we show the degeneration directly under some conditions up to $\dim X=4$.
\\[2mm]
The structure of this paper is as follows.
All central results can be found in Section 1 where we quickly recall Gross and Siebert's
constructions, give the main
definitions h.t. and c.i.t. (1.1), state the result about the spectral sequence computing
the log Hodge groups (1.2) and give the base change result for the affine Hodge numbers
and its consequences for the base change of the ordinary Hodge numbers in low dimensions and
a mirror result on the stringy Hodge numbers (1.3).
In Section~2, we first derive some further consequences of the c.i.t. definition (2.1), in
particular a set of inner monodromy polytopes which we then use to generalize Gross-Siebert's
construction of local models (2.2), the exactness of $\SC^\bullet(\Omega^p)$ and
some further technical properties.
In Section~3, we treat Koszul cohomology for a semi-ample divisor $Z$ in a toric variety and
compute its cohomlogy in terms of $R(Z)$. In 3.3, we compare $R(Z)$ with Jacobian
rings in the case where the corresponding polytope is a simplex. In particular, we
give a monomial basis for the Jacobian ring of a non-degenerate $Z$.
In 3.4, we identify the intermediate cokernels of the Koszul complex with differential
forms having poles on the toric boundary and zeros along $Z$.
These coincide with the summands of $\SC^\bullet(\Omega^p)$, such that the
Koszul complex leads to a resolution of these as we show in 4.1.
In 4.2, we treat the dependency of the choice of a vertex for the Newton polytope
in the resolution. The central result about the computation of the 
log Hodge groups is then proved in 4.3-4.5.
The basechange and mirror symmetry for the twisted sectors 
is the contents of Chapter~5.

I would like to thank my PhD advisor Bernd Siebert for many useful discussions
and his support during my thesis. I also thank Mark Gross for supporting
my coming to the UCSD and helpful input and corrections.
I am grateful to Klaus Altmann, Renzo Cavalieri and Klaus Hulek for
invitations to give a talk on my work.
I had inspiring meetings with Stefan M\"uller-Stach and Stefan Waldmann.
I wish to thank the DFG and Studienstiftung des deutschen Volkes for
financial support and the Mathematische Fakult\"at der
Albert-Ludwigs-Universit\"at for a working environment.

\section{Definitions and central results} \label{results}

\subsection{Toric log Calabi-Yau spaces of hypersurface and complete intersection type}
We fix an algebraically closed field $\k$.
Recall from (\cite{grosie1}, Def. 4.1) that a \emph{toric degeneration}
is flat family $\X\ra \shS=\Spec A$ for some discrete 
valuation ring $A$ with residue field $\k$ such that
\begin{itemize}
\item[a)] the generic fibre $\X_\eta$ is a normal algebraic space,
\item[b)] the special fibre $\X_0$ is a union of toric varieties 
glued along toric boundary strata and
\item[c)] there is a closed subset $\Z\subset \X$ of relative codimension at least two, 
such that every
point in $\X\backslash \Z$ has a neighbourhood which is \'etale locally
equivalent to an affine toric variety where $\X_0$ is identified with the toric
boundary divisor and the deformation parameter is given by a monomial.
\end{itemize}

The Cartier divisor $\X_0$ in $\X$ induces a divisorial log structure on $\X$ which
one may pull back to $\X_0$ to turn it into a log space. For log structures, 
see \cite{kkato}, \cite{grosie1}. The definition of a 
\emph{toric log Calabi Yau space} [\cite{grosie1}, Def. 4.3] (short: toric log CY space) 
is precisely made such that $\X_0$ with its log structure is the key example. 
The authors of \cite{grosie1} demonstrate how to derive 
the \emph{dual intersection complex} $(B,\P)$ from $\X_0$ which is a
real affine manifold $B$ with singularities in codimension two and a polyhedral
decomposition $\P$ with some further properties.
Given \emph{lifted open glueing data} $s$ for $(B,\P)$, one can reconstruct $\X_0$
from the triple $(B,\P,s)$. One might even start directly with such a triple 
to construct a toric log CY space $X_0(B,\P,s)$ 
(if one also adds a suitable log structure).
Recall that a toric log CY space is \emph{positive} if the section of the
log smooth structure moduli bundle on $\X_0\backslash Z$ extends to $\X_0$ by 
attaining zeros rather than poles [\cite{grosie1}, Def. 4.19]. 
An analogous notion of \emph{positivity} for $(B,\P)$ is a 
condition on the local monodromy around the singular locus [\cite{grosie1}, Def.~1.54].
Recall that the set of polytopes $\P$ can be 
considered as a category consisting of lattice polytopes as objects 
and inclusions of faces as morphisms. We recall additional notation:
\begin{itemize}
\item $\P^{[l]}$ for the subset of cells of dimension $l$,
\item $\Delta$ for the \emph{discriminant locus} of $B$,
\item $\Lambda,\check\Lambda$ for the local system of integral tangent and 
cotangent vectors on $B\backslash \Delta$,
\item $i:B\backslash \Delta\ra B$ for the natural inclusion,
\item $\Lambda_\tau$ for the subset of 
tangent vectors parallel to $\tau\in\P$ at a relative interior point of $\tau$ and
\item $\check\Lambda_\tau$ for $\Lambda^\perp_\tau$.
\end{itemize}

Recall that a pair $(\omega,\rho)\in \P^{[1]}\times \P^{[\dim B-1]}$ determines a loop
around the singular locus by going from one vertex of $\omega$ through the interior of one
neighbouring maximal cell of $\rho$ to the other vertex of $\omega$ and returning to the first vertex by passing
through the interior of the other maximal neighbouring cell of $\rho$. 
The order of vertices and maximal cells and thus the orientation of the 
loop can be chosen by fixing integral primitive vectors $d_\omega\in\Lambda_\omega$
and $d_\rho\in\check\Lambda_\rho$. It was shown in loc.cit. that the monodromy in a nearby stalk of $\Lambda$ 
along the so determined homotopy
class of loop has a special shape and can be given by $n\mapsto n+\kappa_{\omega\rho}\langle n, \check d_\rho\rangle d_\omega$
where $\kappa_{\omega\rho}$ is an integer independent of the choices of $d_\omega$ and $\check d_\rho$
(\cite{grosie3}, before Def.~1.4).

Recall that a special fibre $\X_0$ of a toric degeneration is always positive.
We will assume from now on that $X$ is a positive toric log CY space.
The dual intersection complex $(B,\P)$ is then also positive, i.e., 
each $\kappa_{\omega\rho}\ge 0$.

Recall that the \emph{inner monodromy polytope} for $\rho\in\P^{[\dim B-1]}$ is constructed by
fixing a vertex $v\in \rho$ and by taking the convex hull of all $m^\rho_{v,v'}$ where $v'$ is a vertex of $\rho$
and $n\mapsto n+\langle n, \check d_\rho\rangle m^\rho_{v,v'}$ is the monodromy transformation 
of a stalk of $\Lambda$ near $v$ for a loop going from $v$ to $v'$ through the interior of the maximal cell 
on which $\check d_\rho$ is negative and returning through the other one. It is denoted by
$$\widetilde{\Delta}_\rho\subset \Lambda_\rho\otimes_\ZZ \RR.$$
By restricting to vertices in a face $\tau$ of $\rho$, one gets for each $e:\tau\ra\rho$ a polytope
$\widetilde{\Delta}_{\rho,e} \subset \Lambda_\tau\otimes_\ZZ \RR$
which is a face of the previous one. It is clear that the $m^\rho_{v,v'}$ are sums of appropriate 
$(\kappa_{\omega\rho}d_\omega)$'s. Up to integral translation, the monodromy polytopes  are independent of $v$.
Dually, we have the \emph{outer monodromy polytopes}
$$\widetilde{\check\Delta}_\omega\subset \check\Lambda_\omega\otimes_\ZZ \RR\hbox{ and }
\widetilde{\check\Delta}_{\omega,e}\subset \check\Lambda_\tau\otimes_\ZZ \RR$$
given $\omega\in\P^{[1]}$, resp. $e:\omega\ra\tau$. These are constructed from the monodromy of a stalk of
$\check\Lambda$ in some maximal cell $\sigma$ containing $\omega$ along loops passing through the vertices of $\omega$
into other maximal cells $\sigma'$. The transformations have the shape
$m\mapsto m+\langle d_\omega, m\rangle n^{\sigma,\sigma'}_\omega$.
We have decorated the polytopes by \ $\widetilde{ }$\ \ in contrast to \cite{grosie1}
to distinguish them from similar polytopes coming up later on.

Recall that there is a contravariant correspondence of closed strata $X_\tau$ of $X$
and cells $\tau\in\P$. The irreducible components of $X$
are $X_v$ for $v\in\P^{[0]}$.
Because each stratum is a toric variety, 
we also get a decomposition of $X$ in a disjoint
union of locally closed strata
$$X=\coprod_{\tau\in\P} \Int(X_\tau)$$
where $\Int(X_\tau)$ is supposed to be the open torus in $X_\tau$.
For each $\omega\in\P^{[1]}$ there is a possibly empty 
or non-reduced Cartier divisor $\widetilde{Z}_\omega$ in $X_\omega$ such that 
$$Z=\bigcup_{\omega\in\P^{[1]}} \widetilde{Z}_\omega$$ 
is the log singular locus of $X$. 
We have $\widetilde{Z}_\omega=\emptyset$ if and only if $\omega$ doesn't meet $\Delta$.  
For a semi-ample Cartier divisor 
(i.e. one whose invertible sheaf is generated by global sections)
$E$ on a toric variety we denote its Newton polytope 
defined via a linearly equivalent toric divisor by $\Newton(E)$.
For subvarieties $E$ of codimension greater than one, we set
$\Newton(E)=\{0\}$.
Recall from \cite{grosie1} that we have
$$\Newton(\widetilde{Z}_\omega) = \widetilde{\check\Delta}_\omega.$$
We write $\widetilde{Z}_\omega^\red$ for the reduction of 
the effective Cartier Divisor $\widetilde{Z}_\omega$ and 
set $Z_\omega:=\widetilde{Z}_\omega^\red$.
We follow \cite{batyrev1}, \cite{danikhov} and call a
semi-ample divisor $E$ on a toric variety
\emph{non-degenerate} if $\Newton(E)$ up to translation coincides with the convex hull of
all monomials with nontrivial coefficients given an equation of $E$ in a toric chart 
and $E$ has a regular or empty intersection with every torus orbit.

\begin{definition} \label{def_ht} A positive toric log CY space is of \emph{hypersurface type} (short: h.t.) iff
 \begin{enumerate}
 \item The divisor $Z_{\omega}$ is non-degenerate for each $\omega\in\P^{[1]}$
 and for some $a_\omega\in\NN_{\ge1}$
 $$\widetilde{Z}_\omega=a_\omega\cdot Z_{\omega}.$$   
 \item For each $\tau\in\P$, the set 
  $\{Z_{\omega}\cap X_\tau\,|\,\omega\in\P^{[1]}, Z_{\omega}\cap\Int(X_\tau)\neq\emptyset\}$ 
  is either empty or contains only one element which we then denote by $Z_\tau$.
 \end{enumerate}
\end{definition}

The nomenclature is deduced from Batyrev's mirror construction \cite{batyrev2}.
A toric degeneration which is fibrewise embedded as an anticanonical 
hypersurfaces in a Fano toric variety in generic position yields an example of 
a h.t. space. Generally, having an embedding is not necessary of course. 
We will mostly concentrate on the h.t. property in this paper. For
the more general parts, we use the analogue of the 
Batyrev-Borisov mirror construction \cite{borisov} as in the upcoming definition.
We call a set of lattice polytopes
$\Delta_1,...,\Delta_r$ in an $\RR$-vector space $W$ \emph{transverse} if their tangent
spaces form an interior direct sum in $W$. 

\vbox{
\begin{definition} \label{defcit}
 A positive toric log CY space is of \emph{complete intersection type} (short: c.i.t.) iff
 \begin{enumerate}
 \item The divisor $Z_{\omega}$ is non-degenerate for each $\omega\in\P^{[1]}$
 and for some $a_\omega\in\NN_{\ge1}$
 $$\widetilde{Z}_\omega=a_\omega\cdot Z_{\omega}$$ 
 \item For each $\tau$ and $\omega_1,\omega_2\in\P^{[1]}$, we have
 $$\{0\}\neq\Newton(Z_{\omega_1}\cap X_\tau)
 =\Newton(Z_{\omega_2}\cap X_\tau)\ 
 \Rightarrow\ Z_{\omega_1}\cap X_\tau=Z_{\omega_2}\cap X_\tau$$
 \item For each $\tau\in\P$, the set 
 $\{\Newton(Z_{\omega}\cap X_\tau)\,|\,\omega\in\P^{[1]}, Z_{\omega}\cap\Int(X_\tau)\neq\emptyset\}$ 
 is either empty or contains at most $\min(\dim\tau,\codim\tau)$ many elements 
 $\check\Delta_{\tau,1}$,...,$\check\Delta_{\tau,q}$ which are transverse. The corresponding divisors are denoted
 by $Z_{\tau,1},...,Z_{\tau,q}$.
 \end{enumerate}
\end{definition}
}

\subsection{A spectral sequence to compute the log Hodge groups}
In this and in the next section, we are going to summarize the main results of the paper.
We recall some notions of \cite{grosie2} in the following, 
in particular the barycentric resolution of the log Hodge sheaves.
Let $X$ a toric log CY space and $j:X\backslash Z\ra X$ denote the canonical 
inclusion of the log smooth locus. 
\begin{definition} 
The \emph{log Hodge sheaf} $\Omega^r$ of degree $r$ is the 
pushforward of the sheaf of log differential forms, i.e.,
$$\Omega^r:=j_*\Omega^r_{(X\backslash Z)^\log/k^\log}.$$
The \emph{log Hodge group} of index $p,q$ is the cohomology group 
$$H^{p,q}_\llog(X):=H^q(X,\Omega^p).$$
The \emph{log Hodge number} of index $p,q$ is $h^{p,q}_\llog(X) := \dim H^{p,q}_\llog(X)$.
\end{definition}

Where useful, we will write $\Omega^p_X$ for $\Omega^p$. Recall from \cite{grosie2} that $F_s(\tau_0\ra\tau_k): X_{\tau_k}\ra X_{\tau_0}$ is the inclusion
of one stratum of $X$ in another indexed by $\tau_0,\tau_k\in\P$. It is written this way to account for the possibly
non-trivial glueing data $s$. We drop the base scheme $S$ in 
the usual notation $F_{S,s}$ because we always assume $S=\Spec \k$.
Recall that $F_s(e)= F(e) \circ s_e$ where $F(e)$ is the standard toric inclusion 
and $s_e\in \Aut(X_{\tau_2})$ is given by the action of a torus element.
We set
$$\Omega^r_\tau := (\kappa_{\tau})_*\kappa_{\tau}^*(q^*_{\tau}\Omega^r/\Tors)$$
where $q_\tau: X_\tau\ra X$ and 
$\kappa_\tau: X_\tau\backslash (D_\tau\cap q_\tau^{-1}(Z))\ra X_\tau$
are natural inclusions with $D_\tau=X_\tau\backslash \Int(X_\tau)$. 

\begin{definition}\label{defCOmega} 
We recall the barycentric complex given by
$$\SC^k(\Omega^r)=\bigoplus_{\tau_0\ra\dots\ra\tau_k}
(q_{\tau_k})_*((F_s(\tau_0\ra\tau_k)^*\Omega_{\tau_0}^r)/\Tors)$$
where $\Tors$ is the torsion submodule. The differential is
$$\begin{array}{rl}
(d_\bct(\alpha))_{\tau_0\ra\dots\ra\tau_{k+1}}=
&\alpha_{\tau_1\ra\dots\ra\tau_{k+1}}
+\sum_{i=1}^k(-1)^i\alpha_{\tau_0\ra\dots\breve{\tau}_i\ra\dots\ra\tau_{k+1}}\\
&+(-1)^{k+1} F_s(\tau_k\ra\tau_{k+1})^*\alpha_{\tau_0\ra\dots\ra\tau_{k}}.
\end{array}$$
\end{definition}

The proof of the following lemma will be given in Section~\ref{sec_barycentric}.

\begin{lemma} \label{COmegaexact} 
If $X$ is c.i.t., there is an exact sequence
$$0\ra\Omega^r \ra\SC^0(\Omega^r)\stackrel{d_\bct}{\lra} \SC^1(\Omega^r)\ra ...\ .$$
\end{lemma}

Each morphism $e:\tau_1\ra \tau_2$ in $\P$ can be identified with an edge in $B$
and we define the open set $W_e\subset B$ 
as the union of all relative interiors of simplices
in the barycentric subdivision of $\P$ which contain $e$.
We set $Z_e=Z_{\tau_1}\cap X_{\tau_2}$ and then have
$R(Z_e)$ as before defined with respect to the toric variety $X_{\tau_2}$.
We prove part a) of the following result in Section~\ref{Coho_singlestrat} and part b) in Section~\ref{Sec_Proofmainb}.

\begin{theorem} \label{maintheorem} Let $X$ be a h.t. toric log CY space. We fix $r$. 
\begin{itemize}
\item[a)]
  The $E_1$ term of the hypercohomology spectral sequence of $\SC^\bullet(\Omega^r)$ is
  $$E_1^{p,q}:H^q(X,\SC^p(\Omega^r)) \Rightarrow
  \HH^{p+q}(X,\SC^\bullet(\Omega^r))=H^{p+q}(X,\Omega^r),$$
  where
  $$H^q(X,\SC^p(\Omega^r))=
    \bigoplus_{e:\tau_0\ra\dots\ra\tau_p}
    \left\{ 
   \begin{array}{rl}
    \Gamma(W_e,i_*\bigwedge^r\check\Lambda\otimes_\ZZ\k)&\hbox{ for }q=0,\\[8pt]  
    \displaystyle R(Z_e)_q\otimes 
    \frac{\Gamma(W_e,i_*\bigwedge^{r+q}\check\Lambda\otimes\k)}
    {\Gamma(W_e,\bigwedge^{r+q}i_*\check\Lambda\otimes\k)}
    &\hbox{ for }q>0
   \end{array}
  \right\}.$$        
  Note that Lemma~\ref{thmstratafunct} gives the differential $d_1$.
\item[b)] If every $\check\Delta_\omega$ is a simplex, 
the spectral sequence in a) is degenerate at $E_2$ level and
$$E^{p,0}_2=H^p(B,i_*\bigwedge^{r}\check\Lambda\otimes\k).$$
\end{itemize}
\end{theorem}

\subsection{Base change of the logarithmic Hodge groups}
In Section~\ref{subsec_localmod} we give for each point $x$ of a c.i.t. space $X$ 
a \emph{local model} for the log structure, i.e.,
an affine toric variety $Y_\loc$ with a toric Cartier divisor $X_\loc$, s.t. at $x$,
$X$ is \'etale locally equivalent to an open subset of $X_\loc$, and
the log structure on $X$ agrees with the pullback to $X_\loc$ of the
divisorial log structure on $Y_\loc$ given by the divisor $X_\loc$.
This is important for points in $Z$, the others fulfil this by definition.

Analogous to (\cite{grosie2}, Def.~2.7), we say that a toric deformation 
$\X\ra\shS$ where $X=\X_0$ is a c.i.t. space 
is a \emph{divisorial deformation} of $X$ if it is \'etale locally isomorphic to 
the c.i.t. local models $Y_\loc$. We are then going to prove:

\begin{theorem} \label{hyperbasechange}
 Let $\pi:\X\ra\Spec A$ be a divisorial deformation of a c.i.t. toric log CY space, 
 $j:\X\backslash \mathcal{Z}\ra\X$ the inclusion of the log smooth locus
 and write $\Omega_{\X}^\bullet := j_*\Omega^\bullet_{\X^\log/A^\log}$. 
 Then for each $p$, $\HH^p(\X,\Omega^\bullet_\X)$ is a free $A$-module, 
 and it commutes with base change.
\end{theorem}

\begin{corollary} \label{basechange}
Let $\pi:\X\ra\Spec A$ be a divisorial deformation of a c.i.t. toric log CY space $X$. 
If the log Hodge to log de Rham spectral sequence on $X$ (i.e., the hypercohomology spectral
sequence of $\Omega^\bullet_X$) degenerates at $E_1$ then
$H^q(\X,\Omega^p_\X)$ is a free $A$-module, and it commutes with base change.
\end{corollary}

\begin{proof} 
By Grothendieck's cohomology and base change theorem, it 
suffices to prove surjectivity for the restrictions
$H^q(\X,\Omega^p_\X)\ra H^q(X,\Omega^p_X)$. This means surjectivity for
$E_1(\Omega^\bullet_\X)\ra E_1(\Omega^\bullet_X)$. 
Since degeneration is an open property, both spectral sequences 
are degenerate at $E_1$ and we are done if 
we show surjectivity of
$(\op{Gr}_F\HH^{k}(\X,\Omega^\bullet_\X))/\op{Tors}\ra \op{Gr}_F\HH^{k}(X,\Omega^\bullet_X)$
where $F$ is the canonical filtration.
This follows from Thm.~\ref{hyperbasechange} by the surjectivity of 
$\HH^{k}(\X,\Omega^\bullet_\X)\ra \HH^{k}(X,\Omega^\bullet_X)$.
\end{proof}

\begin{remark} \label{remorbifold}
If all inner monodromy polytopes are simplices then the generic fibre $X_\eta$ is an orbifold.
The restriction of $\Omega^r_\X\otimes_{\shO_{\Spec A}} \shO_{\Spec \eta}$ coincides with
the pushforward of $\Omega^r_{(X_\eta\backslash\Sing X_\eta)/\k}$ to $X_\eta$.
By \cite{steenbrink2}, these sheaves give the natural mixed Hodge structure on $X_\eta$ (\cite{deligne3})
which is pure in each cohomology degree.
\end{remark}

\begin{definition} The \emph{affine Hodge group} of degree $(p,q)$ 
of a toric log CY space $X$, resp. its dual intersection complex $(B,\P)$,
is defined as
$$H^{p,q}_\aff(X)=H^{p,q}_\aff(B)=H^q(B,i_*\bigwedge^{p}\check\Lambda\otimes\k).$$
We denote its dimension by $h^{p,q}_\aff(X)$ and call it \emph{affine Hodge number}.
\end{definition}

We are going to prove the following result in Section~\ref{sec_affloghodge}.

\begin{theorem} \label{Thm_affinlog}
Let X be a c.i.t. toric log CY space.
\begin{itemize}
\item[a)] For each $p,q$ there is a natural injection 
$$H^{p,q}_\aff(X) \hra H^{p,q}_\llog(X).$$
\item[b)] For each $k$ there is a natural injection 
$$\bigoplus_{p+q=k} H^{p,q}_\aff(X) \hra \HH^{k}(X,\Omega^\bullet)$$
which is compatible with the canonical filtration induced 
on $\HH^{k}(X,\Omega^\bullet)$.
\end{itemize}
\end{theorem}

\begin{corollary} Let $X_t$ be a general fibre of a toric degeneration with at most
orbifold singularities. 
Assume that the central fibre $X$ is a c.i.t. space. For all $p,q$, we have
$$h^{p,q}_\aff(X) \le h^{p,q}(X_t).$$
\end{corollary}

\begin{proof} By Thm.~\ref{hyperbasechange}, Thm.~\ref{Thm_affinlog} and
Remark~\ref{remorbifold}, we have
$h^{p,q}_\aff(X) \le \dim \op{Gr}_F^p\HH^{p+q}(X,\Omega_X^\bullet) 
\le \rk \op{Gr}_F^p\HH^{p+q}(\X,\Omega_\X^\bullet)/\op{Tors} = h^{p,q}(X_t)$
where $F$ is the canonical filtration on $\Omega_\X^\bullet$. 
\end{proof} 

In Section~\ref{Sec_twistsec}, we give a proof of the following result.

\begin{theorem} \label{Thm_logdeRham}
Let $X$ be a h.t. toric log CY space.
Assume we have one of the following conditions
\begin{enumerate}
\item[a)] $\dim X\le 2$
\item[b)] $\dim X=3$, each $\check\Delta_\tau$ is a simplex and
every component of $\Delta \backslash \Delta^0$ is contractible
where $\Delta^0$ denotes the set of
points in $\Delta$ where the corresponding monodromy polytope
$\check\Delta_\tau$ has dimension two
\item[c)] $\dim X \le 4$ and each $\check\Delta_\tau$ is an elementary simplex
\end{enumerate}
Then the log Hodge to log de Rham spectral sequence
 degenerates at 
   $$E^{p,q}_1:H^p(X,\Omega^q)\Rightarrow \HH^{p+q}(X,\Omega^\bullet).$$
\end{theorem}

\begin{remark}
To prove the degeneration of the log Hodge to log de Rham spectral sequence
in greater generality, a common way would be show that $\Omega_X^\bullet$
carries the structure of a cohomological mixed Hodge complex (\cite{deligne3},~8.1.9).
In particular, this requires a $\ZZ$-structure which one would obtain as the
pushforward from the semi-analytic Kato-Nakayama space $\tilde X\ra X$. One then needs
to show that $\Omega_X^\bullet$ is quasi-isomorphic to a pushforward
of a modified de Rham complex on $\tilde X$ which is in turn a 
resolution of $\ZZ\otimes_\ZZ\CC$ on $\tilde X$. We leave the topological properties 
of the local models to future work.
\end{remark}

\begin{theorem} \label{mirror_affinestringy}
Assume that we are given a h.t. space $X$ and that $X_t$ is
a general fibre of a degeneration into $X$. 
Assume we are in one of the cases of
Thm.~\ref{Thm_logdeRham} and that $X_t$ is an orbifold, 
i.e., each $\Delta_\tau$ is a simplex. We have for each $p,q$,
\begin{enumerate}
\item[a)] $h^{p,q}_\llog(X)=h^{p,q}(X_t)$
\item[b)] If we are in case a) or c), we have
$$h^{p,q}(X_t) - h^{p,q}_\aff (X)\ =\ h^{n-p,q}_\st (\check X_t) - h^{n-p,q}(\check X_t).$$
\end{enumerate}
\end{theorem}

\begin{example} 
Note that Theorem~\ref{mirror_affinestringy}, a)
holds for all Calabi-Yau threefolds
obtained from simplicial subdivisions of reflexive 4-polytopes where
the subdivision doesn't introduce new vertices. In particular, we
obtain for the quintic threefold $X$ in $\PP^3$ as well as for its mirror 
dual orbifold the affine Hodge diamond
$$
\begin{array}{cccccccccccc}
&&&1\\
&&0&&0\\
&0&&1&&0\\
1&&1&&1&&\ 1.\\
\end{array}
$$
The log twisted sectors of $X$ contribute to $h^{2,1}(X)=h^{1,2}(X)=101$ 
by adding $100$ to the affine Hodge numbers and since $X$ is smooth, $h^{p,q}(X)=h_\st^{p,q}(X)$. 
All log twisted sectors of $\check X$ are trivial. 
On the other hand, we obtain non-trivial orbifold twisted sectors in degree $(1,1)$
and $(2,2)$. 
We have $h_\st^{1,1}(\check X)=h^{1,1}(\check X)+100=h_\aff^{1,1}(\check X)+100=101$ 
and the analogous for $h_\st^{2,2}(\check X)$.
\end{example}


\section{Local models for c.i.t. spaces}
\subsection{Reduced inner monodromy polytopes}
The c.i.t. property is a generalization of h.t. 
becoming distinct only if $\dim X\ge 4$. 
It also generalizes simplicity (\cite{grosie1}, Def.~1.60, Rem.~1.61) 
which we referred to in the introduction as a \emph{maximal degeneration}. 
There is a natural bijection 
$$\P \leftrightarrow \{\hbox{ vertices of the barycentric subdivision of }\P\}$$
by identifying a cell with its barycenter. 
Moreover, there is a natural bijection between the set of $d$-dimensional
simplices in the barycentric subdivision and the set of chains of proper inclusions
$\tau_0\ra...\ra\tau_d$ of cells in $\P$.
It follows from (\cite{grosie1}, Def.~1.58) that
the discriminant locus $\Delta$ is the union of all codimension 
two simplices in the barycentric subdivision of $\P$ corresponding 
to chains of the shape
$\omega\ra...\ra\rho$ with $\omega\in\P^{[1]}$, $\rho\in\P^{[\dim B-1]}$ 
and $\kappa_{\omega\rho}\neq 0$.

\begin{lemma} \label{lem_ZDelta}
 Let $X$ be a c.i.t. toric log CY space. Fix $\omega\in\P^{[1]}$.  
 The barycentric edge corresponding to some
 $e:\omega\ra\tau$ is contained in $\Delta$ if and only if
 $$Z_\omega\cap \Int(X_\tau)\neq\emptyset.$$
\end{lemma}

\begin{proof} 
 We just sketch the proof to keep the notation concise.
 The Newton polytope of the closure of $Z_\omega\cap \Int(X_\tau)$ is 
 a face of $\check\Delta_\omega$ contained in a translate 
 of $\tau^\perp$. The intersection is non-trivial if and only if 
 this Newton polytope has positive dimension. This happens if and only if
 it contains an edge of $\check\Delta_\omega$ which in turn corresponds
 to some $\tau\ra\rho$ such that this edge is parallel to $\rho^\perp$.
 This means $\kappa_{\omega\rho}\neq 0$. This happens for some $e:\tau\ra\rho$
 if and only if $e$ is contained in $\Delta$.
\end{proof}

\begin{lemma} \label{kappas} 
Let $X$ be a c.i.t. toric log CY space.
\begin{itemize}
\item[a)] We have
$$\frac{\kappa_{\omega_1\rho}}{a_{\omega_1}}=\frac{\kappa_{\omega_2\rho}}{a_{\omega_2}},$$
whenever the barycentric edges $\omega_1\ra\rho$, $\omega_2\ra\rho$
are contained in $\Delta$.
\item[b)] We define $\check{a}_\rho$ as the integral length of $\Newton((Z\cap X_\rho)^\red)$ and have for each 
$(\omega,\rho)\in \P^{[1]}\times \P^{[\dim B-1]}$
$$\kappa_{\omega\rho}=a_\omega \check{a}_\rho.$$
\end{itemize}
\end{lemma}

\begin{proof}
We prove a). Because $\codim\rho=1$ there is at most one $\check\Delta_{\rho,i}$ by
Def.~\ref{defcit}. Assume we have $\P^{[1]}\ni\omega_j\stackrel{e_j}{\lra}\rho$ for $j=1,2$ 
 such that $Z_{\omega_j}\cap X_\rho\neq\emptyset$. By 
 Lemma~\ref{lem_ZDelta}, this is equivalent to $e_1,e_2\in\Delta$.
 We get $\Newton(Z_{\omega_j}\cap X_\rho)=\check\Delta_{\rho}$.
 Because $\frac{\kappa_{\omega_j,\rho}}{a_{\omega_j}}$ 
 is the integral length of $\check\Delta_{\rho,1}$, we get the assertion.
 Part b) is just rephrasing this. 
\end{proof}

\begin{remark} \label{rem_ht}
 \begin{enumerate}
 \item A positive toric log CY space in dimension 2 where $Z$ is reduced is h.t..
 Not included in the h.t. definition are situations where 
 some $Z_\omega$ is the union of a 
 reduced point and a double point, for instance. 
 Two double points, however, would fulfill h.t. by having $a_\omega=2$.
 \item If $X$ is simple then $X$ is h.t. iff for each $\tau\in\P$ the 
 number of outer (or inner) monodromy polytopes 
 at $\tau$ given in the simplicity definition is less or equal than one.
 The inverse direction follows from the multiplicative condition for the log structure 
 $\prod_\omega d_\omega \otimes f_\omega^{\eps_\tau(\omega)}|_{V_\tau}=1$
 (\cite{grosie1}, Thm 3.22)
 which implies that all $Z_\omega|_{X_\tau}$ for varying $\omega$ are either empty or agree because
 $f_\omega$ is a local equation of $Z_\omega$.
 In particular, if $X$ is simple of dimension $3$ then $X$ is h.t..
 \item Why do we allow $a_\omega>1$? 
 This is best seen by the just mentioned
 multiplicative condition for the log structure. If some inner simplex
 polytope has non-primitive edges of different integral lengths,
 we have to require some $a_\omega>1$ for a log structure to exist 
 on such a space.
 \item Recall that a discrete Legendre transform interchanges inner and outer
 monodromy polytopes. It also interchanges $a_\omega$ and $\check a_\rho$ and we will
 see that there is a collection of reduced inner monodromy polytopes analogous
 to the collection of reduced outer monodromy polytopes in the definition of c.i.t.
 \end{enumerate}
\end{remark}

Here is a lemma which relates the inner and 
outer monodromy polytopes to the $\kappa_{\omega\rho}$.
It is directly deduced from the construction of $\widetilde{\Delta}_\rho$ and 
$\widetilde{\check\Delta}_\omega$.

\begin{lemma} \label{inoutkappa}
 \begin{enumerate}
 \item Given $\rho\in\P^{[\dim B-1]}$, there is a natural surjection
 $$\{\omega\ra\rho\, |\, \omega\in\P^{[1]}, \kappa_{\omega\rho}\neq 0 \}
 \ \rightarrow\ \{\hbox{edges of }\widetilde{\Delta}_{\rho}\}.$$
 Moreover, $\omega$ is collinear to the edge it maps to
 and $\kappa_{\omega\rho}$ is its integral length.
 \item Given $\omega\in\P^{[1]}$, there is a natural surjection
 $$\{\omega\ra\rho\, |\, \rho\in\P^{[\dim B-1]}, \kappa_{\omega\rho}\neq 0 \}
 \ \rightarrow\ \{\hbox{edges of }\widetilde{\check\Delta}_{\omega}\}.$$
 Moreover, a translate of $\rho^\perp$ contains the edge 
 it maps to and $\kappa_{\omega\rho}$ is its integral length.
 \end{enumerate}
\end{lemma}

\begin{lemma} \label{polyunique}
 For $N=\ZZ^n$ and $M=\Hom(N,\ZZ)$, 
 let $\Sigma$ be a complete fan in $N_\RR=N\otimes_\ZZ\RR$ and $\psi$ a piecewise linear function on $N_\RR$
 with respect to $\Sigma$. Assume that $\psi$ comes from a lattice polytope $\Xi\subset M_\RR$, i.e.,
 $$\psi(n)=-\min\{\langle m,n\rangle\,|\, m\in\Xi\}.$$
 Given ${\check\omega}\in\Sigma^{[n-1]}$, let $\sigma_1,\sigma_2\in\Sigma^{[n]}$ 
 the two maximal cones containing $\check\omega$.
 We set \newline $\kappa_{\check\omega}=\hbox{integral length of }m_1-m_2$
 where $m_1,m_2\in M$ with $m_1=\psi|_{\sigma_1}$ and $m_2=\psi|_{\sigma_2}$. The data
 $$k:\Sigma^{[n-1]}\ra \NN,\ k({\check\omega})=\kappa_{\check\omega}$$
 determines $\Xi$ uniquely up to translation.
\end{lemma}

\begin{proof}
 Note that $m_1-m_2$ is collinear to $\check\omega^\perp$ and is thus uniquely 
 determined by $\kappa_{\check\omega}$ up to orientation. 
 The combinatorics of the fan now gives a recipe
 to assemble these edge vectors to the polytope $\Xi$. Fix some maximal cone 
 ${\check v}_0\in\Sigma^{[n]}$. To each chain $\gamma$ of the shape
 $${\check v}_0 \supset {\check\omega}_0\subset {\check v}_1 
 \supset {\check\omega}_1\subset\ ...\ \supset{\check\omega}_l\subset {\check v}_l$$
 with ${\check\omega}_i\in\Sigma^{[n-1]}, {\check v}_i\in\Sigma^{[n]}$, set
 $m_\gamma=\sum_{i=0}^l m_i$
 where $m_i$ is the unique element in $M$ which is collinear to 
 ${\check\omega}_i^\perp$, has integral length 
 $\kappa_{{\check\omega}_i}$ and evaluates positive on the interior of ${\check v}_{i}$.
 We obtain
 $$\Xi = \hbox{convex hull of }\{m_\gamma\,|\,\gamma \hbox{ is a chain}\}.$$
\end{proof}

The following proposition takes care of the inner monodromy polytopes
which are not obvious from the definition of c.i.t. unlike the outer ones.

\begin{proposition} \label{innerpolys}
 Let $X$ be c.i.t. and $\tau\in\P$. 
 Let $\check\Delta_{\tau,1},...,\check\Delta_{\tau,q}$ be the associated set of Newton polytopes.
 There exists a canonical set of lattice polytopes 
 $$\Delta_{\tau,1},...,\Delta_{\tau,q}\subset \Lambda_\tau\otimes_\ZZ \RR$$
 such that, for each $\rho\in\P^{[\dim B-1]}, e:\tau\ra\rho$ and 
 $\widetilde{\Delta}_{\rho,e}$ non-trivial,
 we find a unique $i$ such that $\widetilde{\Delta}_{\rho,e}$ 
 is an integral multiple of $\Delta_{\tau,i}$.
\end{proposition}

\begin{proof}
\underline{The correspondence:}\ 
 We have fixed $\tau$. All $\omega$'s are supposed to be in $\P^{[1]}$ and
 all $\rho$'s in $\P^{[\dim B-1]}$. Consider the diagram
 $$
 \xymatrix{ 
   \{\omega\ra\tau\ra\rho\,|\, \kappa_{\omega\rho}\neq 0\} \ar[r]\ar[d]  
     & \{\tau\ra\rho\,|\,\widetilde{\Delta}_{\rho,\tau\ra\rho}\neq 0\} \ar@{.>}[d]\\  
   \{\omega\ra\tau\,|\, Z_\omega\cap\Int(X_\tau)\neq\emptyset\}\ar[r]
   & \{1,...,q\}.
 } $$
 The upper horizontal arrow is just ``forgetting $\omega$'', the left vertical one is
 ``forgetting $\rho$'' and uses Lemma~\ref{lem_ZDelta}. 
 The lower horizontal map is given by part (3) of the definition of c.i.t..
 There is only one way to define the dotted arrow to make the diagram commute and we
 need to argue why it is well-defined. Assume we have $\omega_1\ra\tau\ra\rho$ and
 $\omega_2\ra\tau\ra\rho$ with $\kappa_{\omega_1\rho}\neq 0\neq \kappa_{\omega_2\rho}$.
 By Lemma~\ref{inoutkappa} we find that $\widetilde{\check\Delta}_{\omega_1}$ 
 and $\widetilde{\check\Delta}_{\omega_2}$
 both have an edge contained in a translate of the straight line $\rho^\perp$.
 The same holds for $\widetilde{\check\Delta}_{\omega_1,{\omega_1\ra\tau}}, 
 \widetilde{\check\Delta}_{\omega_2,{\omega_2\ra\tau}}$ and also for 
 $\frac1{a_{\omega_1}}\widetilde{\check\Delta}_{\omega_1,{\omega_1\ra\tau}}, 
 \frac1{a_{\omega_2}}\widetilde{\check\Delta}_{\omega_2,{\omega_2\ra\tau}}$
 which are the Newton polytopes of 
 $Z_{\omega_1}\cap X_\tau$ and $Z_{\omega_2}\cap X_\tau$, respectively. 
 Thus, these polytopes cannot be transverse and by (3) of the c.i.t. definition 
 they have to be the same up to translation. This makes
 the dotted map well-defined. 

 We denote the preimage of $i$ under the lower horizontal map by $\Upomega_{\tau,i}$.

\underline{Defining the $\Delta_{\tau,i}$:}\ 
 We stay with the previous setup. We define 
 $$\Delta_{\tau,i}:=\frac{1}{\check a_\rho}\widetilde{\Delta}_{\rho,\tau\ra\rho}$$
 where $i$ is the image of $\widetilde{\Delta}_{\rho,\tau\ra\rho}$ under the dotted arrow. It is easy to see that
 (up to translation) this is a lattice polytope where an edge which is 
 the image of some $\omega$ via Lemma~\ref{inoutkappa} has length $a_\omega$.
 We have to show that we get the same $\Delta_{\tau,i}$ if we choose 
 another $\tau\ra\rho'$ with $\kappa_{\omega\rho'}\neq 0$ to define it. 
 We are going to apply Lemma~\ref{polyunique}. 
 By (\cite {grosie1}, Remark~1.59), both 
 $\frac{1}{\check a_\rho}\widetilde{\Delta}_{\rho,\tau\ra\rho}$ 
 and $\frac{1}{\check a_{\rho'}}\widetilde{\Delta}_{\rho',\tau\ra\rho'}$
 give piecewise linear functions on $\check\Sigma_\tau$,
 the normal fan of $\tau$ in $\Hom(\Lambda_\tau,\ZZ)\otimes_\ZZ\RR$. 
 We have an inclusion reversing bijection
 $$\hbox{cones in }\check\Sigma_\tau\ \leftrightarrow\ \hbox{faces of }\tau.$$
 Codimension one cones $\check\omega\in\check\Sigma^{[\dim\tau-1]}_\tau$ 
 correspond to edges $\omega$ of $\tau$.
 So this is consistent with the notation in the lemma. 
 Note that the data $k:\check\Sigma^{[\dim\tau-1]}_\tau\ra\NN$ is the same for both polytopes
 because for each $\check\omega\in\check\Sigma^{[\dim\tau-1]}$ we have 
 $$\kappa_{\check\omega}=\left\{\begin{array}{ll} a_\omega & \hbox{if }\omega\ra\tau\in
 \Upomega_{\tau,i}\\ 0 & \hbox{otherwise}.\end{array}\right.$$
 We deduce that $\frac{1}{\check a_{\rho'}}\widetilde{\Delta}_{\rho,\tau\ra\rho}$ and 
 $frac{1}{\check a_{\rho'}}\Delta_{\rho',\tau\ra\rho'}$
 coincide up to translation.
\end{proof}

We extract a definition from the previous proof.

\begin{definition} 
  Given a c.i.t. $X$ and $\tau\in\P$. For $1\le i\le q$, we define
  $$\Upomega_{\tau,i}=\{\omega\ra\tau\, | \,\widetilde{\check\Delta}_{\omega,\omega\ra\tau}
  =a\cdot\check\Delta_{\tau,i}\hbox{ for some }a>0\}$$
  $$R_{\tau,i}=\{\tau\ra\rho\, | \,\widetilde{\Delta}_{\rho,\tau\ra\rho}
  =a\cdot\Delta_{\tau,i}\hbox{ for some }a>0\}$$
  where $\Delta_{\tau,i}$ is the polytope given in Prop.~\ref{polyunique}.
\end{definition}

Note that for $\omega\ra\tau,\tau\ra\rho$ we have $\kappa_{\omega\rho}\neq 0$ if and only if 
there is some $i$ such that $\omega\ra\tau\in \Omega_{\tau,i}$ and $\tau\ra\rho\in R_{\tau,i}$
which can be deduced from the
diagram in the proof of Prop.~\ref{innerpolys}. 
In view of (\cite{grosie1},~Def.~1.60), we see that this property 
generalizes from \emph{simplicity} to c.i.t. spaces.


\subsection{Toric local models for the log structure} \label{subsec_localmod}
In this section, we give a direct generalization
of the local model construction developed by M.~Gross and B.~Siebert in \cite{grosie2}
to the c.i.t. case. 
The proof will remain sketchy where there is little difference to loc.cit.. 
Recall Construction~2.1 in loc.cit. where $Y$ is the product of a torus with
the affine toric variety given by the cone over the Cayley product of $\tau$ and 
the $\Delta_{\tau,i}$ and $X$ is the invariant divisor given by the rays in $\tau$.
We prefer to call the spaces $X,Y$ of loc.cit. $X_\loc, Y_\loc$ at this point.

\begin{proposition} \label{localmodels}
 Suppose we are given a c.i.t. toric log CY space $X$ and a geometric point in the
 log singular locus $\bar{x}\ra Z\subseteq X$, there exist data 
 $\tau,\check\psi_1,...,\check\psi_q$ as in (\cite{grosie2},Constr.~2.1) 
 defining a monoid $P$ and an element $\rho\in P$, 
 hence affine toric log spaces $Y_\loc^\log, X_\loc^\log\ra\Spec\k^\log$, such that there is a diagram over $\Spec \k^\log$
 $$
 \xymatrix{
 & V^\log\ar[dl]\ar^{\phi}[dr]\\
 X^\log && X_\loc^\log
 }$$
 with both maps strict \'etale and $V^\log$ an \'etale neighbourhood of $\bar{x}$.
\end{proposition}

\begin{proof}
 As in loc.cit., we take the unique $\tau\in\P$ such that $\bar{x}\in\Int(X_\tau)$.
 By the definition of c.i.t., we then have the outer monodromy polytopes 
 $\check\Delta_{\tau,1},...,\check\Delta_{\tau,q}\subset
 \check\Lambda_\tau$.
 By Prop.~\ref{innerpolys}, we also obtain
 $\Delta_{\tau,1},...,\Delta_{\tau,q}\subset\Lambda_\tau$.
 By renumbering, we may assume that $\bar{x}\in Z_{\tau,1},...,Z_{\tau,r}$ and 
 $\bar{x}\not\in Z_{\tau,i}$ for $r<i\le q$. We set
 $$\Delta_i=\left\{\begin{array}{ll} \Delta_{\tau,i} &\hbox{for }1\le i\le r\\ 
 \{0\} &\hbox{for }r< i\le \dim B-\dim\tau
 \end{array}\right.$$
 We redefine $q=\dim B-\dim\tau$.
 The polytopes $\Delta_i$ give piecewise linear functions $\check\psi_i$ on the normal fan 
 $\check\Sigma_\tau$ in $N'_\RR=\Hom(\Lambda_\tau,\ZZ)\otimes_\ZZ\RR$. 
 By (\cite{grosie2}, Constr.~2.1), we obtain a monoid $P'\subseteq N'$ with
 $P'=C(\tau)^\vee\cap N'$, a monoid $P\subseteq N=N'\oplus\ZZ^{q+1}$, 
 $\rho\in P$ given by $\rho=e_0^*$, $Y_\loc=\Spec\k[P]$ and 
 $X_\loc=\Spec\k[P]/(z^\rho)$. 
 To obtain a log-structure on $X_\loc$, we use the pullback of 
 the divisorial log structure given by $X_\loc$ in $Y_\loc$.
 To proceed as in the proof of (\cite{grosie2},~Thm.~2.6), we choose $g:\tau\ra\sigma\in\P^{[\dim B]}$
 to have an \'etale neighbourhood $V(\sigma)$ of $\bar{x}$.
 We are going to construct a diagram with strict \'etale arrows
 $$
 \xymatrix{
 & V(\sigma)^\log \ar[dl]_{p_\sigma} & \,V(\tau)^\log \ar@{_{(}->}[l] & \ar@{_{(}->}[l] \,V^\log \ar^{\phi}[dr]\\
 X^\log &&&& X_\loc^\log.
 }$$
 Recall from (\cite{grosie1}, Thm.~3.27) that pulling back the log-structure from $X^\log$ to $V(\sigma)$ gives a tuple
 $$f=(f_{\sigma,e})_{e:\omega\ra\sigma}\in
 \Gamma\bigg(V(\sigma),\bigoplus_{\atop{e:\omega\ra\sigma}{\dim\omega=1}}\shO_{V_e}\bigg)$$
 where $V_e$ is the closure of $\Int(X_\omega)$ in $V(\sigma)$.
 We write $f_\omega$ for $f_{\sigma,e}$ with $e:\omega\ra\tau$.
 Let $x[\sigma]$ be the unique zero-dimensional torus orbit in $V(\sigma)$. We may assume that $f$
 is normalized, i.e., $f_\omega|_{x[\sigma]}=1$ for each $\omega$. This is possible because, if
 $f$ is not normalized, we may use a pullback by an automorphism of $V(\sigma)$ as explained in 
 [\cite{grosie1}, after Def.~4.23] to obtain a normalized section. 
 Let $p_\sigma:V(\sigma)\ra X$ be an \'etale map whose 
 pullback-log-structure is the normalized section. Note that 
 $$p_\sigma^{-1}\widetilde{Z}_\omega=\{f_\omega=0\}\subseteq V_e$$
 Let $f_\omega^\red$ be such that $(f_\omega^\red)^{a_\omega}=f_\omega$.
 Fixing some $1\le i\le r$, we claim that the functions $f_\omega^\red$ for
 $\omega\ra\tau\in\Upomega_{\tau,i}$ glue to
 a function $f_i$ on $\bigcup_{e\in\Upomega_{\tau,i}} V_e$. 
 This will follow if we show that the
 corresponding $Z_\omega$ glue because then their defining 
 functions $f_\omega^\red$ can at most differ by a non-trivial constant which is $1$ by
 the normalization assumption. To show this, we consider a diagram
 $$
 \xymatrix{
   \omega_1 \ar[rd] \ar@/^/[rrd]^{e_1}\\
   &\tau'\ar[r] & \tau\ar[r]^g & \sigma\\
   \omega_2 \ar[ru] \ar@/_/[rru]_{e_2}
 }$$
 which algebraic geometrically means that we have a stratum 
 $$X_{\tau'} \subset X_{\omega_1}\cap X_{\omega_2}$$
 on which we wish to show $Z_{\omega_1}\cap X_{\tau'}=Z_{\omega_2}\cap X_{\tau'}$. 
 By the c.i.t. property (2), it suffices to show 
 $$\Newton(Z_{\omega_1}\cap X_{\tau'})=\Newton(Z_{\omega_2}\cap X_{\tau'}).$$
 This follows from the c.i.t. property (3) because $\Newton(Z_{\omega_1}\cap X_{\tau'})$ 
 and $\Newton(Z_{\omega_2}\cap X_{\tau'})$ cannot be transverse due to the fact that
 $$\Newton(Z_{\omega_1}\cap X_{\tau})=\Newton(Z_{\omega_2}\cap X_{\tau})=\check\Delta_{\tau,i}$$
 is a non-trivial face of each. 
 Thus, we have functions $f_i$ as claimed before. We can naturally extend these to functions on $V(\tau)$ which 
 we are also going to denote by $f_i$. 
 
 We now choose coordinates $z_1,...,z_q$ on $\Int(X_\tau)\cong (\k^\times)^q$, and pull these back to functions on
 $V(\tau)$. By the c.i.t. property (3) we know that 
 the $Z_{\tau,i}$ meet transversely in 
 $\bar{v}:=p_\sigma^{-1}(\bar{x})$, so we can find a subset $\{i_1,...,i_r\}\subseteq \{1,...,q\}$ such that
 $F:=\det(\partial f_i/\partial z_{i_j})_{1\le i,j\le r}$ is invertible 
 in $\bar{v}$ (see \cite{eisenbud},~Cor.~16.20). 
 By reordering the indices, we can assume $\{i_1,...,i_r\}=\{1,...,r\}$. 
 For $r<i\le q$, we set $f_i:=z_i-z_i(\bar{v})$ 
 so that all $f_i$ vanish in $\bar{v}$ and give a set of local 
 coordinates at $\bar{v}$ when restricted to $\Int(X_\tau)$.
 We fix an isomorphism 
 $$V(\tau)\cong \Spec\k[\partial P']\times (\k^\times)^q.$$
 We are going to choose $V$ as a Zariski open neighbourhood of $\bar{v}$ in $V(\tau)$. 
 Before we say which one exactly, we define the map $\phi:V\ra X_\loc=\Spec \k[\partial P'\oplus\NN^q]$ by
 $$\begin{array}{rcll} 
  \phi^*(z^p)&=&z^p&\hbox{for }p\in\partial P'\\
  \phi^*(u_i)&=&f_i&\hbox{for }1\le i\le q
 \end{array}$$
 where $u_i$ is the monomial in $\k[\partial P'\oplus\NN^q]$ corresponding to 
 the $i$th standard basis vector of $\NN^q$. 
 Now $V$ is chosen in a way such that $\phi$ is \'etale. That this can be done is just a repetition of the
 argument in [\cite{grosie2},~Prf.~of~Thm.~2.6]. 
 To see that the pullbacks of the two log-structures to $V$ coincide, 
 we check that both give the same section in 
 $\Gamma(V,\bigoplus_{\atop{e:\omega\ra\sigma}{\dim\omega=1}}\shO_{V_e})$. 
 Arguing as in [\cite{grosie2},~Prf.~of~Thm.~2.6], this comes down to showing,
 for each $e:\omega\ra\tau$ and $n\in N'$,
 $$(d_\omega\otimes f_\omega)(n)=
 \prod_{i=1}^r f_i^{\langle m_i^--m_i^+,n\rangle}$$
 holds on the closed toric subset $V_e$ where
 \begin{itemize}
 \item $v^-,v^+$ are the vertices of $\omega$, 
 \item $d_\omega$ is the unique primitive vector pointing from $v^-$ to $v^+$ and
 \item $m_i^-, m_i^+\in \Hom(N',\ZZ)$ are $\check\psi_i|_{\check v^-},\check\psi_i|_{\check v^+}$,
 respectively where ${\check v^\pm}$ is the maximal cone in 
 $\check\Sigma_\tau$ corresponding to $v^\pm$.
 \end{itemize}
 We have that $\check\psi_i$ bends at $\check\omega$ 
 if and only if $e\in\Upomega_{\tau,i}$.
 By convexity $m_i^--m_i^+$ is positive on $\check v^-$ like $d_\omega$.
 Combining this with the
 fact that the edge of $\Delta_i$ corresponding to $\omega$ has 
 length $a_\omega$, this is just saying
 $$m_i^--m_i^+=\left\{\begin{array}{ll} 0& e\not\in\Upomega_{\tau,i}\\
 a_\omega d_\omega & e\in\Upomega_{\tau,i}.\end{array}\right.$$
 By construction, we have that $f_i|_{V_e}$ is invertible for $e\not\in\Upomega_{\tau,i}$.
 Since we have chosen $f$ to be normalized and
 the invertible elements of a toric monoid ring $k[P]$ are given as $k^\times\times P^\times$, we have
 $$f_i|_{V_e} =\left\{\begin{array}{ll} 1& e\not\in\Upomega_{\tau,i}\\
 f_\omega^\red|_{V_e} & e\in\Upomega_{\tau,i}.\end{array}\right.$$  
 Using $(f_\omega^\red)^{a_\omega}=f_\omega$, this finishes the proof. 
 Note that we have slightly simplified the proof as compared to loc.cit. by requiring $f$ to be normalized in the
 beginning (this implied $h_p=1$ for all $p$ in the notation of loc.cit.).
\end{proof}

The local models are the key ingredient to prove the base change for the hypercohomology
of the logarithmic de Rham complex:

\begin{proof}[Proof of Theorem~\ref{hyperbasechange}]  As argued in (\cite{kawamatanamikawa}, Lemma 4.1), one may assume that $A$
is a local Artinian $\k[t]$-algebra. Then, using the existence of local models from 
Prop.~\ref{localmodels}, 
the proof becomes literally the same as in (\cite{grosie2},~Thm.~4.1).
\end{proof}


\subsection{The barycentric resolution of the log Hodge sheaves}
\label{sec_barycentric}

The existence of local models for c.i.t. spaces enables many of the further constructions
in \cite{grosie2}. The entire section 3.1 in [\cite{grosie2}, Local calculations] 
doesn't use simplicity arguments and extends directly to the c.i.t. case.
In section 3.2,
Prop.~3.8, Theorem~3.9, Cor.~3.10, Cor.~3.11, Lemma~3.12 and Lemma~3.13 are valid in the
c.i.t. case. 
For this paper, we get three results from this.
First, we obtain a proof of Lemma~\ref{COmegaexact}, i.e., have the exact sequence
$$0\ra\Omega^r\ra \SC^0(\Omega^r)\ra \SC^1(\Omega^r)\ra ...$$
Roughly speaking, this is a resolution given by a kind of Deligne's simplicial scheme: 
We pull back the sheaf to each toric stratum and then apply the inclusion-exclusion-principle
to get a complex. 
Lemma~3.14 and Lemma~3.15 in loc.cit., however, fail to be true in the c.i.t. setting. 
At least in the proof of 3.15 simplicity is being used explicitly. 
In any event, we are not interested in these results. 
We just use a small replacement for Lemma~3.14:

\begin{lemma} \label{vertex}
Given a c.i.t. space $X$, $\tau\in\P$, then $v\in\tau^{[0]}$ induces
a canonical choice of a vertex
$$\Ver_i(v)\in\Delta_{\tau,i}\hbox{ for each }i.$$
\end{lemma}
\begin{proof} Let $\check{v}$ denote the maximal cone in the normal fan 
$\check\Sigma_\tau$ corresponding to $v$. 
Choose $v_i=\Ver_i(v)\in\Delta_{\tau,i}^{[0]}$ such that 
$$v_i-\Delta_{\tau,i}\subseteq \check{v}^\dual.$$
There is at most one such $v_i$ because otherwise $\check{v}^\dual$ would have to contain a
straight line. There exists one because $\check\psi_{\tau,i}$, the piecewise linear
function associated to $\Delta_{\tau,i}$, is linear on $\check{v}$. 
In fact, we have $\check\psi_{\tau,i}|_{\check{v}}=-\langle v_i,\cdot\rangle$.
\end{proof}

The third consequence for c.i.t. spaces 
which is most importance to us is [\cite{grosie2}, Prop.~3.17] which we cite here. 
Just note that
$\op{dlog}(f^a)=a\cdot\op{dlog}f$ has the same poles as $\op{dlog}f$ which is the reason
why the sheaves of differentials only care about $Z^\red$ and not $Z$.

\begin{proposition} \label{delta0kernel} 
Let X be c.i.t.. 
Given $\P^{[0]}\ni v\stackrel{g}{\ra}\tau_1\stackrel{e}{\ra}\tau_2$, the image of the inclusion
$(F_{s}(e)^*\Omega^r_{\tau_1})/\Tors$ in $F_{s}(e\circ g)^*\Omega^r_v$ is
\[
\ker\bigg(
F_{s}(e\circ g)^*\Omega^r_v\stackrel{\delta}{\lra}
\bigoplus_{\atop{i=1,\ldots,q}{w_i\not=v_i}} \Omega^{r-1}_{(Z'_{\tau_2,i})^{\dagger}/\k^{\dagger}}\bigg)
\]

where: 
\begin{enumerate}
\item We set $v_i=\Ver_i(v)$.
The direct sum is over all $i$ and all vertices 
$w_i\in\Delta^{[0]}_{\tau_1,i}$ with $w_i\not=v_i$.
\item $Z'_{\tau_2,i}=F_{s}(e)^{-1}(Z_{\tau_1,i})$ which might be empty.
\item We define a log structure on $X_v$ as the pushforward of the pull-back via 
$X_v\backslash Z\hra X\backslash Z$ ($Z$ has codimension two in $X_v$). Then,
$\Omega^{r-1}_{(Z'_{\tau_2,i})^{\dagger}/\k^{\dagger}}$ is defined by pulling back the
log structure from $X^\log_v$ via $X_{\tau_1}$ and $Z_{\tau_1,i}$ to $Z'_{\tau_2,i}$.
\item
For $\alpha\in F_{S,s}(e\circ g)^*\Omega^r_v$, the component of
$\delta(\alpha)$ in the direct 
summand $\Omega^{r-1}_{(Z'_{\tau_2,i})^{\log}/\k^\log}$
corresponding to some $w_i$ is given by
$\iota(\partial_{w_i-v_i})\alpha|_{(Z'_{\tau_2,i})^{\dagger}}$. 
\end{enumerate}
\end{proposition}

In a setting more general than the simple case with standard outer monodromy simplices 
as dealt with in \cite{grosie2}, the resolution 
$\SC^\bullet(\Omega^r)$ is no longer acyclic. 
One of the main points of this paper is 
to produce an acyclic resolution for this complex. 
We start off by constructing a resolution of a summand
of $\SC^\bullet(\Omega^r)$ on a single stratum 
for which we are going to use a Koszul complex.


\section{Koszul cohomology} \label{section_koszkoho}
This section is separate from the previous ones and 
we will be reusing some of the notations.

\subsection{The general setting}
Koszul cohomology is a globalized version of the Koszul complex for 
affine rings (see \cite{eisenbud}, Exc.17.19).
It was extensively developed and exploited by Mark Green in 
\cite{green1} and \cite{green2}.

Let $X$ be a $\k$-variety, $\shL$ an invertible sheaf on $X$ 
and $V$ a finite dimensional linear subspace of $\Gamma(X,\shL)=\Hom(\shO_X,\shL)$. 
Defining $\phi(1)$ as the canonical map $(V\otimes\shO_X\ra\shL)\in V^*\otimes\shL$, we get a complex
$$0\ra \shO_X \stackrel{\phi}{\ra} \shL\otimes\bigwedge^1 V^* 
\ra \shL^{\otimes 2}\otimes\bigwedge^2 V^*
\ra \shL^{\otimes 3}\otimes\bigwedge^3 V^*\ra...$$
whose differential is given by $\alpha\mapsto \phi(1)\wedge\alpha$.
We introduce the notation
$$\shK^i(V,\shL):=\shL^{\otimes i}\otimes\bigwedge^i V^*$$
and call the according differential $d^i$. 
For $m\in\ZZ$, we define the twists 
$\shK^i(V,\shL,m):=\shL^m\otimes\shK^i(V,\shL)$
and denote the dual by $\shK_i(V,\shL,m)$.

\begin{lemma} \label{koszulexakt}
If $V$ is base point free, i.e., $V\otimes\shO_X\ra\shL,\,s\otimes g\mapsto s(g)$ 
is surjective, then for each $m\in\ZZ$ the Koszul complexes 
$\shK^\bullet(V,\shL,m) \hbox{ and }\shK_\bullet(V,\shL,m)\hbox{ are exact.}$
\end{lemma}
\begin{proof} Since $\shL$ is locally free, it suffices to show exactness
for one $m$. Furthermore, it is enough to show it for $\shK_\bullet(V,\shL)$, and
exactness of the dual follows by
applying $\shHom(\cdot,\shO_X)$ because the complex is locally free.
Exactness can be considered at stalks. After choosing a basis
$\{x_1,...,x_n\}$ of $V$, in the notation of \cite{eisenbud}, the Koszul complex is locally isomorphic to
$K(x_1,...,x_n)$. Due to base point freeness of
$V$, at each stalk at least one of the $x_i$ is invertible.
By [\cite{eisenbud}, Prop.17.14 a)] multiplication with $x_i$ annihilates the cohomology
of $\shK^\bullet(V,\shL)$ which is thus trivial.
\end{proof}


\subsection{Semi-ample line bundles on toric varieties} \label{sec_torickoszul}
In the following, we fix a toric variety $X$ with character lattice $M$.
A Cartier divisor $Z$ on $X$ is linearly equivalent to a 
(non-unique) torus invariant Cartier divisor $D$. By a standard procedure 
(see \cite{fulton}) we associate to $D$ its support function $\psi_D$ which
is a piecewise linear function on the fan of $X$.
By [\cite{danilov1},~Prop.~6.7], 
semi-ampleness of $Z$ is equivalent to the convexity of $\psi_D$.
Convex piecewise linear functions, in turn, correspond to 
lattice polytopes $\Delta\subset M\otimes_\ZZ \RR$ whose normal fan can be refined 
to the fan of $X$, and $\Delta$ is uniquely associated to $Z$ up 
to translation by lattice vectors.
By [\cite{fulton}, Lemma in 3.4], one has
$$\Gamma(X,\shO_X(D))=\bigoplus_{m\in\Delta\cap M} \k\cdot z^m.$$
An element of this corresponds to an effective divisor which is linearly equivalent to $D$.
The correspondence is 1:1 up to the operation of $\k^\times$ on $\Gamma(X,\shO_X(D))$.
The element of the right hand side gives an equation of the divisor on 
the big torus as a Laurent polynomial. 
In the following, when we talk about the \emph{equation of a divisor} we mean a 
corresponding element $f\in\bigoplus_{m\in\Delta\cap M} \k\cdot z^m$.
Occasionally, we will consider $f$ as an element of the field of fractions 
$\Quot(\shO_X)$ in which $\Gamma(X,\shO_X(D))$ naturally embeds as $\{q\,|\, \op{div}(q)\ge -D\}$.
In the following, we always consider a fixed translation representative of~$\Delta$.

\begin{lemma} \label{globalsec}
Let $Z$ be an effective divisor on $X$ which is linearly equivalent to $D$ and is given by an equation
$f$. There are isomorphisms
$$\bigoplus_{m\in\Delta\cap M} \k\cdot z^m \stackrel{\cdot \frac1{f}}{\lra} \Gamma(X,\shO_X(Z))
\qquad\hbox{ and }\qquad
\bigoplus_{m\in\Delta\cap M} \k\cdot ({z^m})^*\stackrel{\cdot f}{\lra} \Gamma(X,\shO_X(Z))^*.$$
\end{lemma}
\begin{proof} 
We have $\shO_X(D)=f\cdot\shO_X(Z)$ in the field of fractions of $\shO_X$.
This induces the isomorphism of global sections and the one for their duals.
\end{proof}

Set $N=\Hom(M,\ZZ)$, and choose some equation $f=\sum_{m\in\Delta\cap M} f_m z^m$ 
of a divisor $Z$. 
Using Lemma~\ref{globalsec}, we define the \emph{log derivation map}
$$\partial_Z:(N\oplus\ZZ)\otimes_\ZZ \k \ra \Gamma(X,\shO_X(Z))$$
via
$(n,a)\mapsto f^{-1}\partial_{(n,a)} f 
= f^{-1}\sum_{m\in\Delta\cap M} f_m \langle (n,a),(m,1) \rangle z^m$
and denote its image by $V$. 
Note that neither $V$ nor $\partial_Z$ depends on the scalar multiple of an equation.
However, $\partial_Z$ depends on the translation representative of $\Delta$
whereas $V$ doesn't depend on it. 
In Remark~\ref{moveDelta}, we discuss what happens if we move $\Delta$.
We define the cone
$C(\Delta)=\{\sum_{v\in\Delta^{[0]}} \lambda_v (v,1)\,|\, \lambda_v\ge 0\}
\subset (M\oplus\ZZ)\otimes_\ZZ\RR.$
where $\Delta^{[k]}$ is the set of $k$-dimensional faces of $\Delta$.
Let $LC(\Delta)$ denote the linear subspace generated by $C(\Delta)$.
We define $\hat T_\Delta=(LC(\Delta)\cap (M\oplus\ZZ))\otimes_\ZZ \k$ and think of
it as the $\k$-valued tangent space of the cone.

\begin{lemma} \label{dimV}
If $Z$ is a semi-ample effective divisor on $X$ with Newton polytope $\Delta$, then
$$0\ra \hat T_\Delta^\perp \ra (N\oplus\ZZ)\otimes_\ZZ \k \stackrel{\partial_Z}{\lra} V\ra0$$
is an exact sequence and thus $\dim V=\dim\Delta+1$.
\end{lemma}

\begin{proof} We show $\dim V\ge\dim\Delta+1$, then the assertion follows because
$LC(\Delta)^\perp\cap((M\oplus\ZZ)\otimes\kk)$ clearly lies in the kernel.
Let $f$ be an equation of $Z$. By the 
hypotheses, $f_v\neq 0$ for $v\in\Delta^{[0]}$. 
Let $\{v_1,...,v_n\}$ be a $\dim\Delta+1$ element subset of $\Delta^{[0]}$ 
whose convex hull has dimension $\dim\Delta$.
For each $1\le i\le n$, choose $n_i\in (N\oplus\ZZ)\otimes_\ZZ \k$
such that $\langle n_i,(v_j,1)\rangle=\delta_{ij},$
the latter being the Kronecker symbol. Consider the composition
$$(N\oplus\ZZ)\otimes_\ZZ \k \ra \Gamma(X,\shO_X(Z))\ra \bigoplus_{i=1}^n \k f^{-1}z^{v_i}$$
where the last map is the natural projection
$\bigoplus_{m\in\Delta\cap M} \k f^{-1}z^{m} \sra \bigoplus_{i=1}^n \k f^{-1} z^{v_i}.$
The image of $n_i$ is $f_{v_i}f^{-1}z^{v_i}$. Therefore, the map is surjective, and we are done.
\end{proof}

We define the ``tangent space''
$T_{\Delta}=\{a\cdot(x-y)\,|\,a\in \k, x,y\in\Delta^{[0]}\}$
and obtain an exact sequence
$$0\ra T_\Delta\ra \hat T_\Delta \stackrel{h}{\lra} \k\ra 0$$
where $h$ is the height function coming from the projection to the $\ZZ$ summand.
The above lemma induces an isomorphism $(\partial_Z)^*:V^*\ra \hat T_\Delta.$
We may plug it into the Koszul complex to get an isomorphism of complexes
$$\begin{CD}
\dots @>>> \shK^{l-1}(V,\shO(Z),m) @>d^{l-1}>>  \shK^l(V,\shO(Z),m) @>d^{l}>> \dots \\
&&@VV{\id\otimes\bigwedge^{l-1}(\partial_Z)^*}V  @VV{\id\otimes\bigwedge^{l}(\partial_Z)^*}V \\
\dots @>>> \shO_X((l-1+m)Z)\otimes_\k\bigwedge^{l-1}\hat{T}_\Delta 
@>>> \shO_X((l+m)Z)\otimes_\k\bigwedge^{l}\hat{T}_\Delta @>>>\dots
\end{CD}$$
We denote the complex in the lower row by $\shK^{\bullet}(T_\Delta,Z,m).$
This complex will play a major role in this paper. 
We now give an explicit description of its differential.
\begin{lemma} \label{dZexplicit}
The differential of $\shK^{\bullet}(T_\Delta,Z,m)$ is
$$d_Z:u\otimes \alpha 
\mapsto u\cdot\sum_{m\in\Delta\cap M} f_m \frac1f z^{m}\otimes (m,1)\wedge \alpha.$$
\end{lemma}

Note that in the lemma, we understand $\frac1f z^{m}$ as a rational function.

\begin{proof} Let us take a look at the composition
$\Gamma(X,\shO_X(Z))^*\ra V^*\stackrel{(\partial_Z)^*}{\lra}\hat{T}_\Delta.$
Its dual map sends $(n,a)$ to 
$\sum_{m\in\Delta\cap M} f_m \langle (n,a),(m,1) \rangle f^{-1}z^m$, so we have
$$f(z^m)^*=(f^{-1}z^m)^*\ \mapsto\  
((n,a)\mapsto f_m \langle (n,a),(m,1) \rangle)=f_m\cdot (m,1)$$
Now it is straightforward to see that the Koszul differential
$u\otimes \alpha 
\mapsto u\cdot\sum_{m\in\Delta\cap M} \frac1f z^{m}\otimes (\frac1f z^{m})^*\wedge \alpha$
transforms to $d_Z$ as given in the assertion.
\end{proof}

Recall the non-degeneracy definition from before Def.~\ref{def_ht}.
The set of non-degenerate divisors form a Zariski open set in 
$\Gamma(X,\shO_X(D))$ which can be deduced from Bertini's theorem, see
\cite{danilov1}.

\begin{lemma} \label{bpf} 
If $Z$ is non-degenerate, then $V$ is base-point free.
\end{lemma}
\begin{proof} See \cite{batyrev1}, Prop.~4.3.
\end{proof}

\begin{lemma} \label{cohotoriclb}
Let $Z$ be a semi-ample effective divisor on a toric variety $X$ whose Newton polytope is
$\Delta$. For $m\in\ZZ$, we have
$$H^i(X,\shO_X(mZ))=0\qquad\hbox{ for }0<i<\dim\Delta\hbox{ and for }i>0,m\ge 0.$$
\end{lemma}

\begin{proof} 
This is easy using the techniques of [\cite{fulton}, 3.5].
The part for $m\ge 0$ is, in fact, the Corollary in loc.cit.. 
By similar arguments, the remaining part can be reduced to showing that
$H^i(\RR^n,\RR^n\backslash C;\k)=H^i_C(\RR^n;\k)=0$ for $0<i<\dim\Delta$ 
where $C$ is either empty or a polyhedral cone whose largest
linear subspace has dimension $\codim\Delta$. The empty case is trivial, otherwise
one may use $H^i_C(\RR^n;\k)=H^{i-1}(\RR^n\backslash C;\k)\hbox{ for }i>1$ and
$H^0_C(\RR^n;\k)=H^1_C(\RR^n;\k)=0$
via the long exact sequence of relative cohomology. There are two possibilities,
either $C$ is a linear subspace or it is not. If it is not then its complement
is contractible and we are done. 
If it is a linear subspace, its dimension is $d=\codim\Delta$. Then
$H^{i-1}(\RR^n\backslash C;\k)=H^{i-1}(\RR^{n-d}\backslash \{0\};\k)$
which vanishes for $(i-1)<n-d-1\ \IFF\ i<\dim\Delta$.
\end{proof}

\begin{proposition} \label{koszulcoho}
Let $Z$ be a non-degenerate divisor on a toric variety $X$ with Newton polytope $\Delta$.
Set $n=\dim\Delta+1$ and 
$HK^i(V,Z,m):=H^i_{d^\bullet}\Gamma(\shK^\bullet(V,\shO_X(Z),m))$.
We have
$$HK^i(V,Z,m)=\left\{\begin{array}{ll}0 & \hbox{ for }i\neq n\\ 
R(Z)_{n+m}\otimes_\k\bigwedge^{n}V^*&\hbox{ for }i=n. \end{array}\right.$$
\end{proposition}

\begin{remark} Using elementary results from 
Section~\ref{sec_jacobianrings} to show $R(Z)_{n+m}=0$ for $m\ge 1$,
this generalizes the $d=1$ case of the vanishing theorem 
\cite{green2},~Thm.~2.2 to toric varieties.
\end{remark}

\begin{proof} 
By Lemma~\ref{dimV}, we have $n=\dim V^*$. 
The case $Z=0$, i.e., $R(Z)_{\bullet>0}=0$, is trivial. So let us assume $n>1$.\\
\underline{Step 1:}\ We first show the vanishing for $i\neq n$.
The vanishing for $i>n$ is clear.
By Lemma~\ref{bpf} and Lemma~\ref{koszulexakt}, the complex
$\shK^{\bullet}(V,Z,m)$ is exact, and we may interpret it as
a resolution of the first non-trivial term. Hence
$H^i(X,\shO(mZ))=\HH^{i+1}(X,\shK^{\bullet>0}(V,Z,m)).$
The vanishing for $i=0,1$ follows from the left-exactness of the
functor $\Gamma$.
By Lemma~\ref{cohotoriclb}, if $m\ge 0$, we are done because
the Koszul complex is an acyclic resolution of the first term, 
so its hypercohomology coincides with its cohomology after taking $\Gamma$.
In general, we may consider the $E_1$-term of the first hypercohomology 
spectral sequence of $\shK^{\bullet>0}(V,Z,m)$.
By Lemma~\ref{cohotoriclb}, it looks like
$$
 \begin{array}{ccccccccc}   
    H^{n-1}(X,\shO_X((m+1)Z))\otimes \bigwedge^1 V^*&\stackrel{d_1}{\ra}
   &H^{n-1}(X,\shO_X((m+2)Z))\otimes \bigwedge^2 V^*&\stackrel{d_1}{\ra} &\cdots\\      
   0&&0\\
   \vdots&&\vdots\\
   0&&0\\
   H^{0}(X,\shO_X((m+1)Z))\otimes \bigwedge^1 V^*&\stackrel{d_1}{\ra}
   &H^{0}(X,\shO_X((m+2)Z))\otimes \bigwedge^2 V^*&\stackrel{d_1}{\ra} &\cdots\\      
 \end{array}
$$ 

Note, that the $d_1$-cohomology of the bottom sequence is what we are interested in. 
The spectral sequence differential
$$d_k:H^{n-1}(X,\shO_X((m+s)Z))\otimes \bigwedge^s V^*
\ra H^{n-k}(X,\shO_X((m+s+k)Z))\otimes \bigwedge^{s+k} V^*$$
hits the bottom line for $k=n$. 
Thus, the leftmost term it reaches is the one with $\bigwedge^{n+1} V^*$ which is zero.
Hence the sequence degenerates at $E_1$ and we have for $0<i<n-1$
$$0=H^i(X,\shO(mZ))=\HH^{i+1}(X,\shK^{\bullet>0}(V,Z,m))
=HK^{i+1}(V,Z,m).$$

\underline{Step 2:}\ Now, let's have a look at the last two non-trivial terms in the 
Koszul complex which are
$\Gamma(X,\shO_X((n-1+m)Z))\otimes_\k\bigwedge^{n-1}V^*\ra
\Gamma(X,\shO_X((n+m)Z))\otimes_\k\bigwedge^{n}V^*.$
Using the identification 
$\bigwedge^{n-1}V^* = 
V\otimes_\k\bigwedge^{n}V^*,$ 
this map is canonically isomorphic to the map
$$\Gamma(X,\shO_X((n-1+m)Z))\otimes_\k V
\ra \Gamma(X,\shO_X((n+m)Z))$$
tensored with the identity on $\bigwedge^{n}V^*$.
Its cokernel is thus $R(Z)_{(n+m)}\otimes_\k\bigwedge^{n}V^*$.
\end{proof}

\begin{corollary} \label{corkoszulcoho}
Let $Z$ be a non-degenerate divisor on a toric variety with Newton polytope $\Delta$.
Set $n=\dim\Delta+1$. We have
$$H^i_d \Gamma(\shK(T_\Delta,Z,m))=\left\{\begin{array}{ll}0 & \hbox{ for }i\neq n\\ 
R(Z)_{n+m}\otimes_\k\bigwedge^{n-1}T_\Delta&\hbox{ for }i=n \end{array}\right.$$
\end{corollary}

\begin{proof} We apply the log derivation map and the contraction by $h$ which yields 
$\bigwedge^{n}\hat{T}_\Delta=\bigwedge^{n-1}T_\Delta$.
\end{proof}

For an injection of $\k$-vector spaces $T_{\Delta}\hra W$ we define $\hat{W}$ by
the cocartesian diagram
$$\begin{CD} T_{\Delta} @>>> W\\
@VVV @VVV\\
\hat{T}_{\Delta}@>>> \hat{W}
\end{CD}$$
and the complex $\shK^l(W,Z,m):= \shO_X((l+m)Z)\otimes_\k\bigwedge^{l}\hat{W}$
for varying $l$ with ``the same'' differential $d_Z$. A good way to think of
$\shK^l(W,Z,m)$ is as being the Koszul complex $\shK^l(T_{\hat\Delta},\hat Z,m)$ 
pulled back from some higher dimensional toric variety $\hat X$ 
in which $X$ embeds equivariantly, see Lemma~\ref{lempullbackres}.

\begin{lemma}  \label{decompose}
There is a non-canonical direct sum decomposition of the complex 
$$\shK^\bullet(W,Z,m)\, \cong\, \bigoplus_{b\ge 0} \shK^{\bullet-b}(T_{\Delta},Z,m+b)
\otimes_\k \bigwedge^b W/ T_{\Delta}.$$
\end{lemma}

\begin{proof} The inclusion $\hat{T}_{\Delta}\hra \hat{W}$ induces a filtration of 
$\bigwedge^l \hat{W}$ for each $l$ which splits as
$\bigwedge^l \hat{W}\,\cong\,\bigoplus_{b\ge 0} \bigwedge^{l-b}\hat{T}_{\Delta} \otimes_\k 
\bigwedge^{b}\hat{W}/ \hat{T}_{\Delta}$
because we are dealing with vector spaces. The differential $d_Z$ respects this
splitting going 
{\scriptsize
$$d_Z:
\shO_X((l+m)Z)\otimes \bigwedge^{l-b}\hat{T}_{\Delta} \otimes \bigwedge^{b} \hat{W}/\hat{T}_{\Delta}
\ra
\shO_X((l+1+m)Z)\otimes \bigwedge^{l+1-b}\hat{T}_{\Delta} \otimes \bigwedge^{b}\hat{W}/\hat{T}_{\Delta}$$
}The result is now just a matter of identifying the terms on the left of the right
tensor symbol with the complex for $T_\Delta$ and using $\hat{W}/\hat{T}_{\Delta}=W/ T_{\Delta}$.
\end{proof}

%
%
Even though the splitting of the complex is not canonical, in a sense, the splitting on
cohomology is. For a vector space $T$, we occasionally 
write $\bigwedge^{\top}T$ for $\bigwedge^{\dim T}T$. For an inclusion of
vector spaces $T\hra U$, whenever there is no confusion with another inclusion, we write
$\langle \bigwedge^t T\rangle_l$ 
for the degree $l$ part of the exterior algebra ideal 
in $\bigwedge^\bullet U$ generated by $\bigwedge^t T$.
We set $n=\dim\Delta+1$ and $HK(W,Z,m) := H^i_d \Gamma(\shK(W,Z,m))$.

\begin{proposition} \label{cohoW}
 With the above notation, we have
 $$\begin{array}{rcl}
 HK^l(W,Z,m) 
 &=& R(Z)_{l+m}\otimes_\k \big\langle\bigwedge^{\top} T_\Delta\big\rangle_{l-1}\\[1mm]
 &=& R(Z)_{l+m}\otimes_\k \bigwedge^{\top} T_{\Delta}
 \otimes\bigwedge^{l-n}W/ T_{\Delta} 
 \end{array}$$
\end{proposition}

\begin{proof} Using Lemma~\ref{decompose}, we have the non-canonical decomposition
 $HK^l(W,Z,m)\, \cong\, \bigoplus_{b\ge 0} HK^{l-b}(T_{\Delta},Z,m+b) \otimes_\k 
 \bigwedge^b W/T_{\Delta}.$
 By Cor.~\ref{corkoszulcoho} the only non-zero term on the 
 right hand side is the one where $l-b=n$. Hence, we have
 $HK^l(W,Z,m)\cong R(Z)_{m+n+b}\otimes\bigwedge^{\top} T_{\Delta}
 \otimes_\k \bigwedge^{b}W/ T_{\Delta}$
 for $b=l-n$.
 It remains to prove the canonicity. 
 Consider the canonical filtration
 $$ \bigwedge^l\hat W = \big\langle \bigwedge^0 \hat{T}_\Delta\big\rangle_l 
 \supset \big\langle \bigwedge^1 \hat{T}_\Delta\big\rangle_l
 \supset \dots \supset \big\langle \bigwedge^{n} \hat{T}_\Delta\big\rangle_l\supset \{0\}$$
 The desired cohomology group comes from the non-trivial bottom term, more precisely, 
 it is the cokernel of the left vertical arrow in the diagram
 $$\begin{CD}
 \shO((m+l)Z)\otimes{\big\langle} \bigwedge^{n} \hat{T}_\Delta{\big\rangle}_l 
 @= \shO((m+l)Z)\otimes\bigwedge^{n} \hat{T}_\Delta\otimes \bigwedge^{l-n}\hat{W}/ \hat{T}_\Delta\\
 @AA{d_Z}A @AA{(d_Z\otimes\id)}A\\
 \shO((m+l-1)Z)\otimes{\big\langle} \bigwedge^{n-1} \hat{T}_\Delta{\big\rangle}_{l-1}
 @<\id\otimes\xi<< \shO((m+l-1)Z)\otimes\bigwedge^{n-1} \hat{T}_\Delta\otimes \bigwedge^{l-n}\hat{W}/ \hat{T}_\Delta
 \end{CD}$$
 The bottom map $\xi$ is the only non-canonical map. It is supposed to be
 a section of the right non-trivial map in the exact sequence
 $$ 0
 \lra \big\langle\bigwedge^{n} \hat{T}_\Delta\big\rangle_{l-1}
 \lra  \big\langle\bigwedge^{n-1} \hat{T}_\Delta \big\rangle_{l-1} 
 \lra \bigwedge^{n-1} \hat{T}_\Delta\otimes \bigwedge^{l-n}\hat{W}/ \hat{T}_\Delta
 \lra 0 .$$
 Then any choice of $\xi$ makes the diagram commute, because 
 $\shO((m+l-1)Z)\otimes\langle\bigwedge^{n} \hat{T}_\Delta\rangle_{l-1}$
 is contained in the kernel of the left vertical map. 
 Thus, we get a canonical identification of the
 cokernels of the vertical maps which shows
 $HK^l(W,Z,m) = C\otimes\bigwedge^{n} \hat{T}_\Delta\otimes 
 \bigwedge^{l-n}\hat{W}/ \hat{T}_\Delta$
 where $C=\op{coker}(q)$ and 
 $q:\Gamma(X,\shO_X((m+l-1)Z))\otimes \hat{T}_\Delta^*\mapsto \Gamma(X,\shO_X((m+l)Z))$, 
 $u\otimes \hat{n}\mapsto u\cdot (\frac1f \partial_{\hat{n}}f)$. 
 We see that $C=R(Z)_{l+m}$. Let $\iota(h)$ denote the contraction 
 by the natural projection $h:\hat{T}_\Delta\sra \k$.
 We can apply the isomorphisms 
 \begin{itemize}
 \item$\iota(h):\bigwedge^{n}\hat{T}_\Delta\ra 
 \bigwedge^{\dim T_\Delta}T_\Delta$
 \item $\hat{W}/ \hat{T}_\Delta=W/ T_{\Delta}$
 \item $\langle\bigwedge^\top T_{\Delta} \rangle_l \ra \bigwedge^\top T_{\Delta}
 \otimes\bigwedge^{l-\dim T_{\Delta}}W/ T_{\Delta}$,\ \ 
 $\alpha_\top\wedge \alpha_W\mapsto\alpha_\top\otimes [\alpha_W]$.
 \end{itemize}
  to obtain the result.
\end{proof}


\subsection{Jacobian rings and Newton polyhedra}
\label{sec_jacobianrings}
We wish to analyze the relation of $R(Z)$ to Jacobian rings in this subsection.
We stay with the previous notation and assume here that $\Delta$ is a simplex and define
a relatively open subset of its cone by
$$C(\Delta)^{\core{}}=C(\Delta)\left\backslash\bigcup_{v\in\Delta^{[0]}}
(v,1)+C(\Delta)\right. .$$
It is easily seen to be the half open parallelepiped
$C(\Delta)^{\core{}}=\{\sum_{v\in\Delta^{[0]}} \lambda_v (v,1)\,|\,
0\le\lambda_v<1\}.$
For each $l\in\ZZ_{\ge 0}$ we may intersect this with the hyperplane 
$\{(m,l)\,|\, m\in M\otimes_\ZZ\RR \}$ and project to the first summand to have
$$\Delta^{\core{l}}=\{\sum_{v\in\Delta^{[0]}} \lambda_v v\,|\,
0\le\lambda_v<1,l=\sum_v\lambda_v\}\ \subseteq\ l\cdot\Delta.$$
This space was already defined in \cite{bormav},~Ch.~9.
One finds
$\Delta^{\core{l}}=\emptyset\ \IFF\ l\ge |\Delta^{[0]}|=\dim\Delta+1.$
This gadget has nice functorial properties. 
For a subset $S\subseteq \RR^n$, let 
$\relint S$ denote the relative interior of $S$ in $\span{\RR}S$.

\begin{lemma} \label{corefunc}
 We have
 $$ \begin{array}{ccc}
 \hbox{a)}\quad \displaystyle\Delta^{\core{l}}\cap (l\cdot F)=F^{\core{l}}&\qquad&
 \hbox{b)}\quad 
 \displaystyle\Delta^{\core{l}}=\coprod_{F\subseteq \Delta}
 \relint F^\core{l}  \end{array}$$
 where in each of these $F\subseteq \Delta$ is supposed to be a face.
\end{lemma}

\begin{proof} 
 To see that a) is true just note that a face $F$ is determined
 by the set of vertices it contains. It is then given by points of $\Delta$ 
 for which $\lambda_v=0$ whenever $v\not\in F$. Then a) easily follows and
 b) is a consequence of a).
\end{proof}

\begin{definition} 
 We denote by $\Gamma^{\core{l}}(Z)$
 the subspaces of $\Gamma(X,\shO_X(lZ))$ generated by the images of $z^m$ under the first map in
 Lemma~\ref{globalsec} for which $m\in \Delta^{\core{l}}$. This is independent of
 a particular equation of $Z$.
\end{definition}

We consider the monoid ring $\k[C(\Delta)\cap (M\oplus\ZZ)]$
which is Noetherian and graded by the second summand. Assume that we are 
given a homogeneous element of degree one
$f =\sum_{m\in\Delta\cap M} f_m z^{(m,1)}.$
One defines the Jacobian ideal of $f$ by
$$J_f= ( \partial_n f \,|\, n\in\Hom(M\oplus\ZZ,\ZZ) )$$
where $\partial_n f =\sum_{m\in\Delta\cap M} \langle (m,1) , n\rangle f_m z^{(m,1)}$.
Relating to Griffith's work, Batyrev has used the notation $R_0,R_1$ for two types of 
toric Jacobian rings in \cite{batyrev1}.
In \cite{bormav}, Borisov and Mavlyutov have picked up this notation.

\begin{definition}[Batyrev, Borisov, Mavlyutov] \label{DefR0R1}
 We set
 $$R_0(f,\Delta)=\k[C(\Delta)\cap (M\oplus\ZZ)]/ J_f.$$
 and define $R_1(f,\Delta)$ as the subspace generated by lattice 
 points in $\relint C(\Delta)$ which yields a module over $R_0(f,\Delta)$.
\end{definition}

We can put a ring structure on $R(Z)$ by writing it as the quotient 
of the global sections tensor algebra $\Gamma(X,\shO_X(\bullet Z))$ 
by the ideal generated in degree one by the linear system of log derivatives $V$.

\begin{lemma} \label{compareRs} 
 There is a graded ring isomorphism
 $$R(Z)\cong R_0(f,\Delta)$$
 which is canonical up to a multiplicative constant.
\end{lemma}

\begin{proof} 
 We obtain the inverse of the desired isomorphism via the unique ring map 
 induced in degree one by
 $\bigoplus_{m\in\Delta\cap M} \k\cdot z^{(m,1)}
 \stackrel{\cdot \frac1{f}}{\lra} \Gamma(X,\shO_X(Z))$
 as given in Lemma~\ref{globalsec}. 
 It maps the respective ideals to each other as can be seen from the definition of
 the log derivation map. The remark about the multiplicative constant addresses 
 the fact that $R_0(f,\Delta)=R_0(af,\Delta)$ for $a\in\k^\times$ whereas
 the isomorphism depends on $a$.
\end{proof}

We define the vector space
$\k^{\Delta^{\core{l}}\cap M}
=\left\{\left. \sum_{m\in\Delta^\core{l}\cap M} a_m z^m\,\right| \, a_m\in\k \right\}.$

\begin{lemma} \label{leminjection}
 Let $f\in\k[C(\Delta)\cap M\oplus\ZZ]_1$ be arbitrary. The map
 $$ \k^{\Delta^{\core{l}}\cap M} \ra R_0(f,\Delta)_l $$
 given by sending $z^m$ to $z^{(m,l)}$ is an injection.
\end{lemma}

\begin{proof} We call an $f$ with the property
$f_m\neq 0 \IFF m\in\Delta^{[0]}$ \emph{Fermat}.
Note that the lemma is true for all Fermat $f$ because then
$J_f=(z^{(v,1)}| v\in\Delta^{[0]})$ and its degree $l$ part is
$(J_f)_l=\left\{\left.\sum_{\atop{m\in l\Delta\cap M}{m\not\in\Delta^\core{l}}} a_m z^{(m,l)}\,\right|\, a_m\in\k \right\}.$
The kernel of the map in the lemma gives a coherent module on the space 
$\Spec \k[ f_m | m\in {\Delta\cap M}]$ of all $f$. It is trivial at Fermat points and
therefore also in a neighbourhood $U$ of these points.
There is an operation of the torus $\mathbb{G}_m(\k)^{\Delta\cap M}$ on this space under which
$\k^{\Delta^{\core{l}}\cap M}$ is invariant. 
We deduce the result for all $f$ which lie in an orbit with non-trivial intersection with $U$.
The other $f$ can easily be checked directly.
\end{proof}

We call $f$ non-degenerate if the corresponding $Z$ is non-degenerate.

\begin{proposition} \label{Fermatiso}
 If $f$ is non-denerate then, for each $l$, the map
 $$ \k^{\Delta^{\core{l}}\cap M} \ra R_0(f,\Delta)_l $$
 given by sending $z^m$ to $z^{(m,l)}$ is an isomorphism.
\end{proposition}

\begin{proof} It is not hard to see that if $f$ is Fermat then $Z$ is non-degenerate and
we saw in the proof of the previous lemma that for these $f$ the assertion is true.
The result then follows from Lemma~\ref{leminjection} and [\cite{batyrev1}, Thm.~4.8] which 
states that two linearly equivalent non-degenerate divisors have the
same graded dimensions for their Jacobian rings.
\end{proof}

A general lattice polytope can always be triangulated by elementary simplices, so they
form the building blocks of lattice polytopes. A special case of these is the
\emph{standard simplex} which is one that is isomorphic to the convex hull of
$0$ and a subset of a lattice basis. Gross and Siebert required these in \cite{grosie2} 
for the outer monodromy polytopes to make their Hodge group computation work.
Here is how these relate to Jacobian rings and thus to the cohomology of the
Koszul complex.

\begin{lemma} \label{standardequiv}
For a lattice simplex $\Delta$ with lattice $M$, the following are equivalent
\begin{enumerate}
\item[a)] $\Delta$ is standard
\item[b)] $\relint( n\cdot\Delta )\cap M=\emptyset$ for $n\le \dim\Delta$.
\item[c)] $R_0(f,\Delta)_l=0$ for $l>0$ and some non-degenerate $f$.
\end{enumerate}
\end{lemma}

\begin{proof} Without loss of generality, we may assume that $\dim \Delta=\rank M$. 
Set $d=\dim\Delta$. Note that b) is equivalent to 
$\relint(d\Delta)\cap M=\emptyset$.
By applying a translation, we may assume that 
some $v_0\in\Delta^{[0]}$ is the origin. 
Let $v_1,...,v_d$ be the other vertices. They form a basis if and only if
there is no lattice point other than the vertices contained in the 
parallelepiped $P=\{\sum_{i=0}^d \lambda_i v_i\,|\, 0\le\lambda_i\le 1\}$ 
spanned by these vectors. Note that $P\subset d\cdot\Delta$ and
$$P\cap \partial(d\cdot \Delta) = \{\textstyle\sum_{i=1}^d v_i\} \cup 
\{x\in P\,|\, x=\textstyle\sum_{i=1}^d \lambda_i v_i
\hbox{ with some }\lambda_i=0\}.$$
To see the implication b)$\Rightarrow$a) now assume a) doesn't hold, so there is some
lattice point $x=\sum\lambda_i v_i\in P$ which isn't a vertex. 
We may assume $\lambda_i>0$ by adding $v_i$ if necessary. 
Now $x\not\in P\cap \partial(d\cdot \Delta)$ and therefore $x\in \relint(d\Delta)\cap M$.
We get a)$\Rightarrow$b) by repeatedly subtracting $v_i$ for each 
$1\le i\le d$ from an arbitrary $x\in \relint(d\Delta)\cap M$ 
as long as the result $x'$ is still contained in $d\Delta$. We have $x'\in P$ and
$x'$ isn't a vertex of $P$.

By Prop.~\ref{Fermatiso}, c) is equivalent to
$$\Delta^{\core{l}}\cap M=0\hbox{ for each }l>0.$$
By the same argument as before, $\Delta$ is non-standard if and only if
there is $x'=\sum_{i=1}^d \lambda_i v_i\in M$ with $0\le \lambda_i<1$
and some $\lambda_i>0$. For such an $x'$, set $I=\{i\,|\,\lambda_i>0\}$
and let $F$ be the face of $\Delta$ which is 
$$F=\left\{\begin{array}{ll}
\hbox{the convex hull of }\{v_i,i\in I\}&\hbox{if }\sum_{i\in I}\lambda_i\in\NN\\
\hbox{the convex hull of }0\hbox{ and }\{v_i,i\in I\}\ \ &\hbox{otherwise}
 \end{array}\right.$$
Let $l$ be the smallest integer greater or equal to $\sum_{i\in I}\lambda_i$
and $\lambda_0=l-\sum_{i\in I}\lambda_i$. 
We find $x'\in F^{\core{l}}$ and by Lemma~\ref{corefunc}~a)
we have $x'\in \Delta^{\core{l}}$ which proves c)$\Rightarrow$a). 
Using Lemma~\ref{corefunc}~a) again, the converse becomes clear because 
$x=\sum_{i=1}^d \lambda_i v_i\in \relint(F^{\core{l}})$ for some $l>0$ and 
some face $F\subseteq\Delta$ yields an element 
$x+\sum_{i:\lambda_i=0} v_i\in\relint(d\Delta)$, i.e., b)$\Rightarrow$c).
\end{proof}


\subsection{The Koszul complex and log differential forms}
We use the notation and setting from the previous section, i.e., we have a 
toric variety $X$ and a non-degenerate semi-ample Cartier divisor $Z$ with Newton polytope $\Delta$.
Let $D$ here, unlike in the previous section, denote the boundary divisor 
of $X$, i.e., the complement of the big torus. For a normal variety $Y$ 
with an effective Cartier divisor $E$ we denote by 
$\Omega^r_{Y/\k}(\op{log\,}(E))$ the sheaf of differential $r$-forms with at most logarithmic
poles along $E$. In general this doesn't need to be a coherent sheaf. 
In our situation dealing with a toric boundary divisor, however, 
this will be the case.
For an $\shO_Y$-module $\shF$, as usual, we set
$\shF(E)=\shF\otimes_{\shO_Y}\shO_Y(E)$. 
Note that there is a canonical isomorphism
$$\Omega^r_{X/\k}(\op{log\,}(D))=\shO_X\otimes_\ZZ\bigwedge^r M$$
by mapping $u\otimes m_1\wedge...\wedge m_r$ on the right 
to $u\cdot\frac{dz^{m_1}}{z^{m_1}}\wedge...\wedge\frac{dz^{m_r}}{z^{m_r}}$ on the left.

\begin{lemma} \label{logexactseq}
 For each $r$ there is an exact sequence
 $$0\ra \Omega^r_{X/\k}(\op{log\,}(D+Z))(-Z)\ra \Omega^r_{X/\k}(\op{log\,} D) \stackrel{\res}{\lra} 
 \Omega^r_{Z/\k}(\op{log\,}(D\cap Z)) \ra 0$$
\end{lemma}

\begin{proof} 
  The first two non-zero terms are naturally contained in the sheaf of rational forms.
  We check that the first injects in the second. Let $g$ 
  be a function on an open chart which is invertible 
  outside $D\cup Z$. By the local irreducibility of $Z$, 
  we may assume that on that chart $Z$ is irreducible.
  Let $f$ be a local equation of $Z$. Either $g$ is invertible 
  outside $D$ or $g=g'\cdot f^k$ with $g'$ invertible outside $D$.
  Assume the latter, then 
  $$f\cdot \frac{dg}{g}=f\cdot \frac{f^kdg'+kg'f^{k-1}df}{f^kg'}
  =\frac{fdg'+kg'df}{g'}=\frac{fdg'}{g'}+kdf$$ is a form with at most logarithmic
  poles along $D$, so the first non-trivial map is well-defined and injective.
  We also see that $\Omega^\bullet_{X/\k}(\op{log\,}(D+Z))(-Z)$ is the
  subalgebra of
  $\Omega^\bullet_{X/\k}(\op{log\,} D)$ locally generated by $f$ and $df$.
  This is the kernel of the surjection to
  $\Omega^\bullet_{Z/\k}(\op{log\,}(D\cap Z))$
  and we are done.
\end{proof}

\begin{remark}
\begin{itemize}
\item[a)] If $E$ is a normal crossings divisor on a 
complex manifold $Y$, 
then \newline $H^k(Y\backslash E,\CC)=\HH^k(Y,\Omega^\bullet_{Y/\CC}(\llog(E)))$
and $H_c^k(Y\backslash E,\CC)=\HH^k(Y,\Omega^\bullet_{Y/\CC}(\llog(E))(-E))$
where $H_c$ is cohomology with compact support. We have, so to say, combined 
these two concepts.
\item[b)] One can show that if $Z$ is non-degenerate, 
we have a toroidal pair $(X,D+Z)$ in the sense of \cite{danilov1}. 
For such a toroidal pair $(Y,E)$,
Danilov defines $\Omega^\bullet_{(Y,E)}$ as the kernel of
$\Omega^\bullet_{Y/\k}\ra\bigoplus_{E'}\Omega^\bullet_{E'/\k}$ 
where the sum ranges over the irreducible components of $E$. 
The author calls these modules 
\emph{differential forms with logarithmic zeros}. 
For $\k=\CC$, he then shows the 
degeneration of the hypercohomology spectral sequence 
of $\Omega^\bullet_{(Y,E)}$ at the $E_1$ term
(see \cite{danilov2}). 
In \cite{batyrev1}, Batyrev uses 
the Poincar\'e dual $\Omega^\bullet_{Y}(\llog(E))$. 
Our constructions reside somewhere in between these two
and are determined to be used as local contributions 
to the Hodge data of toric Calabi-Yau degenerations.
\end{itemize}
\end{remark}

For each rationally generated subspace $T\subseteq M\otimes_\ZZ \k$, in other words, for each
saturated $\ZZ$-submodule $T\cap M$ of $M$, we can view $\shO_X\otimes_\k \bigwedge^r T$ as a 
free submodule of $\Omega^r_{X/\k}(\op{log}\,D)$ which we wish to call 
$T\cap \Omega^r_{X/\k}(\op{log}\,D)$. As long as we make sure that locally
$df\in T\cap \Omega^\bullet_{X/\k}(\op{log}\,D)$, i.e., $T_\Delta\subseteq T$, we obtain by
the previous lemma an induced exact sequence
$$
0\ra T\cap\Omega^r_{X/\k}(\op{log\,}(D+Z))(-Z)
\ra T\cap\Omega^r_{X/\k}(\op{log\,} D) 
\stackrel{\res}{\lra} T\cap\Omega^r_{Z/\k}(\op{log\,}(D\cap Z)) \ra 0.
$$
This sequence will be of most interest to us in the case where $T=T_\Delta$. 
Let $h:\hat{T}_\Delta\ra\kk$
denote the canonical projection as before. Given $T_\Delta\hra W$, we also define it for $\hat{W}$.
There is an exact sequence
$$0\ra W \ra \hat{W} \stackrel{h}{\lra} \kk \ra 0.$$

\begin{definition} Given an inclusion $T_\Delta\hra W$, we define a map
 $$\pi^r: \shK^r(W,Z,-r-1)\ra \shO_X \otimes \bigwedge^r W$$
 by the composition 
 $\shK^r(W,Z,-r-1)
 \stackrel{d_Z}{\lra} \shK^{r+1}(W,Z,-(r+1))
 =\shO_X\otimes_\k\bigwedge^{r+1} \hat{W}
 \stackrel{\id\otimes\iota(h)}{\lra} \shO_X\otimes_\ZZ\bigwedge^r W$
 where $\iota(h)$ means contraction by $h$. 
 We mostly write $\pi$ for $\pi^r$. That the image lies in
 $\bigwedge^r W$ inside $\bigwedge^r \hat{W}$ can be seen by
 applying $\iota(h)$ to the exact sequence
 $$0\ra \bigwedge^{r+1} W\ra \bigwedge^{r+1} \hat{W}
 \ra \hat{W}/W\otimes \bigwedge^{r} W\ra 0.$$
\end{definition}

\begin{lemma} \label{piexplicit}
 Let $f$ be an equation of $Z$. Using Lemma~\ref{globalsec},
 the map $\pi^r$ is explicitly given by
 $$u\mapsto u\cdot\sum_{m\in\Delta\cap M} f_m \frac1f z^{m} 
 \qquad \hbox{for }r=0$$
 $$u\otimes (\hat{v}\wedge\alpha) \mapsto 
 u\cdot\sum_{m\in\Delta\cap M} f_m \frac1f z^{m}\otimes((v-m)\wedge\alpha)  
 \qquad \hbox{for }r>0$$
 where $\alpha\in \bigwedge^{r-1} W$ and 
 $\hat{v}=(v,1)$ for a vertex $v\in\Delta$.
\end{lemma} 

\begin{proof}
  The map $d_Z$ was already given explicitly in Lemma~\ref{dZexplicit}. 
  This lemma directly follows by composing with $\iota(h)$.
\end{proof}

For the next theorem we need the following elementary lemma

\begin{lemma} \label{lemmawedgesubspace}
\begin{itemize}
\item[a)]
 For a vector space $V$, a subspace $T$ of codimension one, $v\in V\backslash T$,
 $\alpha\in\bigwedge^{r-1}V$ and $v\wedge\alpha\in\bigwedge^r T$, we have
 $v\wedge\alpha=0$
\item[b)] For a vector space $V$, a subspace $T$ of codimension $k$, $v_1,...,v_k\in V$ which are linearly
 independent modulo $T$ and $\alpha_1,...,\alpha_k\in\bigwedge^{r-1}V$ we have
 $\sum_{i=1}^k v_i\wedge\alpha_i\in\bigwedge^r T\ \Rightarrow \sum_{i=1}^k v_i\wedge\alpha_i=0$
\end{itemize}
\end{lemma}

\begin{proof} It is clear that b) implies a), so we only need to prove b).
  Because $v_1,...,v_k$ descend to a basis of $V/T$, 
  they induce a splitting $V\cong T\oplus V/T$. This in turn induces an isomorphism of
  graded algebras
 $\bigwedge^\bullet V \cong \bigoplus_{j\ge 0} \bigwedge^{\bullet-j} T \otimes \bigwedge^{j} V/T.$
 By construction $v_i\in \bigwedge^{0}T\otimes\bigwedge^1 V/T$ so we have
 $\sum_{i=1}^k v_i\wedge\alpha_i\ \in\ 
 \bigoplus_{j\ge 1} \bigwedge^{r-j} T \otimes \bigwedge^{j} V/T.$
 On the other hand $\bigwedge^r T=\bigoplus_{j=0} \bigwedge^{r-j} T \otimes \bigwedge^{j} V/T$. This implies
 the assertion.
\end{proof}

The following theorem is going to tell us that the Koszul complex resolves the sheaves
$T_\Delta\cap\Omega^r_{X/\k}(\op{log\,}(D+Z))(-Z)$. Later, we 
will be using this resolution to compute the cohomology of these 
specific log differential forms which are the building blocks
of the log Hodge sheaves on a h.t. toric log Calabi Yau space as 
we will to see in the next chapter.

\begin{theorem} \label{longexactseq}
 For each $r$ there is an exact sequence
 $$0\lra\shK^0(T_\Delta,Z,-(r+1)) 
 \lra\shK^1(T_\Delta,Z,-(r+1)) \lra\qquad$$
 $$\qquad ...\lra \shK^r(T_\Delta,Z,-(r+1))
 \stackrel{\pi}{\lra} T_\Delta\cap \Omega^r_{X/\k}(\op{log\,} D) 
 \stackrel{\res}{\lra} T_\Delta\cap \Omega^r_{Z/\k}(\op{log\,}(D\cap Z))
 \lra 0.$$
\end{theorem}

\begin{proof}
 The exactness of the first part follows by the overall non-degeneracy
 hypothesis, \newline Lemma~\ref{bpf} 
 and Lemma~\ref{koszulexakt}. Moreover, the sequence is exact at $\shK^r(T_\Delta,Z,-(r+1))$
 if and only if $\iota(h)$ is injective on the image of $d_Z:\shK^r(T_\Delta,Z,-(r+1))\ra\shK^{r+1}(T_\Delta,Z,-(r+1))$.
 Note that $\shK^{r+1}(T_\Delta,Z,-(r+1))=\shO_X\otimes\bigwedge^r \hat{T}_\Delta$. We have 
 $\ker\iota(h)=\shO_X\otimes\bigwedge^r T_\Delta$. We may consider the map $d_Z$ over the field of
 fractions $\Quot(\shO_X)$ in which $\shO_X(lZ)$ canonically embeds for each $l$. The advantage is that via
 Lemma~\ref{dZexplicit} the map $d_Z$ is then given by wedging with an 
 element $dF$ of $\Quot(\shO_X)\otimes \hat{T}_\Delta$ 
 which is not in $\Quot(\shO_X)\otimes T_\Delta$.
 Applying Lemma~\ref{lemmawedgesubspace},~a) yields
 $$dF\wedge\alpha\in \shO_X\otimes\bigwedge^r T_\Delta\ \IFF\ \alpha=0.$$
 This finishes showing the exactness at $\shK^r(T_\Delta,Z,-(r+1))$. 
 By Lemma~\ref{logexactseq}, the only thing left to prove is that
 $$\im(\pi^r)=T_\Delta\cap \Omega^r_{X/\k}(\op{log\,}(D+Z))(-Z).$$
 We take a look at the image of $\pi^r$ in a toric chart. Let $\sigma$ be a maximal cone in the 
 fan of $X$ and $U=\Spec \k[P_\sigma]$ the corresponding chart where $P_\sigma=\sigma^\dual\cap M$.
 There is a unique vertex $v\in\Delta$ such that $\Delta-v\subset\sigma^\dual$
 (otherwise $\sigma^\dual$ would have to contain a straight line which is
 impossible for a full-dimensional $\sigma$). Choose the
 standard local trivialization $\shO_X(-Z)|_U\cong\shO_U$ such that the 
 section $\frac1f z^{m}$ is given by
 $z^{m-v}$. Let us first consider the case $r=0$. By Lemma~\ref{piexplicit},
 the map $\pi^0|_U$ becomes multiplication with 
 $\sum_{m\in\Delta\cap M} f_m z^{m-v}$ which is just an
 equation of $Z$ on $U$. Thus, the case $r=0$ reduces to the standard sequence
 $$ 0 \ra \shO_X(-Z) \ra \shO_X \ra \shO_Z\ra 0 $$
 which is exact. 
 Now assume $r>0$. By Lemma~\ref{piexplicit}, the map $\pi^r$ becomes
 $$\begin{array}{rcl}
 \pi^r|_U:\shO_U\otimes_\k\bigwedge^r \hat{T}_\Delta  &\ra & \shO_U\otimes_\k \bigwedge^r T_\Delta\\
 u\otimes \hat{w}\wedge\alpha &\mapsto& u\cdot\sum_{m\in\Delta\cap M} f_m z^{m-v}\otimes(w-m)\wedge\alpha.
 \end{array}$$
 In other words, the image of $\pi^r|_U$ is the degree $r$ part of the exterior algebra ideal in
 $\shO_U\otimes\bigwedge^\bullet T_\Delta$ generated by
 $$L_w:= \sum_{m\in\Delta\cap M} f_m z^{m-v}\otimes(w-m) \qquad \hbox{for }w\in\Delta^{[0]}.$$
 We set $f_U:=\sum_{m\in\Delta\cap M} f_m z^{m-v}$. As we already mentioned, this defines $Z\cap U$. 
 We have
 $$\begin{array}{rcl}
 \qquad\qquad\qquad\qquad L_{v}  &=&  \sum_{m\in\Delta\cap M} f_m z^{m-v}\otimes (m-v)\\
 	&=&  \sum_{m\in\Delta\cap M} f_m z^{m-v}\frac{dz^{m-v}}{z^{m-v}}\qquad\qquad  \hbox{as an element in }\Omega_{X/\k}(\op{log\,} D)\\
 	&=&  df_U,\\
 \end{array}$$
 $$\begin{array}{rcl}
 L_w-L_{v}  &=&  \sum_{m\in\Delta\cap M} f_m z^{m-v}\otimes ((m-w)-(m-v))\\
            &=&  \sum_{m\in\Delta\cap M} f_m z^{m-v}\otimes (v-w)\\
 	    &=&  f_U\otimes (v-w).\\
 \end{array}$$
 The set $\{ v-w\,|\, w\in\Delta^{[0]} \}$ spans $T_\Delta$. 
 With the same argument as at the end of the proof of Lemma~\ref{logexactseq} where we
 described $T_\Delta\cap \Omega^r_{X/\k}(\op{log\,}(D+Z))(-Z)$ as the exterior algebra generated by 
 $f_U$ and $df_U$,
 we see that $\im(\pi^r)=T_\Delta\cap \Omega^r_{X/\k}(\op{log\,}(D+Z))(-Z)$,
 and we are done.
\end{proof}

\begin{corollary} \label{corlongexactseq}
 For each $r\in \NN_{\ge 0}$, there is an exact sequence
 $$0\lra\shK^0(W,Z,-(r+1)) 
 \lra\shK^1(W,Z,-(r+1)) \lra\qquad$$
 $$\qquad ...\lra \shK^r(W,Z,-(r+1))
 \stackrel{\pi}{\lra} \shO_X \otimes \bigwedge^r W 
 \stackrel{\res}{\lra}  \coker(\pi)
 \lra 0.$$
\end{corollary}

\begin{proof}
  This follows from Lemma~\ref{decompose} and Theorem~\ref{longexactseq}.
\end{proof}

We wish to give the generalized $\Omega^r_{X/\k}(\op{log\,}(D+Z))(-Z)$ a functional name.

\begin{definition} 
 Given an inclusion $T_\Delta\hra W$ we define
 $\shC^r(W,Z)$ by the exact sequence
 $$\shK^{r-1}(W,Z,-(r+1)) \stackrel{d_Z}{\lra} \shK^r(W,Z,-(r+1)) \ra \shC^r(W,Z)\ra 0.$$
\end{definition}

By Corollary~\ref{corlongexactseq}, we have
$$\shC^r(W,Z)=\im(\pi^r).$$

\begin{proposition} \label{naturalres}
 For each $r$, there is an acyclic resolution
 $$0\ra \shC^r(W,Z) \ra \shK^{r+1}(W,Z,-(r+1)) \ra ...\ra \shK^{\dim W}(W,Z,-(r+1))\ra 0$$
 which is functorial in $W$.
\end{proposition}

\begin{proof}
  The exactness, once again, follows from the overall non-degeneracy hypothesis,   \newline
  Lemma~\ref{bpf} and Lemma~\ref{koszulexakt}. Acyclicity follows from
  Lemma~\ref{cohotoriclb}.  
  For $T_\Delta\hra W'$, a $T_\Delta$-map $W\ra W'$ clearly induces a map 
  $\bigwedge^l W \ra \bigwedge^l W'$ for each $l$ and a map of complexes
  $\shK^\bullet(W,Z,-(r+1))\ra \shK^\bullet(W',Z,-(r+1))$.
\end{proof}

\begin{corollary} \label{cohoC} 
  If $Z\neq\emptyset$, we have for all $r,p$
  $$H^p(X,\shC^r(W,Z))=HK^{r+p+1}(W,Z,-(r+1))=
  R(Z)_{p}\otimes_\k \langle\bigwedge^\top T_{\Delta} \rangle_{r+p},$$  
  in particular, $H^p(X,\shC^r(T_\Delta,Z))=0$ for $p+r\neq \dim\Delta$.
\end{corollary}

\begin{proof}
  For $p=0$ this uses 
  $\Gamma(\shK^{r}(W,Z,-(r+1)))=\Gamma(\shO_X(-Z)\otimes_\k \bigwedge^{r}\hat{W})=0,$
  and it follows from Prop.~\ref{naturalres} and Prop.~\ref{cohoW}.
\end{proof}

\begin{remark} \label{deRham}
We may naturally extend the de Rham differential on log forms by the unique
derivation extending the following map on monomial functions
$$\begin{array}{rrcl}
d:&\shO_X\otimes_\k\bigwedge^r W&\ra& \shO_X\otimes_\k\bigwedge^{r+1} W\\
&z^m\otimes\alpha &\mapsto& z^m\otimes m\wedge\alpha.
\end{array}$$
This map is compatible with the inclusions of $\shC^r(W,Z)$ by 
Theorem~\ref{longexactseq}. In particular $d$ is trivial on 
``constant differential forms'', i.e., those where $m=0$.
\end{remark}

\begin{remark}[Moving $\Delta$] \label{moveDelta}
 Recall that we have fixed a translation representative of the 
 lattice polytope $\Delta$ which is the Newton polytope of the non-degenerate divisor $Z$.
 We want to discuss now what happens if we move $\Delta$.
 Set $T_\RR=\span{\RR}\{v-w\,|\,v,w\in\Delta^{[0]}\}$ 
 Assume that we have two embeddings
 $p_1:\Delta\hra T_\RR,\ p_2:\Delta\hra T_\RR$ which differ by a translation
 by some integral vector $v\in T_\RR$, i.e.,
 $p_2=p_1+v$.
 For any $T_\Delta\hra W$, this induces an automorphism
 $$S_v:\hat W\ra \hat W,\qquad w\mapsto w+\iota(h)(w)\cdot v$$
 which maps the cone over $p_1(\Delta)$ to the cone over $p_2(\Delta)$. 
 Moreover, it induces an isomorphism of complexes
 $$\shK_{p_1}^\bullet(W,Z,r)\stackrel{\bigwedge^\bullet S_v}{\lra}\shK_{p_2}^\bullet(W,Z,r)$$
 where we use the indices $p_1$, $p_2$ to denote by which translation representative the
 complex is constructed. To see this, consider Lemma~\ref{dZexplicit}. 
 This isomorphism coincides with taking the detour via $V^*$, i.e., 
 $S_v=\partial^*_{Z,p_2}\circ (\partial^*_{Z,p_1})^{-1}$, see the proof of
 Lemma~\ref{dZexplicit}. 
 An important point is that the map $\pi^l$ commutes with $S_v$ which follows from
 Lemma~\ref{piexplicit}. The cohomology of $\shK^\bullet$ is invariant under $S_v$ because 
 $\bigwedge^\top \hat{T}_\Delta$ is. Having said all this, 
 whenever $\shK^\bullet$ comes up in this paper, we keep the position of the
 polytope arbitrary and just need to make sure that all additional constructions commute with 
 translations of $\Delta$.
\end{remark}

\begin{lemma} \label{lempullbackres}
 Let $\Delta'$ be a face of $\Delta$ and $X'$ the corresponding toric subvariety of $X$.
 We assume that $Z$ is non-degenerate and so $Z'=Z\cap X'$ is also non-degenerate.
 There is a canonical isomorphism of sequences on $X'$
 {\scriptsize
 $$\begin{CD}
 ...@>>> \shK^r(T_{\Delta},Z,-(r+1))|_{X'}
 @>>{\pi}> T_{\Delta}\cap \Omega^r_{X/\k}(\op{log\,} D)|_{X'}
 @>>{\res}> T_{\Delta}\cap \Omega^r_{Z/\k}(\op{log\,} (D\cap Z))|_{X'} &\ra 0\\
 &&@VVV @VVV @VVV\\
 ...@>>> \shK^r(T_{\Delta},Z',-(r+1))
 @>>{\pi}> \shO_{X'}\otimes \bigwedge^r T_{\Delta}
 @>>{\res}> \coker(\pi) &\ra 0\\
 \end{CD}$$
 }
 where the bottom sequence comes from $T_{\Delta'}\hra T_{\Delta}$.
 More generally, given $T_{\Delta'}\hra W$, we have an analogous isomorphism 
 $\shK^\bullet(W,Z,-(r+1))|_{X'}\ra \shK^\bullet(W,Z',-(r+1))$.
\end{lemma}

\begin{proof} By Remark~\ref{moveDelta} we can move $\Delta$ in the unique position
  such that $\Delta'$ embeds in it as the face corresponding to the stratum $X'$ in $X$. 
  Let $f$ be an equation of $Z$ then
  $f'=f|_{X'}$ is an equation of $Z'$ on $X'$.
  The vertical isomorphisms are induced by $\shO_X(lZ)|_{X'}=\shO_{X'}(lZ')$ 
  for varying $l$. The lemma becomes clear after checking the behaviour of the 
  differential after restriction to $X'$.
  The differential in the upper row is
  $$u\otimes\alpha \mapsto 
  u\cdot\sum_{m\in\Delta\cap M} f_{m} (f^{-1}z^{m})|_{X'}\otimes (m,1)\wedge \alpha.$$
  Since $f^{-1}z^{m}|_{X'}=0$ for $m\not\in\Delta'$ and 
  $f^{-1}z^{m}|_{X'}=(f')^{-1}z^{m}$ for $m\in\Delta'$,
  this is just
  $$u\otimes\alpha \mapsto 
  u\cdot\sum_{m\in\Delta'\cap M} f_{m} (f')^{-1}z^{m}\otimes (m,1)\wedge \alpha.$$
  which is the differential in the lower row.
\end{proof}


\section{The cohomology of $\SC(\Omega^r)$} \label{section_main}
\subsection{An acyclic resolution on a stratum}
\label{section_res_single_strat}
In this section we wish to apply the constructions of chapter \ref{section_koszkoho}. The key
object will be $\shC^r(...)$ for various parameters.
The overall hypothesis is now that $X$ is a h.t. toric log CY space. 
This implies that, for each $\tau$, there is only one
$Z_\tau,\check\Delta_\tau,\Delta_\tau,\Upomega_\tau, R_\tau$, respectively.
We will implicitly use the following lemma a lot.

\vbox{
\begin{lemma}  \label{equivalencesfore}
  Let $e:\tau_1\ra\tau_2$ in $\P$. The following are equivalent.
 \begin{center}
 \begin{tabbing}
 \= i) $W_e\cap \Delta \neq\emptyset$\qquad \qquad
 \= ii) $e\in\Delta$\\
 \> iii) $Z_{\tau_1}\cap X_{\tau_2}\neq\emptyset$ 
 \> iv) $Z_{\tau_1}\cap X_{\tau_2}= Z_{\tau_2}$ \\
 v) There is some $h\in\Upomega_{\tau_2}$ which factors through $e$. \\
 vi) There is some $h\in R_{\tau_1}$ which factors through $e$. \\
 vii) $\{\omega\ra\tau_1\stackrel{e}{\ra}\tau_2\ra\rho\,|\,\omega\in\P^{[1]},
 \rho\in\P^{[\dim B-1]},\kappa_{\omega\rho}\neq 0\}\neq \emptyset$
 \end{tabbing}
 \end{center}
\end{lemma}
}
\begin{proof} iii) and iv) are equivalent by the h.t.\! property. 
ii)$\Rightarrow$ i) is trivial. The inverse direction is clear if 
$\tau_1\neq\tau_2$ because $e$ is the only edge of nontrivial intersection with $W_e$.
If $\tau_1=\tau_2$, then $e$ is a point which is contained in every edge of 
$\Delta\cap W_e$. 
The equivalence of ii) and vii) follows from the description of $\Delta$ and
v)$\IFF$ vi) $\IFF$ vii) is trivial. 
It remains to show that ii) is equivalent to iii).
If $Z_{\tau_1}=\emptyset$ then the negation of vii) follows. On the other hand,
vii) implies $Z_{\tau_1}\neq \emptyset$ via Lemma~\ref{lem_ZDelta}, 
so we may assume $Z_{\tau_1}\neq\emptyset$.
Because $Z_{\tau_1}\subseteq X_{\tau_1}$ is a divisor not containing 
any toric stratum and $X_{\tau_2}\subseteq X_{\tau_1}$ is a stratum, we have
$$Z_{\tau_1}\cap X_{\tau_2}\neq \emptyset \IFF 
Z_{\tau_1}\cap \Int(X_{\tau_2})\neq \emptyset.$$
By the assumption $Z_{\tau_1}\neq \emptyset$ there exists 
$\P^{[1]}\ni\omega_0 \stackrel{h}{\ra}\tau_1$ 
which represents an edge contained in $\Delta$.
By Lemma~\ref{lem_ZDelta}, 
$$e\circ h\in\Delta\ \IFF\ Z_{\omega_0}\cap X_{\tau_2}\neq\emptyset
 \ \IFF\ Z_{\omega_0}\cap X_{\tau_1}\cap X_{\tau_2}\neq\emptyset
 \ \IFF\ Z_{\tau_1}\cap X_{\tau_2}\neq\emptyset.$$

We therefore need to show that $e\in\Delta\IFF e\circ h\in\Delta$ but
this is just ii)$\IFF$v) which we have shown already.
\end{proof}

\begin{remark}
 Note that possibly $Z_{\tau_1},Z_{\tau_2}\neq\emptyset$ but 
 $Z_{\tau_1}\cap X_{\tau_2}=\emptyset$. This happens if we have a diagram
 $$
 \xymatrix{    
   \omega_1\ar[r] & \tau_1\ar^{e}[rd] \ar[rr] && \rho_1\\
   \omega_2\ar[rr] && \tau_2 \ar[r] & \rho_2  
 } $$
 where none of the maps factors through any other map and 
 $\kappa_{\omega_1\rho_1},\kappa_{\omega_2\rho_2}>0$. In other words, the points
 $\id_{\tau_1}$ and $\id_{\tau_2}$ might be contained in $\Delta$ 
 but not the connecting edge.
\end{remark}

\begin{definition} Let $M$ be a lattice, 
$\sigma$ some lattice polytope in $M\otimes_\ZZ\RR$, $\Sigma$ its normal fan.
Let $\psi$ be a piecewise linear function with respect to $\Sigma$ 
coming from another lattice polytope
$\Delta_\sigma\subseteq M\otimes_\ZZ\RR$ in the sense of Lemma~\ref{polyunique}.
Let $\tau\subseteq\sigma$ be a face and $\check\tau\in\Sigma$ the corresponding cone. 
We can restrict $\psi$ to the star of $\check\tau$ [see \cite{fulton}, p.52] and
obtain a piecewise linear function on the normal fan of $\tau$ which comes from a
face of $\Delta_\sigma$. We denote this face by
$$\Delta_\sigma\cap\tau$$
because it is the intersection of $\Delta_\sigma$ 
with a translate of the tangent space of $\tau$.
\end{definition}

\begin{definition} For any $e:\tau_1\ra\tau_2$, we define 
$Z_e=Z_{\tau_1}\cap X_{\tau_2}$ and polytopes
$$\begin{array}{l}
\check\Delta_{e}=
\check\Delta_{\tau_1}\cap \tau_2
=\Newton (Z_e)\qquad\qquad\hbox{and}\\
\Delta_{e}=\Delta_{\tau_2}\cap \tau_1.
\end{array}$$
\end{definition}

It is not hard to see that we get a similar statement as in 
Lemma~\ref{inoutkappa}. There are surjections
$$
\xymatrix{    
\{\hbox{ edges of }\check\Delta_e\}&
\{\omega\ra\tau_1\stackrel{e}{\ra}\tau_2\ra\rho \,|\, \kappa_{\omega\rho}\neq 0\}
\ar@{->>}[l]\ar@{->>}[r]
&\{\hbox{ edges of }\Delta_e\}.
} $$
In particular, $\Delta_{e}=\{0\}$ if and only if $e\not\in\Delta$.
The following lemma is a direct consequence of Lemma~\ref{equivalencesfore} and
the definition of h.t..
\begin{lemma} \label{Ze}
For $e:\tau_1\ra\tau_2$, we have
$$Z_e,\Delta_e,\check\Delta_e=
\left\{
\begin{array}{ll}
Z_{\tau_2},\Delta_{\tau_1},\check\Delta_{\tau_2}&\hbox{ if }e\in\Delta\\
\emptyset,\{0\},\{0\}&\hbox{ if }e\not\in\Delta.
\end{array}
\right.$$
\end{lemma} 

By the naturality of Prop.~\ref{naturalres}, we have, for each $r$ and 
$\P^{[0]}\ni v\stackrel{g}{\ra}\tau_1\stackrel{e}{\ra}\tau_2$, a commutative diagram
of $\shO_{X_{\tau_2}}$-modules
\begin{equation}\label{squarediagram}
\begin{CD}
\shC^r(\Delta_e^\perp,Z_e) @>{\pi}>> \shO_{X_{\tau_2}}\otimes_\k\bigwedge^r \Delta_e^\perp\\
@VVV @VVV\\
\shC^r(\check\Lambda_{v,\k},Z_e)@>{\pi}>> \shO_{X_{\tau_2}}\otimes_\k\bigwedge^r\check\Lambda_{v,\k}\\
\end{CD}
\end{equation}
where the vertical maps are injections. Using Lemma~\ref{decompose} to decompose the $\shC^r$'s into
$\shC^{r-s}(T_{\check\Delta_e},Z_e)$'s, one finds that the diagram is cartesian. We thus have

\begin{lemma} \label{leftexactseq}
The sequence
$$0\ra \shC^r(\Delta_e^\perp,Z_e) 
\ra \shC^r(\check\Lambda_{v,\k},Z_e)\oplus \shO_{X_{\tau_2}}\otimes_\k\bigwedge^r \Delta_e^\perp
\ra \shO_{X_{\tau_2}}\otimes_\k\bigwedge^r\check\Lambda_{v,\k}$$
is exact where the first non-trivial map is 
$(\shC^r(\Delta_e^\perp\hra \check\Lambda_v),-\pi)$ and 
the second is $\pi\, +\, \id\otimes \bigwedge^r (\Delta_e^\perp\hra \check\Lambda_v)$.
\end{lemma}

Adapting Prop.~\ref{delta0kernel} to the h.t. situation, we get
$$(F_{s}(e)^*\Omega^r_{\tau_1})/\Tors=
\bigcap_{\atop{w\not=\Ver(v)}{w\in\Delta_{\tau_1}}} 
\ker\bigg( F_{s}(e\circ g)^*\Omega^r_v
\stackrel{\iota(\partial_{w-\Ver(v)})|_{(Z_e)^{\dagger}}}{\lra}
 \Omega^{r-1}_{(Z_e)^{\dagger}/\k^{\dagger}} \bigg).$$

We are going to use the canonical identification 
$\Omega_v^r=\Omega_{X_v}^r(\op{log}\,D_v)=\shO_{X_v}\otimes_\k\bigwedge^r
\check\Lambda_{v,\k}$ 
as given in \cite{grosie2}, Lemma 3.12. 
Pulling back differentials simplifies to restricting functions and thus
$$F_{s}(e\circ g)^*\Omega^r_v=F_{s}(e\circ g)^*\shO_{X_v}\otimes_\k\bigwedge^r
\check\Lambda_{v,\k} =\shO_{X_{\tau_2}}\otimes_\k\bigwedge^r \check\Lambda_{v,\k}.$$

A choice of splitting
$\check\Lambda_{v,\k}\,\cong\, T_{\check\Delta_e}\,\oplus \,
\Delta_{e}^\perp/T_{\check\Delta_e}\, \oplus\, 
\check\Lambda_{v,\k}/\Delta_e^\perp$
induces an isomorphism
$$\bigwedge^r \check\Lambda_{v,\k}\cong 
\bigoplus_{a,b\ge 0} 
\bigwedge^a T_{\check\Delta_{e}}\otimes_\k 
\bigwedge^{r-a-b} \Delta_{e}^\perp/T_{\check\Delta_{e}}\otimes_\k 
\bigwedge^b \check\Lambda_{v,\k}/\Delta_{e}^\perp
.$$

This induces a decomposition
$$F_{s}(e\circ g)^*\Omega^r_v\cong 
\bigoplus_{a,b\ge 0} 
\big(T_{\check\Delta_e}\cap\Omega^{a}_{{X_{\tau_2}}/\k}(\op{log}\, D_{\tau_2})\big) \otimes_\k 
\bigwedge^{r-a-b} \Delta_e^\perp/T_{\check\Delta_e}\otimes_\k 
\bigwedge^b \check\Lambda_{v,\k}/\Delta_e^\perp.
$$ 

\begin{proposition} \label{Omegatau12split}
 Given a choice of splitting as above, the image of the inclusion
 of $(F_{s}(e)^*\Omega^r_{\tau_1})/\Tors$ in $F_{s}(e\circ g)^*\Omega^r_v$ inherits a
 decomposition as
 $$\bigoplus_{a,b\ge 0}\shW_{a,b}$$
 where
 $$\shW_{a,b}= \shO_{X_{\tau_2}}\otimes_\k\bigwedge^a T_{\check\Delta_e} \otimes_\k 
 \bigwedge^{r-a-b} \Delta_e^\perp/T_{\check\Delta_e}\otimes_\k 
 \bigwedge^b \check\Lambda_{v,\k}/\Delta_e^\perp$$
 for $b=0$ and
 $$\shW_{a,b}= \ker(\res^a) \otimes_\k 
 \bigwedge^{r-a-b} \Delta_e^\perp/T_{\check\Delta_e}\otimes_\k 
 \bigwedge^b \check\Lambda_{v,\k}/\Delta_e^\perp$$
 for $b>0$ with
 $\res^a: T_{\check\Delta_e}\cap\Omega^{a}_{{X_{\tau_2}}/\k}(\op{log}\, D_{\tau_2})
 \ \ra \ T_{\check\Delta_e}\cap\Omega^{a}_{{Z_e}/\k}(\op{log}\, (Z_e\cap
 D_{\tau_2})).$
\end{proposition}

\begin{proof} Note that the assertion is trivial 
if $e\not\in\Delta$ because then, by Lemma~\ref{Ze}, 
$Z_e=\emptyset, \check\Delta_e=\{0\},\Delta_e^\perp=\check\Lambda_v$ 
and so only the component with $b=0$ contributes. 
Let us now assume that $e\in\Delta$ which implies
$Z_e=Z_{\tau_2},\ \Delta_e=\Delta_{\tau_1},\ \check\Delta_e=\check\Delta_{\tau_2}$
using Lemma~\ref{Ze}.
We are going to show the existence of such a decomposition first. 
This is a consequence
of Prop.~\ref{delta0kernel} together with the following two observations
\begin{enumerate}
\item For each $w\not=\Ver(v)$, $\iota(\partial_{w-\Ver(v)})$
respects the decomposition in the sense of being the identity on the first two tensor factors 
when written down as
$$
\big(T_{\check\Delta_{\tau_2}}\cap\Omega^{a}_{{X_{\tau_2}}/\k}(\op{log}\, D_{\tau_2})\big) \otimes_\k 
\bigwedge^{r-a-b} \Delta_{\tau_1}^\perp/T_{\check\Delta_{\tau_2}}\otimes_\k 
\bigwedge^b \check\Lambda_{v,\k}/\Delta_{\tau_1}^\perp \qquad$$
$$\qquad\lra
\big(T_{\check\Delta_{\tau_2}}\cap\Omega^{a}_{{X_{\tau_2}}/\k}(\op{log}\, D_{\tau_2})\big) \otimes_\k 
\bigwedge^{r-a-b} \Delta_{\tau_1}^\perp/T_{\check\Delta_{\tau_2}}\otimes_\k 
\bigwedge^{b-1} \check\Lambda_{v,\k}/\Delta_{\tau_1}^\perp.
$$

\item The restriction to $Z_{\tau_2}$ respects the decomposition by being $\res^a$ on the first 
and the identity of the last two tensor factors being written as
$$
\big(T_{\check\Delta_{\tau_2}}\cap\Omega^{a}_{{X_{\tau_2}}/\k}(\op{log}\, D_{\tau_2})\big) \otimes_\k 
\bigwedge^{r-a-b} \Delta_{\tau_1}^\perp/T_{\check\Delta_{\tau_2}}\otimes_\k 
\bigwedge^b \check\Lambda_{v,\k}/\Delta_{\tau_1}^\perp \qquad\qquad$$
$$\qquad\lra
\big(T_{\check\Delta_{\tau_2}}\cap\Omega^{a}_{{Z_{\tau_2}}/\k}
(\op{log}\, (Z_{\tau_2}\cap D_{\tau_2}))\big) \otimes_\k 
\bigwedge^{r-a-b} \Delta_{\tau_1}^\perp/T_{\check\Delta_{\tau_2}}\otimes_\k 
\bigwedge^b \check\Lambda_{v,\k}/\Delta_{\tau_1}^\perp.
$$
\end{enumerate}

Note that
$\{w-\Ver(v)\,|\, w\not=\Ver(v), w\in\Delta^{[0]}_{\tau_1}\}$ generates 
$T_{\Delta_{\tau_1}}$ and therefore
$$\bigcap_{\atop{w\not=\Ver(v)}{w\in\Delta_{\tau_1}}} 
\ker\bigg(\bigwedge^r \check\Lambda_{v,\k} \stackrel{\iota(w-\Ver(v))}{\lra}
\bigwedge^{r-1} \check\Lambda_{v,\k}\bigg) = \bigwedge^r \Delta_{\tau_1}^\perp$$
which implies the assertion for the $b=0$ case.

On the other hand if $\alpha$ is a form in a component of some $a,b$ with $b>0$ 
then there is a $w$ such that $\iota(\partial_{w-v})\alpha\neq 0$. 
For $\alpha$ to be in $\shW_{a,b}$, we must have that $\iota(\partial_{w-v})\alpha$ 
restricts to $0$ under $\res$. This, however, is equivalent to $\alpha$ 
itself restricting to $0$ under $\res$. This finishes the proof.
\end{proof}

The following proposition adds to Lemma~\ref{leftexactseq}.

\begin{proposition} \label{squareseq}
Given $e:\tau_1\ra\tau_2$, there is a split exact sequence
$$0\ra \shC^r(\Delta_e^\perp,Z_e) \ra \shC^r(\check\Lambda_{v,\k},Z_e)\oplus
\shO_{X_{\tau_2}}\otimes\bigwedge^r \Delta_e^\perp \ra F_s(e)^*\Omega_{\tau_1}^r/\Tors \ra
0$$
where the last non-trivial map depends on $e\circ g:v\ra \tau_2$. It induces
$$F_s(e)^*\Omega_{\tau_1}^r/\Tors
\cong \shO_{X_{\tau_2}}\otimes\bigwedge^r \Delta_e^\perp
\oplus
\big(\shC^r(\check\Lambda_{v,\k},Z_e)/\shC^r(\Delta_e^\perp,Z_e)\big).$$
\end{proposition}

\begin{proof} 
From Lemma~\ref{leftexactseq}, we know that the beginning is exact and just 
need to care about the last term.
As in the proof of Prop.~\ref{Omegatau12split} 
the assertion is trivial for $e\not\in\Delta$. In the other case, we have
$Z_e=Z_{\tau_2},\ \Delta_e=\Delta_{\tau_1},\ \check\Delta_e=\check\Delta_{\tau_2}.$
Choose a splitting $\check\Lambda_{v,\k}\,\cong\, T_{\check\Delta_{\tau_2}}\,\oplus \,
\Delta_{\tau_1}^\perp/T_{\check\Delta_{\tau_2}}\, \oplus\, 
\check\Lambda_{v,\k}/\Delta_{\tau_1}^\perp$.
We are going to use Prop.~\ref{Omegatau12split}. Given its notation, all we need to show is that 
$$\im(\pi\, +\, \id\otimes \bigwedge^r F(g)^*)=\bigoplus_{a,b\ge 0} \shW_{a,b}.$$
By Lemma~\ref{decompose}, the entire sequence splits up in $a,b$-components. 
For the components with $b=0$ the assertion is obvious because 
$$\bigoplus_{\atop{a\ge0}{b=0}} \shW_{a,b} \cong \shO_{X_{\tau_2}}\otimes_\k\bigwedge^r\Delta_{\tau_1}^\perp$$
which clearly coincides with the image. 
For $b>0$ we have by Prop.~\ref{Omegatau12split}
$$\shW_{a,b}=\ker(\res) \otimes_\k 
\bigwedge^{r-a-b} \Delta_{\tau_1}^\perp/T_{\check\Delta_{\tau_2}}\otimes_\k 
\bigwedge^b \check\Lambda_{v,\k}/\Delta_{\tau_1}^\perp$$
and by Theorem~\ref{longexactseq}, 
$\ker(\res^a)=\im(\pi^a)=\shC^a(T_{\check\Delta_{\tau_2}},Z_{\tau_2})$. Now, the
assertion becomes clear by writing down the $a,b$-decomposition of
$\shC^r(\check\Lambda_{v,\k},Z_{\tau_2})$ which is
$$\shC^r(\check\Lambda_{v,\k},Z_{\tau_2})\cong \shC^a(T_{\check\Delta_{\tau_2}},Z_{\tau_2}) \otimes_\k 
\bigwedge^{r-a-b} \Delta_{\tau_1}^\perp/T_{\check\Delta_{\tau_2}}\otimes_\k 
\bigwedge^b \check\Lambda_{v,\k}/\Delta_{\tau_1}^\perp.$$
\end{proof}

By Prop.~\ref{naturalres}, the exact sequence of the previous proposition 
yields an exact sequence of complexes where the right column is defined 
as the cokernel sequence.

\newlength{\dumlength}
\settowidth{\dumlength}{$\oplus \shO_{X_{\tau_2}}\otimes\bigwedge^r\Delta^\perp_{e}$}

{\scriptsize
$$
\begin{array}{ccccccccc}
\vdots && \vdots &&&& \vdots\\
\uparrow &&\uparrow&&&&\uparrow\\
\shK^{r+2}(\Delta^\perp_e,Z_e,-(r+1))&
\hra&\shK^{r+2}(\check\Lambda_{v,\k},Z_e,-(r+1))&&&
\sra& Q^1(F_s(e)^*\Omega_{\tau_1}^r,g)
\\
\uparrow &&\uparrow&&&&\uparrow\\
\shK^{r+1}(\Delta^\perp_e,Z_e,-(r+1))& 
\hra&\shK^{r+1}(\check\Lambda_{v,\k},Z_e,-(r+1))&\oplus&\shO_{X_{\tau_2}}\otimes\bigwedge^r\Delta^\perp_e&
\sra& Q^0(F_s(e)^*\Omega_{\tau_1}^r,g)
\\
\uparrow &&\uparrow&&\rotatebox{90}{\Large=} &&\uparrow\\
 \shC^r(\Delta^\perp_e,Z_e)& 
\hra&\shC^r(\check\Lambda_{v,\k},Z_e)&\oplus&\shO_{X_{\tau_2}}\otimes\bigwedge^r\Delta^\perp_e&
\sra& F_s(e)^*\Omega_{\tau_1}^r/\Tors 
\\
\uparrow &&\uparrow&&&&\uparrow\\
0 && 0 &&&&0\\
\end{array}
$$
}

\begin{proposition} \label{resQ}
We have an acyclic resolution
$$0\ra F_s(e)^*\Omega_{\tau_1}^r/\Tors \ra Q^\bullet(F_s(e)^*\Omega_{\tau_1}^r,g)$$
\end{proposition}

\begin{proof} This directly follows from Prop.~\ref{naturalres} and Lemma~\ref{decompose}.
\end{proof}


\subsection{Independence of the vertex}
\label{sec_indep_vertex}
In this section we wish to show that the previously built resolution
$Q^\bullet(F_s(e)^*\Omega_{\tau_1}^r,g)$ doesn't depend on $g$ in 
the sense that for another $g'$ there is a
natural commutative diagram
\begin{center}
\begin{picture}(0,0)%
\includegraphics{change.pstex}%
\end{picture}%
\setlength{\unitlength}{4144sp}%
\begingroup\makeatletter\ifx\SetFigFontNFSS\undefined%
\gdef\SetFigFontNFSS#1#2#3#4#5{%
  \reset@font\fontsize{#1}{#2pt}%
  \fontfamily{#3}\fontseries{#4}\fontshape{#5}%
  \selectfont}%
\fi\endgroup%
\begin{picture}(6150,858)(346,-670)
\put(361,-241){\makebox(0,0)[lb]{\smash{{\SetFigFontNFSS{12}{14.4}{\familydefault}{\mddefault}{\updefault}{\color[rgb]{0,0,0}$F_s(e)^*\Omega_{\tau_1}^r/\Tors$}%
}}}}
\put(2251,-601){\makebox(0,0)[lb]{\smash{{\SetFigFontNFSS{12}{14.4}{\familydefault}{\mddefault}{\updefault}{\color[rgb]{0,0,0}$Q^0(F_s(e)^*\Omega_{\tau_1}^r,g')$}%
}}}}
\put(2251, 29){\makebox(0,0)[lb]{\smash{{\SetFigFontNFSS{12}{14.4}{\familydefault}{\mddefault}{\updefault}{\color[rgb]{0,0,0}$Q^0(F_s(e)^*\Omega_{\tau_1}^r,g)$}%
}}}}
\put(4366,-601){\makebox(0,0)[lb]{\smash{{\SetFigFontNFSS{12}{14.4}{\familydefault}{\mddefault}{\updefault}{\color[rgb]{0,0,0}$Q^0(F_s(e)^*\Omega_{\tau_1}^r,g')$}%
}}}}
\put(4366, 29){\makebox(0,0)[lb]{\smash{{\SetFigFontNFSS{12}{14.4}{\familydefault}{\mddefault}{\updefault}{\color[rgb]{0,0,0}$Q^1(F_s(e)^*\Omega_{\tau_1}^r,g)$}%
}}}}
\put(6481,-601){\makebox(0,0)[lb]{\smash{{\SetFigFontNFSS{12}{14.4}{\familydefault}{\mddefault}{\updefault}{\color[rgb]{0,0,0}$\dots$}%
}}}}
\put(6481, 29){\makebox(0,0)[lb]{\smash{{\SetFigFontNFSS{12}{14.4}{\familydefault}{\mddefault}{\updefault}{\color[rgb]{0,0,0}$\dots$}%
}}}}
\end{picture}%

\end{center}

Before we go on, recall from Lemmma~\ref{vertex} that
$g:v\ra\tau_1$ induces a vertex $\Ver(g)\in\Delta_{\tau_1}^{[0]}$. 
We will sometimes also denote it by $\Ver(v)$.
Since $\Delta_e\in\{\Delta_{\tau_1},\{0\}\}$, we may also understand this as
$\Ver(g)\in\Delta_e^{[0]}$. There is a dual version as well.
The data $h:\tau\ra\sigma\in\P^{[\dim B]}$ determines a maximal cone
$K_{\sigma}$ in the fan $\Sigma_\tau$, see [\cite{grosie1}, Def.~1.35], on which $\check\Delta_\tau$
defines a piecewise linear function, so we can define
$$\Ver(h)\in \check\Delta^{[0]}_\tau$$
such that $\check\Delta_\tau-\Ver(h)\subseteq K_{\sigma}^\dual$. 
We also use this as $\Ver(h)\in \check\Delta^{[0]}_e$.

Let us assume we have $e:\tau_1\ra\tau_2$, two 
vertex embeddings $v\stackrel{g_v}{\ra}\tau_1$, $w\stackrel{g_w}{\ra}\tau_1$ 
and $h:\tau_2\ra\sigma_h$ an embedding in a maximal cell.
Set $m_h=\Ver(h)$ and $\hat{m}_h=(m_h,1)\in \hat{T}_{\check\Delta_e}$.
We define an isomorphism
$$\phi^h_{g_v,g_w}:\k \cdot\hat{m}_h \oplus\check\Lambda_{v,\k} \ra \k\cdot\hat{m}_h \oplus \check\Lambda_{w,\k}$$
as follows. Let $\gamma_h$ be some path from $v$ to $w$ through the interior of $\sigma$ and
$T_{\gamma_h}:\check\Lambda_{v,\k}\ra \check\Lambda_{w,\k}$ be the isomorphism induced by
parallel transport along $\gamma_h$. We set
$$\phi^h_{g_v,g_w}|_{\check\Lambda_{v,\k}}(m)
=T_{\gamma_h}(m)+\langle m,\Ver(g_w)-\Ver(g_v)\rangle\cdot \hat{m}_h,
\qquad \phi^h_{g_v,g_w}(\hat{m}_h)=\hat{m}_h.$$

\begin{lemma} \label{indepchoiceh}
 Let $v\stackrel{g_v}{\ra}\tau_1$, $w\stackrel{g_w}{\ra}\tau_1$ be two
 vertex embeddings. The isomorphism 
 $\phi_{g_v,g_w}$
 defined by the commutative diagram
 $$\begin{CD}
 \widehat{\check\Lambda}_{v,\k} 
 @>{\phi_{g_v,g_w}}>> 
 \widehat{\check\Lambda}_{w,\k}\\
 @| @| \\
 \k \hat{m}_h \oplus\check\Lambda_{v,\k} @>{\phi^h_{g_v,g_w}}>> \k \hat{m}_h \oplus \check\Lambda_{w,\k}
 \end{CD}$$
 is independent of the choice of $h:\tau_2\ra\sigma_h$.
\end{lemma}

\begin{proof} 

\underline{Step 1:}\ The case where $v,w$ are connected by an edge $\omega$.\\
 Let $o:\omega\ra\tau_2$ be the embedding of the edge in $\tau_1$ 
 composed with $e:\tau_1\ra\tau_2$. 
 Let $h':\tau_2\ra\sigma_{h'}$ be another inclusion in a maximal cell and 
 $\gamma_{h'}$ a path from $v$ to $w$ through the interior of $\sigma_{h'}$. 
 We have
 $$T_{\gamma_h}=T_{\gamma_{h'}}\circ T^{-1}_{\gamma_{h'}} \circ T_{\gamma_h}.$$
 Note that $T^{-1}_{\gamma_{h'}} \circ T_{\gamma_h}=T_{\gamma_h\circ \gamma^{-1}_{h'}}$ 
 is a monodromy transformation along the loop $\gamma_h\circ \gamma^{-1}_{h'}$ based at $v$ 
 which is by [\cite{grosie1}, section 1.5] (choosing $d_\omega$ to point from $v$ to $w$) given as
 $$\begin{array}{rcl}
 T_{\gamma_h\circ \gamma^{-1}_{h'}}(m)
 &=&(T^{h\circ o,h'\circ o}_\omega)^*(m)\\
 &=&m + \langle m,\Ver(g_w)-\Ver(g_v)\rangle \cdot n^{h\circ o,h'\circ o}_\omega \\
 &=&m + \langle m,\Ver(g_w)-\Ver(g_v)\rangle (\Ver(h')-\Ver(h))\\
 &=&m + \langle m,\Ver(g_w)-\Ver(g_v)\rangle (\hat{m}_{h'}-\hat{m}_h)
 \end{array}$$
 Using this and that $\Ver(h')-\Ver(h)\in\tau_1^\perp$ 
 is invariant under local monodromy, we get
 $$\begin{array}{rcl}
 \phi^h_{g_v,g_w}|_{\check\Lambda_{v,\k}}(m)
 &=&T_{\gamma_h}(m)-\langle m,\Ver(g_w)-\Ver(g_v)\rangle \hat{m}_h\\
 &=&T_{\gamma_{h'}}\circ T_{\gamma_h\circ \gamma^{-1}_{h'}}(m)
    -\langle m,\Ver(g_w)-\Ver(g_v)\rangle \hat{m}_h\\
 &=&T_{\gamma_{h'}}(m + \langle m,\Ver(g_w)-\Ver(g_v)\rangle (\hat{m}_{h'}-\hat{m}_h))\\
    &&\quad +\,\langle m,\Ver(g_w)-\Ver(g_v)\rangle \hat{m}_h\\
 &=&T_{\gamma_{h'}}(m) + \langle m,\Ver(g_w)-\Ver(g_v)\rangle \hat{m}_{h'}\\
 &=&\phi^{h'}_{g_v,g_w}|_{\check\Lambda_{v,\k}}(m)
 \end{array}$$
 Note that $\phi^{h}_{g_v,g_w}(\hat{m}_{\tilde{h}})=\hat{m}_{\tilde{h}}$ 
 for all $\tilde{h}:\tau_2\ra\sigma_{\tilde{h}}\in\P^{[\dim B]}$ because
 $\hat{m}_{\tilde{h}}-\hat{m}_h$ is monodromy invariant 
 and $\hat{m}_h$ is fixed by definition. \smallskip

\underline{Step 2:}\ Chains of edges.\\
 Pick some $h:\tau_2\ra\sigma_{h}$ and
 let $\omega_1,...,\omega_k$ be a chain of edges of $\tau_1$ 
 connecting vertex $v$ to vertex $w$. Let
 $v^-_{\omega_i}, v^+_{\omega_i}$ be the vertices of $\omega_i$, 
 s.t. $v=v^-_{\omega_1}$, $v^+_{\omega_i}=v^-_{\omega_{i+1}}$ and $v^+_{\omega_k}=w$.
 Let $g_{v^-_{\omega_i}}, g_{v^+_{\omega_i}}$ denote the respective embeddings in $\tau_1$.
 We set $\phi^{h}_{\omega_i}:=\phi^{h}_{g_{v^-_{\omega_i}},g_{v^+_{\omega_i}}}$ for each $i$ and claim
 $$\phi^h_{g_v,g_w}=\phi^h_{\omega_k}\circ \dots\circ \phi^h_{\omega_1}$$
 Note that $\phi^h_{g_v,g_w}(\hat{m}_h)=\phi^h_{\omega_k}\circ \dots\circ
 \phi^h_{\omega_1}(\hat{m}_h)$. Let $\gamma^h_{\omega_i}$ be a path through the interior
 of $\sigma_h$ connecting $v^-_{\omega_i}$ and $v^+_{\omega_i}$, then
 $\gamma_h\sim\gamma^h_{\omega_k}\circ \dots\circ \gamma^h_{\omega_1}$.
 We compute
 $$\begin{array}{rcl}
 \phi^h_{\omega_k}\circ \dots\circ \phi^h_{\omega_1}(m)
 &=&T_{\gamma^h_{\omega_k}}(...(T_{\gamma^h_{\omega_1}}(m)
   +\langle m,\Ver(v^+_{\omega_1})-\Ver(v^-_{\omega_1})\rangle \hat{m}_h)+...)\\
   &&\quad +\,\langle T_{\gamma^h_{\omega_{k-1}}}\circ\dots\circ 
   T_{\gamma^h_{\omega_1}}(m),\Ver(v^+_{\omega_k})-\Ver(v^-_{\omega_k})\rangle \hat{m}_h\\
 &=&T_{\gamma^h_{\omega_k}}\circ\dots\circ T_{\gamma^h_{\omega_1}}(m)
   +\langle m,\Ver(v^+_{\omega_1})-\Ver(v^-_{\omega_1})\rangle \hat{m}_h+...\\
   &&\quad +\,\langle m,\Ver(v^+_{\omega_k})-\Ver(v^-_{\omega_k})\rangle \hat{m}_h\\
 &=&T_{\gamma^h}(m)
   +\langle m,\Ver(v^+_{\omega_k})-\Ver(v^-_{\omega_1})\rangle \hat{m}_h\\
 &=&\phi^h_{g_v,g_w}(m)\\
 \end{array}$$
 where we have used that 
 $\langle T_\gamma(m), \Ver(v_1)-\Ver(v_2)\rangle=
 \langle m, \Ver(v_1)-\Ver(v_2)\rangle$ holds for each
 path $\gamma$ connecting some $v_1,v_2\in \tau^{[0]}_1$ through $\sigma_h$. \smallskip

\underline{Step 3:} Combining Step 1 and 2.\\
 Let $\omega_1,...,\omega_k$ be a chain of edges of $\tau$ 
 connecting $v$ to $w$ as described in Step 2. We conclude
 $$\phi^h_{g_v,g_w}
 \stackrel{\tiny\hbox{Step 2}}{=}\phi^h_{\omega_k}\circ \dots\circ \phi^h_{\omega_1}
 \stackrel{\tiny\hbox{Step 1}}{=}\phi^{h'}_{\omega_k}\circ \dots\circ \phi^{h'}_{\omega_1}\\
 \stackrel{\tiny\hbox{Step 2}}{=}\phi^{h'}_{g_v,g_w}.$$

\end{proof}

\begin{lemma}[Changing $v$] \label{changev} 
 Let $v\stackrel{g_v}{\ra}\tau_1$, $w\stackrel{g_w}{\ra}\tau_1$ be two
 vertex embeddings.
 We have a commutative diagram\\
 \begin{center}
 \resizebox{15cm}{!}{
 \begin{picture}(0,0)%
\includegraphics{changev.pstex}%
\end{picture}%
\setlength{\unitlength}{4144sp}%
\begingroup\makeatletter\ifx\SetFigFontNFSS\undefined%
\gdef\SetFigFontNFSS#1#2#3#4#5{%
  \reset@font\fontsize{#1}{#2pt}%
  \fontfamily{#3}\fontseries{#4}\fontshape{#5}%
  \selectfont}%
\fi\endgroup%
\begin{picture}(7275,2388)(211,-2290)
\put(226,-61){\makebox(0,0)[lb]{\smash{{\SetFigFontNFSS{12}{14.4}{\familydefault}{\mddefault}{\updefault}{\color[rgb]{0,0,0}$\dots$}%
}}}}
\put(226,-2221){\makebox(0,0)[lb]{\smash{{\SetFigFontNFSS{12}{14.4}{\familydefault}{\mddefault}{\updefault}{\color[rgb]{0,0,0}$\dots$}%
}}}}
\put(7471,-2221){\makebox(0,0)[lb]{\smash{{\SetFigFontNFSS{12}{14.4}{\familydefault}{\mddefault}{\updefault}{\color[rgb]{0,0,0}$\dots$}%
}}}}
\put(7471,-61){\makebox(0,0)[lb]{\smash{{\SetFigFontNFSS{12}{14.4}{\familydefault}{\mddefault}{\updefault}{\color[rgb]{0,0,0}$\dots$}%
}}}}
\put(856,-1051){\makebox(0,0)[lb]{\smash{{\SetFigFontNFSS{12}{14.4}{\familydefault}{\mddefault}{\updefault}{\color[rgb]{0,0,0}{\scriptsize$\id\otimes\bigwedge^r\phi_{g,g'}$}}%
}}}}
\put(6076,-1051){\makebox(0,0)[lb]{\smash{{\SetFigFontNFSS{12}{14.4}{\familydefault}{\mddefault}{\updefault}{\color[rgb]{0,0,0}{\scriptsize$\id\otimes\bigwedge^{r+1}\phi_{g,g'}$}}%
}}}}
\put(3151,-1141){\makebox(0,0)[lb]{\smash{{\SetFigFontNFSS{12}{14.4}{\familydefault}{\mddefault}{\updefault}{\color[rgb]{0,0,0}$F_s(e)^*\Omega^r_{\tau_2}\big/\Tors$}%
}}}}
\put(4501,-781){\makebox(0,0)[lb]{\smash{{\SetFigFontNFSS{12}{14.4}{\familydefault}{\mddefault}{\updefault}{\color[rgb]{0,0,0}$F_s(e\circ g_v)^*\Omega^r_{v}$}%
}}}}
\put(3151,-511){\makebox(0,0)[lb]{\smash{{\SetFigFontNFSS{12}{14.4}{\familydefault}{\mddefault}{\updefault}{\color[rgb]{0,0,0}$\shC^r(\check\Lambda_v,Z_e)$}%
}}}}
\put(3151,-1771){\makebox(0,0)[lb]{\smash{{\SetFigFontNFSS{12}{14.4}{\familydefault}{\mddefault}{\updefault}{\color[rgb]{0,0,0}$\shC^r(\check\Lambda_w,Z_e)$}%
}}}}
\put(4501,-1501){\makebox(0,0)[lb]{\smash{{\SetFigFontNFSS{12}{14.4}{\familydefault}{\mddefault}{\updefault}{\color[rgb]{0,0,0}$F_s(e\circ g_w)^*\Omega^r_{w}$}%
}}}}
\put(991,-2221){\makebox(0,0)[lb]{\smash{{\SetFigFontNFSS{12}{14.4}{\familydefault}{\mddefault}{\updefault}{\color[rgb]{0,0,0}$\shK^r(\check\Lambda_w,Z_e,-(r+1))$}%
}}}}
\put(5131,-2221){\makebox(0,0)[lb]{\smash{{\SetFigFontNFSS{12}{14.4}{\familydefault}{\mddefault}{\updefault}{\color[rgb]{0,0,0}$\shK^{r+1}(\check\Lambda_w,Z_e,-(r+1))$}%
}}}}
\put(5131,-61){\makebox(0,0)[lb]{\smash{{\SetFigFontNFSS{12}{14.4}{\familydefault}{\mddefault}{\updefault}{\color[rgb]{0,0,0}$\shK^{r+1}(\check\Lambda_v,Z_e,-(r+1))$}%
}}}}
\put(991,-61){\makebox(0,0)[lb]{\smash{{\SetFigFontNFSS{12}{14.4}{\familydefault}{\mddefault}{\updefault}{\color[rgb]{0,0,0}$\shK^r(\check\Lambda_v,Z_e,-(r+1))$}%
}}}}
\end{picture}%

 }
 \end{center}
 and thus a canonical isomorphism of
 $Q^\bullet(F_s(e)^*\Omega_{\tau_1}^r,g)$ 
 and $Q^\bullet(F_s(e)^*\Omega_{\tau_1}^r,g')$ as desired at the beginning of this section.
\end{lemma}

\begin{proof}
Note that the outer rectangle clearly commutes. The only interesting new information
is, in fact, the subdiagram consisting of the five left-most terms. We are going to
use the comparison of the two outgoing arrows of $F_s(e)^*\Omega_{\tau_1}^r/\Tors$ which
was given in \cite{grosie2}, Lemma 3.13. 
We may assume that $v$ and $w$ are connected by an edge $\omega$ because
any two vertices can always be connected by a chain of edges and, having proved the
edge version, we have a chain of commutative diagrams inducing commutativity 
of the first and the last. Let $o:\omega\ra\tau_1$ be this edge.
We choose some $h:\tau_2\ra\sigma_h\in\P^{[\dim B]}$ which determines a chart 
$U_{\sigma_h}$ of $X_{\tau_2}$ on which we show the commutativity of the diagram.
Let $f$ be an equation of $Z_{e\circ o}$, i.e., $f$ is constant if 
$e\circ o\not\in \Upomega_{\tau_2}$ and is an equation of $Z_{\tau_2}$ otherwise. 
We also assume that $\Ver(h)\in\check\Delta_{\tau_2}^{[0]}$ lies in the origin, 
such that $f$ is a regular function on $U_{\sigma_h}$.

Let $\gamma_h$ be some path from $v$ to $w$ through the
interior of $\sigma_h$ giving the identification $T_{\gamma_h}$ of
$\check\Lambda_v$ and $\check\Lambda_w$ which we also identify with a fixed lattice $M$. 
Note that in loc.cit.
this is denoted $N$ and note further that we are only interested in the $/\k^\log$ case.
We denote the field of fractions construction by $\Quot$. 
By loc.cit., the map $\Gamma_\omega: F_s(e\circ g_v)^*\Omega_{v}^r\ra F_s(e\circ g_w)^*\Omega_{w}^r$ 
written as a map 
$$\Quot(F_s(e\circ g_v)^*\shO_{X_v})\otimes_\ZZ\bigwedge^r M\ra 
\Quot(F_s(e\circ g_v)^*\shO_{X_w})\otimes_\ZZ\bigwedge^r M$$
on the toric chart $U_{\sigma_h}$ determined by $\sigma_h$ is given by 
$$\Gamma_{\omega}(1\otimes \alpha) = 1\otimes \alpha+\frac{d \tilde{f}}{\tilde{f}}
  \wedge (\iota(d_\omega) \alpha)$$
where $\tilde{f}=f^{a_\omega}$ is giving the log structure at $\omega$ and 
$d_\omega$ denotes the primitive vector pointing from $v$ to $w$. 
Using $a_\omega d_\omega=\Ver(v)-\Ver(w)$ whenever $df\neq 0$, we obtain
$$\Gamma_{\omega}(1\otimes \alpha) = 1\otimes \alpha+\frac{d f}{f}
  \wedge (\iota(\Ver(v)-\Ver(w)) \alpha).$$
The question now becomes whether
$$
\begin{CD}
\shO_{U_{\sigma_h}}(-Z_{\tau_2})\otimes_\ZZ\bigwedge^r (M\oplus \ZZ \hat{m}_h)
  @>{\pi}>> \shC(M,Z_{\tau_2})|_{U_{\sigma_h}}\ \hra\ & \Quot\shO_{U_{\sigma_h}}\otimes_\ZZ\bigwedge^r M\\
@V{\phi\,:=\,\id\otimes \bigwedge^r \phi^h_{g_v,g_w}}VV & @VV{\Gamma_\omega}V\\
\shO_{U_{\sigma_h}}(-Z_{\tau_2})\otimes_\ZZ\bigwedge^r (M\oplus \ZZ \hat{m}_h) 
  @>{\pi}>> \shC(M,Z_{\tau_2})|_{U_{\sigma_h}}\ \hra\ & \Quot\shO_{U_{\sigma_h}}\otimes_\ZZ\bigwedge^r M\\
\end{CD}
$$
commutes. Recall from the proof of Theorem~\ref{longexactseq}, that for
a suitable trivialization of $\shO_{X_{\tau_2}}(-Z_{e\circ o})|_{U_{\sigma_h}}$, 
$\alpha \in \bigwedge^{r-1} M$ and $\beta \in \bigwedge^{r} M$ we have
$$\pi|_{U_{\sigma_h}}(1\otimes \hat{m}_h\wedge\alpha)= df\wedge\alpha
\qquad\qquad \pi|_{U_{\sigma_h}}(1\otimes\beta)=f\otimes\beta.$$
It follows
$$\begin{array}{rcl}
(\Gamma_\omega\circ\pi) (1\otimes \hat{m}_h\wedge\alpha)
&=&df\wedge\alpha + \frac{d f}{f}\wedge\iota(\Ver(w)-\Ver(v)) (df\wedge \alpha)\\
&=&df\wedge\alpha + \frac{d f}{f}\wedge df
  \wedge (\iota(\Ver(w)-\Ver(v)) \alpha)\\
&=&df\wedge\alpha
\end{array}$$
$$\begin{array}{rcl}
(\Gamma_\omega\circ\pi) (1\otimes \beta)&=&
f\otimes\beta + \frac{d f}{f}\wedge 
(f\otimes\iota(\Ver(w)-\Ver(v))\beta)\\
&=& f\otimes\beta + d f\wedge (\iota(\Ver(w)-\Ver(v))\beta)\\
\end{array}$$

The map $\phi^h_{g_v,g_w}|_M$ reads $m\mapsto m+\iota(\Ver(w)-\Ver(v))m\cdot \hat{m}_h$
which extends in this form to $M\oplus \ZZ \hat{m}_h$ by setting 
\mbox{$\iota(\Ver(w)-\Ver(v))\hat{m}_h=0$}. A simple computation then shows for 
$\varepsilon\in\bigwedge^r (M\oplus \ZZ \hat{m}_h)$
$$\bigwedge^r\phi^h_{g_v,g_w}:\, \varepsilon\, \mapsto\, 
  \varepsilon+\hat{m}_h\wedge \iota(\Ver(w)-\Ver(v))\varepsilon.$$
We obtain
$$\begin{array}{rcl}
(\pi\circ \phi)(1\otimes \hat{m}_h\wedge\alpha)
&=&\pi(1\otimes \hat{m}_h\wedge\alpha + \hat{m}_h\wedge \iota(\Ver(w)-\Ver(v)) (\hat{m}_h\wedge\alpha))\\
&=&df\wedge\alpha + \pi(\hat{m}_h\wedge \hat{m}_h\wedge-\iota(\Ver(w)-\Ver(v))\alpha)\\
&=&df\wedge\alpha,\\
\end{array}$$
$$\begin{array}{rcl}
(\pi\circ \phi)(1\otimes \beta)&=&
\pi(1\otimes \beta)+\pi(1\otimes \hat{m}_h\wedge \iota(\Ver(w)-\Ver(v))\beta)\\
&=& f\otimes\beta + df\wedge (\iota(\Ver(w)-\Ver(v))\beta).\\
\end{array}$$
We have shown $\Gamma_\omega\circ\pi=\phi\circ\pi$ and thus the above diagram commutes. 
We arrive at the last part of the assertion. Note that $\Delta_e^\perp$ is invariant under 
monodromy and $\phi_{g_v,g_w}|_{\widehat{\Delta}_e^\perp}$ is, in this sense, the identity.
Looking at the definition of $Q^\bullet$, we see that the only term affected by changing 
the vertex is $\shK^\bullet(\check\Lambda_{v,\k},Z_e,-(r+1))$. 
It is not hard to see now that $\phi_{g_v,g_w}$ yields the claimed
isomorphism of the $Q^\bullet$'s.
\end{proof}

\begin{definition} 
We use the notation $\Phi_{g,g'}$ for the just constructed isomorphism
$$\Phi_{g,g'}:Q^\bullet(F_s(e)^*\Omega_{\tau_1}^r,g)
\ra Q^\bullet(F_s(e)^*\Omega_{\tau_1}^r,g').$$
\end{definition}

By the results of this section, from now on, we will sometimes use the notation
$Q^\bullet(F_s(e)^*\Omega_{\tau_1}^r)$ for the resolution of 
$F_s(e)^*\Omega_{\tau_1}^r/\Tors$ and only specify/choose some 
$g$ when necessary for computations.

\begin{remark} 
\begin{enumerate}
\item To pick up the discussion from Rem.~\ref{moveDelta}, note that
the span of $C(\check\Delta_e)$ is invariant under monodromy and fixed by $\phi_{g_v,g_w}$.
The map $\phi_{g_v,g_w}$ depends on the position of $\check\Delta_e$. 
It is not hard to see, however, that moving $\check\Delta_e$ commutes with $\phi_{g_v,g_w}$.
\item The main point of $\phi_{g_v,g_w}$ is that the projection 
$h:\widehat{\check\Lambda}_{v,\k}\ra\k$ doesn't commute with $\phi_{g_v,g_w}$ 
if $\Delta_e$ is non-trivial. In fact, each vertex of $\Delta_e$ gives one such projection.
If we dualise, the projections turn into inclusions of rays. This fits in with the construction
of $(B,\P)$ from a polytope as described in \cite{grosie1},~Ex.~1.18. 
What we have produced here is some sort of a local version of this. 
One can show that this yields a local system of rank $\dim B+1$ along 
the discriminant locus $\Delta$. If $X$ comes from the Batyrev construction,
this local system is the restriction of the constant sheaf 
on $B$ induced from the embedding into the surrounding vector space.
\end{enumerate}
\end{remark}


\subsection{Cohomology on a single stratum} \label{Coho_singlestrat}
As before, we assume throughout that $X$ is a h.t. toric log CY space.
In this section, we compute the cohomology of a summand of the complex $\SC^\bullet(\Omega^r)$.
We first need a lemma on locally monodromy invariant forms. 
Recall that $i:B\backslash\Delta\hra B$ denotes the inclusion of the non-singular locus of $B$.

\begin{lemma} \label{lem_localaffhodge}
Given $\tau_1\stackrel{e}{\ra}\tau_2$, the space
$\Gamma(W_e,i_*\bigwedge^r\check\Lambda\otimes_\ZZ\k)$
is generated by $\bigwedge^{r}\Delta_{e}^\perp$ and $\langle\bigwedge^\top T_{\check\Delta_{e}} \rangle_r$.
\end{lemma}

\begin{proof} 
Given any point $y\in W_e\backslash \Delta$, we may identify 
$\Gamma(W_e,i_*\bigwedge^r\check\Lambda\otimes_\ZZ\k)$ with the subspace of 
$\bigwedge^r\check\Lambda_{y,\k}\otimes_\ZZ\k$ of forms invariant under monodromy transformations 
by loops in $W_e\backslash \Delta$. 
If $e\not\in \Delta$, we have $\Delta_{e}^\perp=\check\Lambda_y$ and the assertion is trivial.
Let us assume $e\in \Delta$.
Recall from [\cite{grosie1}, Section~1.5] that the group of monodromy transformations is 
generated by
$$\alpha \mapsto \alpha \pm \kappa_{\omega\rho}\cdot(\iota(d_\omega)\alpha)\wedge d_\rho$$
where $d_\omega$ is a primitive integral vector parallel to some $\omega\in\P^{[1]}$ and 
$d_\rho$ is a primitive integral vector in $\rho^\perp$ for some $\rho\in\P^{[\dim B-1]}$ 
such that there is an edge $\hat{e}:\omega\ra \rho$ with $\hat{e}\in\Delta$ which factors through $e$ 
(otherwise $\kappa_{\omega\rho}=0$).
By Lemma~\ref{inoutkappa}, such a $d_\omega$ is parallel to an edge of $\Delta_e$ 
and $d_\rho$ is parallel to an edge of $\check\Delta_e$. It now becomes obvious that
$\bigwedge^{r}\Delta_{e}^\perp$ and $\langle\bigwedge^\top T_{\check\Delta_{e}} \rangle_r$ are contained in
$\Gamma(W_e,i_*\bigwedge^r\check\Lambda\otimes_\ZZ\k)$. 

Now assume $\beta\not \in \langle\bigwedge^\top T_{\check\Delta_{e}} \rangle_r+\bigwedge^{r}\Delta_{e}^\perp$. 
We will exhibit some monodromy transformation which doesn't fix $\beta$.
We choose $\omega_1,...,\omega_m$ such that $d_{\omega_1},...,d_{\omega_m}$ form a basis of $T_{\Delta_e}$
and such that there is some $\rho\in\P^{[\dim B-1]}$ with $\kappa_{\omega_i\rho}\neq 0$ for $1\le i\le m$. 
Similarly we choose $\rho_1,...,\rho_n$ such that $d_{\rho_1},...,d_{\rho_n}$ form a basis of $T_{\check\Delta_e}$
and such that there is some $\omega\in\P^{[1]}$ with $\kappa_{\omega\rho_i}\neq 0$ for $1\le i\le n$.
By Lemma~\ref{equivalencesfore}, we have $\kappa_{\omega_i\rho_j}\neq 0$ for all $i,j$.
We may complement $d_{\rho_1},...,d_{\rho_n}$ to a basis $\frak{B}$ of $\check\Lambda_{y,\k}$ by adding in particular
vectors $d^*_{\omega_1},...,d^*_{\omega_n}$ with the property 
$\iota(d_{\omega_i})b=0$ for $b\in \frak{B}\backslash \{ d^*_{\omega_i}\}$.
The basis $\frak{B}$ of $\check\Lambda_{y,\k}$ induces a basis $\bigwedge^{r}\frak{B}$ of $\bigwedge^{r}\check\Lambda_{y,\k}$.
We represent $\beta$ in this basis. Note that $\bigwedge^{r}\Delta_{e}^\perp$ and $\langle\bigwedge^\top T_{\check\Delta_{e}} \rangle_r$
are both generated by a subset of $\bigwedge^{r}\frak{B}$. Thus, by assumption, $\beta$ has a non-zero coefficient for some 
basis element $\frak{b}$ in $\bigwedge^{r}\frak{B}$ which is not contained in these subsets.
Therefore, there is some $i$ such that $\iota(d_{\omega_i})\frak{b} \neq 0$ and
there is some $j$ such that $d_{\rho_j}\wedge \frak{b}\neq 0$.
We claim that the monodromy transformation
$\alpha \mapsto \alpha \pm \kappa_{\omega_i\rho_j}\cdot(\iota(d_{\omega_i})\alpha)\wedge d_{\rho_j}$
changes $\beta$. This is equivalent to saying
$(\iota(d_{\omega_i})\beta)\wedge d_{\rho_j}\neq 0.$
This follows from $(\iota(d_{\omega_i})\frak{b})\wedge d_{\rho_j}\neq 0$
and a linear independence argument.
\end{proof}

\begin{theorem} \label{cohotau}
  For $v\stackrel{g}{\ra}\tau_1\stackrel{e}{\ra}\tau_2$ with $v\in\P^{[0]}$, we have
  $$H^p(X_{\tau_2},F(e)^*\Omega^r_{\tau_1}/\Tors)
  =\left\{ 
   \begin{array}{rl}
    \Gamma(W_e,i_*\bigwedge^r\check\Lambda\otimes_\ZZ\k)&\hbox{ for }p=0,\\[8pt]  
    \displaystyle R(Z_{e})_p\otimes_\kk 
    \frac{\langle\bigwedge^\top T_{\check\Delta_{e}} \rangle 
    \cap \bigwedge^{r+p}\check\Lambda_{v,\k}}
    {\langle\bigwedge^\top T_{\check\Delta_{e}} \rangle
    \cap \bigwedge^{r+p}\Delta_{e}^\perp} &\hbox{ for }p>0.
   \end{array}
  \right.$$
\end{theorem}

\begin{remark} 
  A close look makes it apparent that
  this representation is independent of the choice of
  $v$, resp. $g$, because different choices of local edge connecting paths
  induce the same isomorphisms.
\end{remark}

\begin{proof} 
 If $e\not\in\Delta$, we have, by Lemma~\ref{equivalencesfore},
 $Z_e=\emptyset$. This means $R(Z_e)_\bullet=0$ and 
 $F_s(e)^*\Omega^r_{\tau_1}=\shO_{X_{\tau_2}}\otimes_\k\bigwedge^r\check\Lambda_{v,\k}$
 so the assertion is true. We now assume $e\in\Delta$ and get
 $Z_e=Z_{\tau_2}, \Delta_e=\Delta_{\tau_1}, \check\Delta_e=\check\Delta_{\tau_2}$.
 We apply the functor $H^p$ to the diagram (\ref{squarediagram}) to obtain
 $$\begin{CD}
 H^p(\shC^r(\Delta^\perp_{\tau_1},Z_{\tau_2})) @>{H^p(\pi)}>> 
 H^p(\shO_{X_{\tau_2}}\otimes_\k\bigwedge^r\Delta^\perp_{\tau_1})\\
 @VVV @VVV\\
 H^p(\shC^r(\check\Lambda_{v,\k},Z_{\tau_2})) 
 @>{H^p(\pi)}>> H^p(F_s(e)^*\Omega_{\tau_1}^r/\Tors).\\
 \end{CD}$$

 Note that the corresponding sequence on cohomology splits
 because the original sequence splits.
 Let's consider the case $p=0$ first. We can read off the exact sequence
 $$0 \ra \bigwedge^r\Delta^\perp_{\tau_1}
 \ra H^0(F_s(e)^*\Omega_{\tau_1}^r/\Tors) 
 \ra H^0(\shC^r(\check\Lambda_{v,\k},Z_{\tau_2}))
 /H^0(\shC^r(\Delta^\perp_{\tau_1},Z_{\tau_2}))
 \ra 0.$$
 We may use Cor.~\ref{cohoC} to obtain
 $H^0(\shC^r(\check\Lambda_{v,\k},Z_{\tau_2}))
 =R(Z_{\tau_2})_0\otimes_\k \langle\bigwedge^\top T_{\Delta_{\tau_2}} \rangle\cap
 \bigwedge^r\check\Lambda_{v,\k}$
 and
 $H^0(\shC^r(\Delta^\perp_{\tau_1},Z_{\tau_2}))
 =R(Z_{\tau_2})_0\otimes_\k \langle\bigwedge^\top T_{\Delta_{\tau_2}} \rangle\cap
 \bigwedge^r\Delta^\perp_{\tau_1}.$
 We have $R(Z_{\tau_2})_0=\Gamma(\shO_{X_{\tau_2}})=\k$, 
 so the exact sequence reads
 $$0 \ra \bigwedge^r\Delta^\perp_{\tau_1}
 \ra H^0(F_s(e)^*\Omega_{\tau_1}^r/\Tors) 
 \ra 
 \frac{\langle\bigwedge^\top T_{\Delta_{\tau_2}} \rangle\cap\bigwedge^r\check\Lambda_{v,\k}}
 {\langle\bigwedge^\top T_{\Delta_{\tau_2}} \rangle\cap\bigwedge^r\Delta^\perp_{\tau_1}}
 \ra 0.$$
 Therefore, $H^0(F_s(e)^*\Omega_{\tau_1}^r/\Tors)$ is identified with the subspace
 of $\bigwedge^r\check\Lambda_{v,\k}$ which is generated by
 $\bigwedge^r\Delta^\perp_{\tau_1}$ and 
 $\langle\bigwedge^\top T_{\Delta_{\tau_2}} \rangle\cap
 \bigwedge^r\check\Lambda_{v,\k}$. By Lemma~\ref{lem_localaffhodge}, the assertion for $p=0$ follows. 
 
 The case where $p>0$ is even simpler because
 $H^p(\shO_{X_{\tau_2}}\otimes_\k\bigwedge^r\Delta^\perp_{\tau_1})=0$. Again using
 Cor.~\ref{cohoC}, we directly have the assertion.
\end{proof}

\begin{proof}[Proof of Theorem~\ref{maintheorem},a)] 
Note that the functors $\Gamma(W_e,\cdot)$ and $\bigwedge^{r+p}$
commute on presheaves of vectors spaces on $W_e$.
We choose some $g:v\ra\tau_1$ with $v\in\P^{[0]}$. 
The assertion follows from Thm.~\ref{cohotau}
if we show that the cokernels $C_1,C_2$ in the diagram
$$
\xymatrix{    
  \langle\bigwedge^\top T_{\check\Delta_{e}} \rangle
  \cap \bigwedge^{r+p}\Delta_{e}^\perp 
  \ar@{^{(}->}[r] \ar@{^{(}->}[d] 
  &\langle\bigwedge^\top T_{\check\Delta_{e}} \rangle 
  \cap \bigwedge^{r+p}\check\Lambda_{v,\k}
  \ar@{->>}[r] \ar@{^{(}->}[d] 
  &C_1 \ar[d]\\  
  \Gamma(W_e,\bigwedge^{r+p}i_*\check\Lambda\otimes\k)
  \ar@{^{(}->}[r]
  &\Gamma(W_e,i_*\bigwedge^{r+p}\check\Lambda\otimes\k)
  \ar@{->>}[r]
  &C_2\\  
} $$
are isomorphic. The natural map $C_1\ra C_2$ is injective because the left
square is cartesian which follows from $\Delta_e^\perp=\Gamma(W_e,i_*\check\Lambda\otimes_\ZZ \k)$. 
Surjectivity is a consequence of the
fact that the term in the middle of the bottom row is generated
by the images of the two incoming arrows which we know by Lemma~\ref{lem_localaffhodge}.
\end{proof}


\subsection{The strata combining differential}
\label{secstratadiff}
Up to now, we have been working on a single stratum $X_{\tau_2}$ only. 
Now we take into consideration the barycentric differential $d_\bct$. 
We are going to produce an acyclic resolution of the complex $\SC^\bullet(\Omega^r)$ to
have an explicit description of the hypercohomology spectral sequence 
of $\SC^\bullet(\Omega^r)$ for the proof of Thm.~\ref{maintheorem},b).
Our overall hypothesis is that $X$ is a h.t. toric log CY space.
Recall that the \'etale locally closed embedding 
of the stratum to $X$ is denoted by $q_\tau:X_\tau\ra X$.
For the first half of this section, we fix a chain of maps 
$$ v\stackrel{\hat{g}}{\ra} 
\overbrace{\sigma_1\ra\tau_1 \stackrel{e}{\ra}\tau_2\ra\sigma_2}^{\hat{e}}.$$
with $v\in\P^{[0]}$. We denote the composition of the first two 
maps by $g:v\ra\tau_1$. It is not hard to see that
$Z_{\hat{e}}= Z_e\cap X_{\sigma_2}$,
$\Delta_{\hat{e}}=\Delta_e\cap \sigma_1$,
$\check\Delta_{\hat{e}}=\check\Delta_e\cap \sigma_2^\perp$.
Recall from [\cite{grosie2}, Prop. 3.8] that we get a map
$$q_{\tau_2,*}F(e)^*\Omega^r_{\tau_1}/\Tors\ra
q_{\sigma_2,*}F(\hat{e})^*\Omega^r_{\sigma_1}/\Tors$$
which factors through the restriction to $X_{\sigma_2}$.

\begin{lemma} \label{lemstratafunct}
We have a commutative diagram
{\scriptsize
$$
\begin{CD}
0 \ra & q_{\tau_2,*}\shC^r(\Delta^\perp_{e},Z_{e}) 
&\ra &q_{\tau_2,*}\shC^r(\check\Lambda_{v,\k},Z_{e})\oplus
q_{\tau_2,*}\shO_{X_{\tau_2}}\otimes\bigwedge^r\Delta^\perp_{e} 
&\ra &q_{\tau_2,*}F_s(e)^*\Omega_{\tau_1}^r/\Tors 
&\ra 0\\
& @VVV @VVV @VVV\\
0 \ra & q_{\sigma_2,*}\shC^r(\Delta^\perp_{\hat{e}},Z_{\hat{e}}) 
&\ra &q_{\sigma_2,*}\shC^r(\check\Lambda_{v,\k},Z_{\hat{e}})\oplus
q_{\sigma_2,*}\shO_{X_{\sigma_2}}\otimes\bigwedge^r\Delta^\perp_{\hat{e}} 
&\ra &q_{\sigma_2,*}F_s(\hat{e})^*\Omega_{\sigma_1}^r/\Tors 
&\ra 0
\end{CD}
$$}

where the rows are given by Prop~\ref{squareseq}, 
the right vertical map is the one just mentioned and 
the left two vertical maps are the composition of the 
restriction to $X_{\sigma_2}$, the map induced by 
$\Delta_e^\perp\ra \Delta_{\hat{e}}^\perp$ and
Lemma~\ref{lempullbackres}.
\end{lemma}

\begin{proof} 
 Replacing the right most non-trivial terms with
 $$F_s(\tau_2\ra\sigma_2)^*:
 q_{\tau_2,*}\shO_{X_{\tau_2}}\otimes_\k \bigwedge^r\check\Lambda_{v,\k}
 \ra q_{\sigma_2,*}\shO_{X_{\sigma_2}}\otimes_\k \bigwedge^r\check\Lambda_{v,\k}$$
 clearly gives a commutative diagram by the functoriality from Prop.~\ref{naturalres}
 and the fact that $F_s(\tau_2\ra\sigma_2)^*$ is a functor. The assertion then follows
 from the commutative diagram
 $$
 \begin{CD}
 F_s(e)^*\Omega_{\tau_1}^r/\Tors @>>> F_s(e\circ g)^*\Omega_{v}^r  
   & \, =\shO_{X_{\tau_2}}\otimes_\k \bigwedge^r\check\Lambda_{v,\k}\\
 @V{F_s(\tau_2\ra\sigma_2)^*}VV @VV{F_s(\tau_2\ra\sigma_2)^*}V \\
 F_s(\hat{e})^*\Omega_{\sigma_1}^r/\Tors @>>> F_s(\hat{e}\circ \hat{g})^*\Omega_{v}^r
   & \,=\shO_{X_{\sigma_2}}\otimes_\k \bigwedge^r\check\Lambda_{v,\k}
 \end{CD}
 $$
 and the way of construction of the sequence in Prop.~\ref{squareseq}.
\end{proof}

For exactly the same reasons, we also obtain a diagram of resolutions,
replacing $\shC^r$'s with the suitable $\shK^{r+1+\bullet}$'s and removing
the summands $q_{\tau_2,*}\shO_{X_{\tau_2}}\otimes\bigwedge^r\Delta_e^\perp$ if $\bullet>0$.
We call this map on the $Q^\bullet$'s
$$d_e^{\hat{e}}:q_{\tau_2,*}Q^\bullet(F_s(e)^*\Omega_{\tau_1}^r,g)\ra
q_{\sigma_2,*}Q^\bullet(F_s(\hat{e})^*\Omega_{\sigma_1}^r,\hat{g}).$$
We make use of the statement of Theorem~\ref{cohotau} in the following Lemma.

\begin{lemma} \label{thmstratafunct}
The map
$H^p(X,q_{\tau_2,*}F_s(e)^*\Omega_{\tau_1}^r)\ra
H^p(X,q_{\sigma_2,*}F_s(\hat{e})^*\Omega_{\sigma_1}^r)$
is, for $p=0$, the restriction
$$\Gamma(W_e,i_*\bigwedge^r\check\Lambda\otimes_\ZZ\k)
\ra\Gamma(W_{\hat{e}},i_*\bigwedge^r\check\Lambda\otimes_\ZZ\k)$$
induced by the canonical isomorphism
$$\Gamma(W_{\hat{e}},i_*\bigwedge^r\check\Lambda\otimes_\ZZ\k)=
\Gamma(W_e\cap W_{\hat{e}},i_*\bigwedge^r\check\Lambda\otimes_\ZZ\k)$$
and the inclusion of open sets
$W_e\cap W_{\hat{e}}\subseteq W_e$.
It is, for $p>0$, induced by
\begin{itemize}
\item[ ] $F_s(\tau_2\ra\sigma_2)^*:R(Z_e)_\bullet \sra R(Z_{\hat{e}})_\bullet$,
  \nopagebreak
\item[ ] ${\check\Delta_{\hat{e}}}\hra {\check\Delta_{e}}$ and hence
  $\langle\bigwedge^\top T_{\check\Delta_{e}} \rangle
  \hra 
  \langle\bigwedge^\top T_{\check\Delta_{\hat{e}}} \rangle$,
  \nopagebreak
\item[ ] $\Delta_{\hat{e}}\hra\Delta_{e}$ and hence 
$\Delta_e^\perp\hra \Delta_{\hat{e}}^\perp.$
\end{itemize}
\end{lemma}

\begin{proof} 
The isomorphism $\Gamma(W_{\hat{e}},i_*\bigwedge^r\check\Lambda\otimes_\ZZ\k)=
\Gamma(W_e\cap W_{\hat{e}},i_*\bigwedge^r\check\Lambda\otimes_\ZZ\k)$ follows from the
fact that $e$ and $\hat{e}$ can be joined by a simplex in $\Delta$ with codimension two in $B$.
Such a simplex is given by a chain $\tau_0\subsetneq...\subsetneq \tau_{\dim B-2}$ with 
$\tau_0\in\P^{[1]},\tau_{\dim B-2}\in\P^{[\dim B-1]}$ which is a 
refinement of $e$ and $\hat{e}$. Thus by the proof of [\cite{grosie1}, Lemma~5.5], $(W_e\cap W_{\hat{e}})\backslash\Delta$ 
is a deformation retract of $W_{\hat{e}}\backslash\Delta$.

The assertion follows from computing the map $d_e^{\hat{e}}$ on the 
$Q^\bullet$'s and is straightforward. We just discuss the map of the $R$'s.
We set $F=F_s(\tau_2\ra\sigma_2)$. 
For each $l$ the natural adjunction map becomes
$$a:\Gamma(X_{\tau_2},\shO_{X_{\tau_2}}(pZ_e))\ra
\Gamma(X_{\tau_2},F_*F^*\shO_{X_{\tau_2}}(pZ_e))
=\Gamma(X_{\tau_2},F_*\shO_{X_{\tau_2}}(pZ_{\hat{e}}))$$
because $F^*Z_e=Z_{\hat{e}}$. Let $V_e$, $V_{\hat{e}}$ 
denote the linear systems via the log derivation map
for $Z_e$, $Z_{\hat{e}}$,  respectively, 
as given after Lemma~\ref{globalsec}. 
We may assume that $\check\Delta_{e}$ is embedded such that
the embedding of $\check\Delta_{e}$ is induced by restriction to
the corresponding face.
If $f$ is an equation of $Z_e$, then $F^*f$ is an equation of
$Z_{\hat{e}}$.
We get a map $V_e\ra V_{\hat{e}}$ by the diagram
$$
\begin{CD}
(N_{\tau_2}\oplus\ZZ)\otimes_\ZZ\kk  
@>{\partial_{Z_e}}>>
\Gamma(X_{\tau_2},\shO_{X_{\tau_2}}(Z_e))\\
@VVV @VVaV \\
(N_{\sigma_2}\oplus\ZZ)\otimes_\ZZ\kk
@>{\partial_{Z_{\hat{e}}}}>>
\Gamma(X_{\tau_2},F_*\shO_{X_{\tau_2}}(Z_{\hat{e}})).
\end{CD}
$$
The map on the $R$'s then is the cokernel of the diagram 
$$
\begin{CD}
\Gamma(X_{\tau_2},\shO_{X_{\tau_2}}((p-1)Z_{e}))\otimes_\k V_e 
@>>> \Gamma(X_{\tau_2},\shO_{X_{\tau_2}}(pZ_{e}))\\
@VVV @VVV \\
\Gamma(X_{\tau_2},F_*\shO_{X_{\sigma_2}}((p-1)Z_{\hat{e}}))\otimes_\k 
V_{\hat{e}} 
@>>> \Gamma(X_{\tau_2},F_*\shO_{X_{\sigma_2}}(pZ_{\hat{e}})).
\end{CD}
$$
Note that the right vertical map is surjective because 
by Lemma~\ref{globalsec}
it can be described by 
$\k^{l\cdot\check\Delta_{e}\cap M_{X_{\tau_2}}} \sra 
\k^{l\cdot\check\Delta_{\hat{e}}\cap M_{X_{\tau_2}}},\quad z^m\mapsto 0
\hbox{ if }m\not\in l\cdot\check\Delta_{\hat{e}}\hbox{ and }z^m\mapsto z^m\hbox{
otherwise}.$
\end{proof}

Recall the Gross-Siebert resolution from Def.~\ref{defCOmega}.
For each $r$, we are going to construct
an acyclic resolution of $\SC^\bullet(\Omega^r)$.
For each $\tau\in\P$, choose some $g_\tau:v_\tau\ra\tau$ with $v_\tau\in\P^{[0]}$.
We define the double complex 
$$\shQ^{k,l}(\Omega^r)=
\bigoplus_{\underbrace{\tau_0\ra\dots\ra\tau_k}_{e}}
q_{\tau_k,*}Q^l(F_s(e)^*\Omega_{\tau_0}^r/\Tors,g_{\tau_0})$$
where the differential in the $l$-direction is the usual one 
on $Q^\bullet$ which we are going to denote by $\delta$. 
The differential in the $k$-direction is
$$\begin{array}{rl}
(d_\bct(\alpha))_{\tau_0\ra\dots\ra\tau_{k+1}}=
&d_{\tau_1\ra\tau_{k+1}}^{\tau_0\ra\tau_{k+1}}\circ
\Phi_{g_{\tau_1},(\tau_0\ra\tau_1)\circ g_{\tau_0}} 
(\alpha_{\tau_1\ra\dots\ra\tau_{k+1}})\\
&+\sum_{i=1}^k(-1)^i\,\id(\alpha_{\tau_0\ra\dots\breve{\tau}_i\ra\dots\ra\tau_{k+1}})\\
&+(-1)^{k+1} d_{\tau_0\ra\tau_{k}}^{\tau_0\ra\tau_{k+1}}
  (\alpha_{\tau_0\ra\dots\ra\tau_{k}}).
\end{array}$$
We collect some results in the following lemma.

\begin{lemma} 
On the space $X$, we have for each $r$ a 
double complex of $\Gamma$-acyclic sheaves 
$$\shQ^{\bullet,\bullet}(\Omega^r)$$
which is exact in both directions except at the respective 
first non-trivial terms.
We have the augmentation
$$0\ra \SC^\bullet(\Omega^r)\ra\shQ^{\bullet,0}(\Omega^r).$$
\end{lemma}

\begin{proof} 
The exactness of the augmentation and acyclicity are the 
content of Prop.~\ref{resQ}.
The commutativity of differentials reduces to 
Lemma~\ref{lemstratafunct} and what was mentioned afterwards.
We are going to prove the exactness of $d_\bct$.
We set 
$Q_{\k}^0=\frac{\bigwedge^{r+1}\widehat{\check\Lambda}_{v_{\tau_0,\k}}
\ \oplus\ 
\bigwedge^{r}\Delta_{\tau_0}^\perp}
{\bigwedge^{r+1}\widehat{\Delta}^\perp_{\tau_0}}.$
Recall that, for $e:\tau_0\ra\tau_k$,
$$Q^l(F_s(e)^*\Omega_{\tau_0}^r/\Tors,g_{\tau_0})
=\left\{\begin{array}{ll}
\shO_{X_{\tau_{k}}}\otimes Q_{\k}^0
&\hbox{for }l=0\\
\shO_{X_{\tau_{k}}}(lZ_{\tau_k})\otimes
\frac{\bigwedge^{r+l+1}\widehat{\check\Lambda}_{v_{\tau_0,\k}}}
{\bigwedge^{r+l+1}\widehat{\Delta}^\perp_{\tau_0}}
&\hbox{for }l>0.\\
\end{array}\right.$$

What we want to prove is a local issue. Let $p\in X$ be some geometric point 
and $\tau\in\P$ be such that $p\in\Int(X_\tau)$. For $e:\tau_0\ra\tau_k$, we have
$$q_{\tau_k,*}Q^l(F_s(e)^*\Omega_{\tau_0}^r/\Tors,g_{\tau_0})_p=0
\hbox{ if there is no } \tau_k\ra\tau.$$
By Lemma~\ref{baryacyclic}, we are done if we 
show criterion $\op{(L)}$ from Section~\ref{A_bct}. 
We match the notation by setting $\Xi=\tau$ and
$M_{(\tau_0,\tau_k)}
=q_{\tau_k,*}Q^l(F_s(e)^*\Omega_{\tau_0}^r/\Tors,g_{\tau_0})_p$.
We may fix some $\tau_0,\tau_{k-1}\subseteq\tau$ with 
$\tau_0\subseteq\tau_{k-1}$. 
Let $(f_e)_e\in \bigoplus_{\tau_k\supsetneq\tau_{k-1}} M_{(\tau_0,\tau_k)}$
be a compatible collection. We want to show that it lifts as required 
in criterion $\op{(L)}$ from Section~\ref{A_bct}.
\\[4mm]
\underline{The case $p\in Z_{\tau_0}$:} 
This implies $Z_\tau=Z_{\tau_0}\cap X_\tau$ and thus 
$\Delta_e=\Delta_{\tau_0}$ for each $e:\tau_0\ra\tau'$ 
with $\tau'\subseteq\tau$.
We claim that each $M_{(\tau_0,\tau_k)}$ is the pullback of 
$M_{(\tau_0,\tau_{k-1})}$, i.e., 
for $F=F(\tau_{k-1}\ra\tau_k)$, the map 
$d_{\tau_0\ra\tau_{k-1}}^{\tau_0\ra\tau_{k}}$ induces
$$Q^l((F^*F_s(e)^*\Omega_{\tau_0}^r)/\Tors,g_{\tau_0})_p
=F^*Q^l(F_s(e)^*\Omega_{\tau_0}^r/\Tors,g_{\tau_0})_p.$$
Indeed, both $F^*M_{(\tau_0,\tau_{k-1})}$ 
and $M_{(\tau_0,\tau_{k})}$ are
$$\left\{
\begin{array}{ll}
\shO_{X_{\tau_{k}},p}\otimes Q_{\k}^0
&\hbox{for }l=0\\
\shO_{X_{\tau_{k}}}(lZ_{\tau_k})_p\otimes
\frac{\bigwedge^{r+l+1}\widehat{\check\Lambda}_{v_{\tau_0,\k}}}
{\bigwedge^{r+l+1}\widehat{\Delta}^\perp_{\tau_0}}
&\hbox{for }l>0.\\
\end{array}
\right.$$
Now criterion $\op{(L)}$ follows from the
fact that $M_{(\tau_0,\tau_{k-1})}$ is locally free on $X_{\tau_{k-1}}$ and that we can always
lift functions from subvarieties.
\\[4mm]
\underline{The case $p\not\in Z_{\tau_0}$:} 
We set
$U=\{\tau_k\subseteq \tau\,|\,\tau_k\supsetneq\tau_{k-1}, 
Z_{\tau_0}\cap X_{\tau_{k}}\neq\emptyset\}.$
Consider the following diagram with exact rows and columns.
$$
\xymatrix{ 
&K \ar@{^{(}->}[d]\ar@{=}[r] &K \ar@{^{(}->}[d] \\
\bigwedge^r \Delta^\perp_{\tau_0}\ar[r]\ar@{=}[d] & Q^0_{\k}\ar[r]\ar@{->>}[d]_{\iota(h)} & 
\frac{\bigwedge^{r+1}\widehat{\check\Lambda}_{v_{\tau_0,\k}}}
{\bigwedge^{r+1}\widehat{\Delta}^\perp_{\tau_0}}\ar@{->>}[d]_{\iota(h)} \\
\bigwedge^r \Delta^\perp_{\tau_0}\ar[r] & \bigwedge^{r}\check\Lambda_{v_{\tau_0,\k}}\ar[r]\ar@/_1pc/@{.>}[u] & 
\frac{\bigwedge^{r}\check\Lambda_{v_{\tau_0,\k}}}
{\bigwedge^{r}\Delta^\perp_{\tau_0}} 
 } $$

We choose a splitting as indicated by the dashed arrow.
We claim that the compatible collection decomposes as
$$(f_e)_e=(f^1_e,f^2_e)_e\in 
\bigoplus_{\tau_k\supsetneq\tau_{k-1}} (M^1_{(\tau_0,\tau_k)}\oplus M^2_{(\tau_0,\tau_k)})$$
where
$$M^1_{(\tau_0,\tau_k)}=\left\{
\begin{array}{ll}
\shO_{X_{\tau_{k}},p}\otimes \bigwedge^{r+l}\check\Lambda_{v_{\tau_0},\k}& l=0\\
0 &\hbox{otherwise}
\end{array}
\right.$$
and
$$
M^2_{(\tau_0,\tau_k)}=\left\{
\begin{array}{ll}
\shO_{X_{\tau_{k}},p}\otimes K & l=0,\tau_k\in U\\
\shO_{X_{\tau_{k}},p}\otimes \frac{\bigwedge^{r+l+1}\widehat{\check\Lambda}_{v_{\tau_0},\k}}
{\bigwedge^{r+l+1}\widehat{\Delta}_{\tau_0}^\perp} & l>0,\tau_k\in U\\
0 &\hbox{otherwise}.
\end{array}
\right.
$$
Indeed, we can identify $\shO_{X_{\tau_{k}}}(lZ_{\tau_k})_p=\shO_{X_{\tau_{k}},p}$ and
there is only one obvious way to decompose using the map $\iota(h)$ 
and the chosen splitting above. One can now show that both $(f^1_e)_e$ and $(f^2_e)_e$ lift.
The reason is again that functions from subvarieties lift. 
For $(f^2_e)_e$, one uses that $\{ X_{\tau_k}\,|\, \tau_k\in U\}$ is a set of strata 
closed under intersection and then the functions $f^2_e$ actually glue to a function on the
corresponding subspace. We have shown $\op{(L)}$ and can apply Lemma~\ref{baryacyclic}.
\end{proof}

As a corollary, we obtain a different proof of Lemma~\ref{COmegaexact} 
still using the same argument for the first term as in [\cite{grosie2},~Thm.~3.5] though.

The finiteness of the dimensions of global sections and 
the computability of cohomology is the major strength of $\shQ^{\bullet,\bullet}$.
The downside, however, is that it is impossible to extend the de Rham
differential to it to obtain a triple complex. Roughly speaking, differentiating elements
of $\shO_{X_{\tau_2}}(lZ_{\tau_2})$ yields something in $\shO_{X_{\tau_2}}((l+1)Z_{\tau_2})$
whereas for compatibility it would have to stay in $\shO_{X_{\tau_2}}(lZ_{\tau_2})$.
We will later make use of the fact that exterior differentiation can at least be defined on
$\shQ^{\bullet,0}(\Omega^\bullet)$.


\subsection{Degeneration at $E_2$} \label{Sec_Proofmainb}
We are now going to prove Thm.~\ref{maintheorem}, b).
The key ingredient is the previously constructed double complex
$\shQ^{\bullet,\bullet}(\Omega^r)$
We assume now that each $\check\Delta_e$ is a simplex in order to be able to use techniques 
from Section~\ref{sec_jacobianrings}.

\begin{definition} For each $r$, we define the subcomplex
 $\shQ_\top^{\bullet,\bullet}(\Omega^r)\subseteq \shQ^{\bullet,\bullet}(\Omega^r)$
 by setting
 $$\shQ_\top^{k,l}(\Omega^r)=\bigoplus_{\underbrace{\tau_0\ra\dots\ra\tau_k}_{e}}
 q_{\tau_k,*}\shO_{X_{\tau_k}}(lZ_e)\otimes_\kk 
 \bigg(\langle\bigwedge^\top \widehat{T}_{\check\Delta_{e}} \rangle 
     \cap \bigwedge^{r+l+1}\widehat{\check\Lambda}_{v_{\tau_0},\k}
 \mod \bigwedge^{r+l+1} \widehat{\Delta}_{e}^\perp
 \bigg)$$
 for $l>0$ and $\shQ_\top^{k,0}(\Omega^r)=0$.
\end{definition}

Note that the differential $\delta$ is trivial on
$\shQ_\top^{\bullet,\bullet}(\Omega^r)$. So to see that it is a 
subcomplex, we just check closedness under $d_\bct$. 
This follows from the closedness under the change of 
vertex operator $\Phi$ and under $d_e^{\hat{e}}$. The latter is
because
$\langle\bigwedge^\top \widehat{T}_{\check\Delta_{e}} \rangle
\subseteq
\langle\bigwedge^\top \widehat{T}_{\check\Delta_{\hat{e}}} \rangle.$

\begin{definition} 
We define the subcomplex
$Q_\top^{\bullet,\core{\bullet}}(\Omega^r)$
of $\Gamma(X,\shQ_\top^{\bullet,\bullet}(\Omega^r))$ by replacing
each $\Gamma(X_{\tau_k},\shO_{X_{\tau_k}}(lZ_e))$ in $\Gamma(X,\shQ_\top^{k,l}(\Omega^r))$
by $\Gamma^\core{l}(Z_e)$. Again, $\delta$ is trivial, so we have to show closedness under $d_\bct$.
For this, we need to show that the image of $\Gamma^\core{l}(Z_e)$ under the restriction map
$d_e^{\hat{e}}:\Gamma(X_{\tau_k},\shO_{X_{\tau_k}}(lZ_e))\ra \Gamma(X_{\tau_{k+1}},\shO_{X_{\tau_{k+1}}}(lZ_{\hat{e}}))$
is contained in $\Gamma^\core{l}(Z_{\hat{e}})$. The Newton polytope of $lZ_{\hat{e}}$ is a face of the 
Newton polytope of $lZ_{e}$, so this follows from Lemma~\ref{corefunc} and 
the definition of $\Gamma^\core{l}(Z_{e})$.
\end{definition}

\begin{lemma} \label{topinjection}
Assume that $X$ is a Fermat toric log CY space.
For $k\ge 0, l>0$, we have an injection
$$ Q_\top^{k,\core{l}}(\Omega^r)\hra
H^l_\delta\Gamma(X,\shQ^{k,\bullet}(\Omega^r)).$$
\end{lemma}

\begin{proof} This follows from Thm.~\ref{cohotau}, Prop.~\ref{Fermatiso} and
Lemma~\ref{compareRs}.
\end{proof}

\begin{lemma} \label{criteriondeg}
  Let $\bigoplus_{p,q}K^{p,q}$ be a double complex with differentials 
  $d',d''$ which is bounded in $p$ and $q$, 
  denote by $D=d'+(-1)^p d''$ the differential on the total complex
  $\Tot^\bullet(K^{\bullet,\bullet})$.
  Assume the following criterion:
  $$\begin{array}{lcc}
  \begin{array}{l}
  \hbox{For each }x\in K^{p,q}\hbox{ with }d''x=0$ and $d'x=d''y\\
  \hbox{for some }y\in K^{p+1,q-1},
   \hbox{there is some }\\
   z\in K^{p,q-1}\hbox{ such that }d'(x+d''z)=0.
  \end{array}
  &\qquad &
  \begin{array}{ccccc}
  0\\
   \hbox{\rotatebox{90}{$\!\!\mapsto$}}\\
  x&\stackrel{d'}{\mapsto}& d''y\\
   \hbox{\rotatebox{90}{$\!\!\leadsto$}}&&\hbox{\rotatebox{90}{$\!\!\mapsto$}}\\
  z&&y
  \end{array}
  \end{array}
  $$
  Then, its first spectral sequence degenerates at
  $$E_2^{p,q}:H_{d'}^pH_{d''}^q(K^{\bullet,\bullet})
   \Rightarrow H^{p+q}_D(\Tot^\bullet(K^{\bullet,\bullet})).$$
\end{lemma}

\begin{proof}
  Let $[\cdot]_k$ mean taking the class in $E_k$.
  For $x_1\in K^{p,q}$, the image of $[x_1]_k$ under the 
  differential $d_k$
  is given by $[d'(x_{k})]_k$ for some zig-zag
  \begin{center}
   \begin{picture}(0,0)%
\includegraphics{zigzag.pstex}%
\end{picture}%
\setlength{\unitlength}{4144sp}%
\begingroup\makeatletter\ifx\SetFigFontNFSS\undefined%
\gdef\SetFigFontNFSS#1#2#3#4#5{%
  \reset@font\fontsize{#1}{#2pt}%
  \fontfamily{#3}\fontseries{#4}\fontshape{#5}%
  \selectfont}%
\fi\endgroup%
\begin{picture}(3810,2020)(706,-1484)
\put(901,389){\makebox(0,0)[lb]{\smash{{\SetFigFontNFSS{12}{14.4}{\familydefault}{\mddefault}{\updefault}{\color[rgb]{0,0,0}$0$}%
}}}}
\put(901,-61){\makebox(0,0)[lb]{\smash{{\SetFigFontNFSS{12}{14.4}{\familydefault}{\mddefault}{\updefault}{\color[rgb]{0,0,0}$x_1$}%
}}}}
\put(1801,-61){\makebox(0,0)[lb]{\smash{{\SetFigFontNFSS{12}{14.4}{\familydefault}{\mddefault}{\updefault}{\color[rgb]{0,0,0}$d'x_1$}%
}}}}
\put(1801,-511){\makebox(0,0)[lb]{\smash{{\SetFigFontNFSS{12}{14.4}{\familydefault}{\mddefault}{\updefault}{\color[rgb]{0,0,0}$x_2$}%
}}}}
\put(2251,-511){\makebox(0,0)[lb]{\smash{{\SetFigFontNFSS{12}{14.4}{\familydefault}{\mddefault}{\updefault}{\color[rgb]{0,0,0}$\mapsto$}%
}}}}
\put(2701,-511){\makebox(0,0)[lb]{\smash{{\SetFigFontNFSS{12}{14.4}{\familydefault}{\mddefault}{\updefault}{\color[rgb]{0,0,0}$d'x_2$}%
}}}}
\put(2701,-961){\makebox(0,0)[lb]{\smash{{\SetFigFontNFSS{12}{14.4}{\familydefault}{\mddefault}{\updefault}{\color[rgb]{0,0,0}$\ddots$}%
}}}}
\put(3601,-961){\makebox(0,0)[lb]{\smash{{\SetFigFontNFSS{12}{14.4}{\familydefault}{\mddefault}{\updefault}{\color[rgb]{0,0,0}$\ddots$}%
}}}}
\put(4051,-1411){\makebox(0,0)[lb]{\smash{{\SetFigFontNFSS{12}{14.4}{\familydefault}{\mddefault}{\updefault}{\color[rgb]{0,0,0}$\mapsto$}%
}}}}
\put(1801,-286){\makebox(0,0)[lb]{\smash{{\SetFigFontNFSS{12}{14.4}{\familydefault}{\mddefault}{\updefault}{\color[rgb]{0,0,0}\,\rotatebox{90}{$\!\!\mapsto$}}%
}}}}
\put(3601,-1186){\makebox(0,0)[lb]{\smash{{\SetFigFontNFSS{12}{14.4}{\familydefault}{\mddefault}{\updefault}{\color[rgb]{0,0,0}\,\rotatebox{90}{$\!\!\mapsto$}}%
}}}}
\put(901,164){\makebox(0,0)[lb]{\smash{{\SetFigFontNFSS{12}{14.4}{\familydefault}{\mddefault}{\updefault}{\color[rgb]{0,0,0}\,\rotatebox{90}{$\!\!\mapsto$}}%
}}}}
\put(1351,-61){\makebox(0,0)[lb]{\smash{{\SetFigFontNFSS{12}{14.4}{\familydefault}{\mddefault}{\updefault}{\color[rgb]{0,0,0}$\mapsto$}%
}}}}
\put(721,164){\makebox(0,0)[lb]{\smash{{\SetFigFontNFSS{12}{14.4}{\familydefault}{\mddefault}{\updefault}{\color[rgb]{0,0,0}{\scriptsize$d''$}}%
}}}}
\put(1411, 74){\makebox(0,0)[lb]{\smash{{\SetFigFontNFSS{12}{14.4}{\familydefault}{\mddefault}{\updefault}{\color[rgb]{0,0,0}{\scriptsize$d'$}}%
}}}}
\put(3601,-1411){\makebox(0,0)[lb]{\smash{{\SetFigFontNFSS{12}{14.4}{\familydefault}{\mddefault}{\updefault}{\color[rgb]{0,0,0}$x_{k}$}%
}}}}
\put(4501,-1411){\makebox(0,0)[lb]{\smash{{\SetFigFontNFSS{12}{14.4}{\familydefault}{\mddefault}{\updefault}{\color[rgb]{0,0,0}$d'x_{k}$}%
}}}}
\end{picture}%

  \end{center}
  The criterion implies that for some representative $x_1$
  all $x_k$ for $k\ge 2$ can be chosen to be zero and thus
  $d_k=0$ for $k\ge 2$.
\end{proof}

\begin{proof}[Proof of Theorem~\ref{maintheorem},b)] 
We are going to apply Lemma~\ref{criteriondeg} to the double complex \linebreak
$\Gamma(X,\shQ^{\bullet,\bullet}(\Omega^r))$. Let $\delta$ denote the differential in
the second direction.
Suppose $x\in\Gamma(X,\shQ^{k,l}(\Omega^r))$ with $\delta x=0$ and
$d_\bct x=\delta y$ for some $y\in\Gamma(X,\shQ^{k+1,l-1}(\Omega^r))$.
For $l=0$ we have $y=0$ and there is nothing to show, so assume $l>0$.
By Prop.~\ref{koszulcoho} and Prop.~\ref{Fermatiso}, changing $x$ by adding
a $\delta$-coboundary, we may assume that $x\in Q_\top^{k,\core{l}}(\Omega^r)$.
Then $d_\bct x\in Q_\top^{k+1,\core{l}}(\Omega^r)$, and the injection
$$ Q_\top^{k+1,\core{l}}(\Omega^r)\hra H^l_\delta\Gamma(X,\shQ^{k+1,\bullet}(\Omega^r))$$
from Lemma~\ref{topinjection}
shows that $y=0$. This establishes the hypothesis of Lemma~\ref{criteriondeg}
which we may now apply.
\smallskip
\end{proof}

\section{Mirror symmetry of stringy and affine Hodge numbers}
\subsection{Base change of the affine Hodge groups}
\label{sec_affloghodge}
Recall from [\cite{grosie2}, Lemma 3.12], for $v\in\P^{[0]}$, the identification
$\Omega^r_{v}=\shO_{X_v}\otimes_\k \bigwedge^r \check\Lambda_{v,\k}$.
We define
$${\bLambda}^r=\bigoplus_{v\in\P^{[0]}} (q_{v})_* \shO_{X_v}\otimes_\k \bigwedge^r \check\Lambda_{v,\k}$$
which becomes a complex $({\bLambda}^\bullet,d)$ under exterior differentiation.
There is a barycentric resolution in the sense of Section~\ref{A_bct} 
for this complex via
$$\SC^k(\bLambda^r)=\bigoplus_{e:\tau_0\ra...\ra\tau_k} 
\bigoplus_{\atop{g:v\ra\tau_0}{v\in\P^{[0]}}}
(q_{\tau_k})_*F_s(e\circ g)^*\shO_{X_v}\otimes_\k \bigwedge^r \check\Lambda_{v,\k},$$
and the barycentric differential is induced by the restriction for a vertex which factors
through the extended edge and the zero map on that vertex otherwise.
We won't make use of the following lemma but give it for completeness.

\begin{lemma} \label{res_bLambda}
For each $r$, we have an exact sequence
$$0\ra{\bLambda}^r\ra \SC^0(\bLambda^r) \ra \SC^1(\bLambda^r)\ra ....$$
where the first non-trivial map is the obvious one.
\end{lemma}

\begin{proof} Injectivity at the first non-trivial term is obvious.
Given any $\tau$ and a geometric point 
$x\in\Int(X_\tau)$, we have 
$\SC^0(\bLambda^r)_x= \bigoplus_{\atop{g:v\ra\tau_0,\tau_0\ra\tau}{v\in\P^{[0]}}}
(q_{\tau_k,*}F_s(g)^*\shO_{X_v})_x\otimes_\k \bigwedge^r \check\Lambda_{v,\k}.$
An element of this maps to zero under $d_\bct$ if and only if it is a compatible
collection which implies that it lifts to 
$\bigoplus_{v\ra\tau} (q_{v,*}\shO_{X_v})_x\otimes_\k \bigwedge^r \check\Lambda_{v,\k}$
for each $v$ componentwise. The inverse is also true, so we have exactness also at the
second non-trivial term.
The exactness of the tail follows from Lemma~\ref{baryacyclic} upon verifying 
criterion $\op{(L)}$ which is easy.
\end{proof}

\begin{lemma} \label{specialbottomterm}
Given a c.i.t. toric log CY space $X$, there is an acyclic 
resolution $\shI^{\bullet,\bullet,\bullet}$ with augmentation
$$0\ra \SC^\bullet(\Omega^\bullet) \ra \shI^{\bullet,\bullet,0}$$
such that the exterior differential $d$ is trivial on
$\Gamma(X,\shI^{\bullet,\bullet,0})$.
\end{lemma}

\begin{proof} It suffices to construct an injective map of complexes
$\SC^\bullet(\Omega^\bullet) \ra \shI^{\bullet,\bullet,0}$ where $\shI^{k,r,0}$ is 
acyclic for each $k,r$. The remainder 
of $\shI^{\bullet,\bullet,\bullet}$ can then be added by an injective resolution, e.g., 
Godement's canonical resolution.
We claim that we may just take $\shI^{k,r,0}=\SC^k(\bLambda^r)$. 
By the c.i.t. hypothesis and by what we said at the beginning of 
Section~\ref{sec_barycentric}, namely that we have the result [\cite{grosie2}, Prop.~3.8],
i.e., for each $g:v\ra\tau_0$ and $e:\tau_0\ra\tau_k$, 
we have an inclusion $F_s(e)^*\Omega^r/\Tors\hra F_s(e\circ g)^*\Omega^r=F_s(e\circ g)^*\shO_{X_v}\otimes_\k \bigwedge^r \check\Lambda_{v,\k}$.
We may use this to get the injection $\SC^\bullet(\Omega^\bullet) \hra \shI^{\bullet,\bullet,0}$
termwise as 
$$ F_s(e)^*\Omega^\bullet_{\tau_0}/\Tors \ra
\bigoplus_{\atop{g:v\ra\tau_0}{v\in\P^{[0]}}}
(q_{\tau_k})_*F_s(e\circ g)^*\shO_{X_v}\otimes_\k \bigwedge^r \check\Lambda_{v,\k}.$$
Because $\Gamma(X,\SC^k(\bLambda^r))$ consists of constant differential forms only, $d$ is trivial on it.
Acyclicity is apparent.
\end{proof}

\begin{proof}[Proof of Theorem~\ref{Thm_affinlog}] 
We show that, for $e:\tau_1\ra\tau_2$,
\begin{equation} \label{Gammaincit}
H^0(X_{\tau_2},F(e)^*\Omega^p_{\tau_1}/\Tors)
=\Gamma(W_e,i_*\bigwedge^p\check\Lambda\otimes_\ZZ\k).
\end{equation}
For the h.t. case this is part of Thm.~\ref{cohotau}. 
We can extended this to the c.i.t. case as follows. 
Recall that there is a set of Cartier divisors 
$Z_{\tau_1,1},...,Z_{\tau_1,t}$ which are the 
reduced components of the closure of $Z\cap\Int(X_\tau)$. We set
$Z_{e,i}=Z_{\tau_1,i}\cap X_{\tau_2}$ which might be empty. 
The empty ones won't play a role in the following, so let us exclude them.
Fix some $\P^{[0]}\ni v\stackrel{g}{\lra}\tau_1$ and define $\Omega^p_i$ by the exact sequence
$$0\ra \Omega^p_i\ra F_{s}(e\circ g)^*\Omega^p_v
\stackrel{\delta_i}{\lra}\Omega^{p-1}_{(Z_{e,i})^{\dagger}/\k^{\dagger}}\ra 0$$
where $\delta_i$ is the map $\delta$ in Prop.~\ref{delta0kernel} composed with the $i$th
projection. By loc.cit., we then get
$(F_{s}(e)^*\Omega^p_{\tau_1})/\Tors = \bigcap_{i=1}^t \Omega^p_i.$
By Thm.~\ref{cohotau}, we have
$$\Gamma(X_{\tau_2},\Omega^p_i) = \big(\bigwedge^p\check\Lambda_v\otimes_\ZZ\k\big)^{G_i}$$
where $G_i$ is the group of those local monodromy transformations which are 
transvections that fix $(\Delta_{\tau_1,i}\cap \tau_2)^\perp$ and shear by a vector in 
$T_{(\check\Delta_{\tau_2,i})\cap\tau_1^\perp}$. Because the monodromy 
on $W_e\backslash\Delta$ is generated by $\{G_i\,|\,1\le i\le t\}$, we have 
$$\Gamma(W_e,i_*\bigwedge^p\check\Lambda\otimes_\ZZ\k)=\bigcap_{i=1}^t
\big(\bigwedge^p\check\Lambda_v\otimes_\ZZ\k\big)^{G_i}$$
and conclude (\ref{Gammaincit}) by the left-exactness of the functor $\Gamma$.
To prove a), note that 
$$H^{p,q}_\llog(X)=H^q(X,\Omega^p)=\HH^q(X,\SC^\bullet(\Omega^p)).$$
Let $\shJ^{\bullet,\bullet}$ be an injective resolution of $\SC^\bullet(\Omega^p)$ with the augmentation
$\SC^\bullet(\Omega^p)\hra \shJ^{\bullet,0}$.
If we denote by $D$ the total differential of the double complex $\Gamma(X,\shJ^{\bullet,\bullet})$ then
$$H^{p,q}_\llog(X)=H^q_D \Gamma(X,\shJ^{\bullet,\bullet}).$$
There is an injection
$$\frac{\ker (D|_{\Gamma(X,\shJ^{q,0})})}{\im D\cap \Gamma(X,\shJ^{q,0})} \hra H^q_D \Gamma(X,\shJ^{\bullet,\bullet})$$
The left hand side can be rewritten as
$H^q_{d_\bct} \Gamma(X,\SC^\bullet(\Omega^p)).$
By what we said before this coincides with the \v{C}ech cohomology group
$\check{H}^q(\{W_\tau\,|\, \tau\in\P\},i_*\bigwedge^p\check\Lambda\otimes_\ZZ\k)$
 and we are done with part a).

The proof of b) is similar. Let $\shI^{\bullet,\bullet,\bullet}$ be a resolution as given 
in Lemma~\ref{specialbottomterm} and let $D'$ denote the total differential on $\Gamma(X,\shI^{\bullet,\bullet,\bullet})$.
We have 
$\HH^{k}(X,\Omega^\bullet)=H^k_{D'} \Gamma(X,\shI^{\bullet,\bullet,\bullet}).$
Because $d$ is trivial on $\Gamma(X,\shI^{\bullet,\bullet,0})$ by arguing as in a) we get for each $p,q$ with $p+q=k$
an injection 
$$H^{p,q}_\aff(X) \hra \HH^{k}(X,\Omega^\bullet).$$
Once again from the triviality of $d$ on $\Gamma(X,\shI^{\bullet,\bullet,0})$ one concludes that these injections can be extended to 
their direct sum as required in the assertion.
\end{proof}

\begin{definition} For a c.i.t. toric log CY space $X$, we call 
$T^{p,q}_\llog(X)\, =\, H^{p,q}_\llog(X)/H^{p,q}_\aff(X)$
the \emph{log twisted sectors}.
\end{definition}

As explained in [\cite{grosie2},~Cor.~3.24], by [\cite{grosie1},~Prop.~1.50],
we obtain, for a general $(B,\P)$, the following.
\begin{theorem}[Gross, Siebert] \label{affduality}
Assume that the holonomy 
of $B$ is contained in $\op{SL}_n(\ZZ)\ltimes\ZZ^n$ where $n={\dim B}$.
Let $\varphi$ be some multi-valued strictly convex piecewise linear 
function on $(B,\P)$ and $(\check B,\check\P,\check\varphi)$ be the
discrete Legendre dual. If $\check X$ is a toric log CY space with dual
intersection complex $(\check B,\check\P)$ then
$$H_\aff^{p,q}(X)\cong H_\aff^{n-p,q}(\check X).$$
\end{theorem}

So we see that the affine Hodge numbers fulfill mirror symmetry duality.
The duality for the ordinary Hodge numbers follows for the case where the twisted sectors vanish.
We are going to consider what happens if this is not the case.


\subsection{Twisted sectors in low dimensions} \label{Sec_twistsec}
We now consider the situation up to dimension $4$ as defined in 
Theorem~\ref{Thm_logdeRham}. 

\begin{lemma} \label{elemthreesimplex}
Let $\check\Delta$ be an elementary simplex with $\dim\check\Delta=3$. 
Let $f$ be non-degenerate.
For $k>0$, we have $R_0(f,C(\check\Delta))_k=R_1(f,C(\check\Delta))_k.$ 
Moreover, $R_1(f,C(\check\Delta))_k=0$ for $k\neq 2$.
\end{lemma}

\begin{proof}
The natural inclusion $R_1(f,C(\check\Delta))_k\hra R_0(f,C(\check\Delta))_k$
becomes an isomorphism for $k>0$ because each lattice point in 
$k\check\Delta$ which isn't a sum of a lattice point in $(k-1)\check\Delta$
and one in $\check\Delta$ lies in the relative interior of $k\check\Delta$ because otherwise 
it would have to be in some facet of $\check\Delta$. This, however, can be excluded by 
the fact that each facet of $\check\Delta$ is
is a two dimensional elementary simplex, thus a standard simplex, and Lemma~\ref{standardequiv}.

The second assertion then works out as follows. It is clear for $k=0$. 
It follows from elementarity of $\check\Delta_\tau$ and Prop.~\ref{Fermatiso} for $k=1$.
The cases $k=3,4$ then follow by the 
pairing given in \cite{bormav}, Prop.~6.7.
\end{proof}

\begin{proof}[Proof of Theorem~\ref{Thm_logdeRham}] 
By Thm.~\ref{Thm_affinlog}, $H_\aff:=\bigoplus_{p,q} H^{p,q}_\aff$ injects in the $E_1$-term of the hypercohomology spectral sequence of $\Omega^\bullet$
and survives to the limit. Thus, $\ker d_1\cap H_\aff=0$ and $\im d_1\cap H_\aff=0$.
Cases a) and c) follow if we show that
$H^{p,q}_\llog\neq H^{p,q}_\aff$ for only one pair $p,q$.
For a) this is $p=q=1$ which we deduce from Thm.~\ref{maintheorem}.
For c) the exceptional pair is $p=q=2$ which we also deduce from Thm.~\ref{maintheorem} together with
Lemma~\ref{standardequiv} to see that only three-dimensional 
simplices contribute to higher cohomology terms and eventually Lemma~\ref{elemthreesimplex} and 
Thm.~\ref{maintheorem} to locate the contribution.
Similarly to show b), 
we demonstrate that $H^{p,q}_\llog\neq H^{p,q}_\aff$ only for $(p,q)\in\{(1,2),(2,1)\}$.
We compute the log twisted sectors 
via Thm.~\ref{cohotau} and Lemma~\ref{thmstratafunct}. 
We keep the convention that $\omega$'s denote
one-dimensional and $\tau$'s two-dimensional faces.
Note that $R(Z_\omega)_1=\Gamma^{\core{1}}(Z_\omega)$ 
contains a canonical subspace induced from lattice points in the
relative interior of $\check\Delta_\omega$ 
which we denote by $\Gamma^{\ocore{1}}(Z_\omega)$. 
For $e:\omega\ra\tau$, $Z_e=Z_\tau$, we have 
$\dim\Delta_e=\dim\Delta_\omega=1$ and
$\dim\check\Delta_e=\dim\check\Delta_\tau=1$. Given $g:v\ra\omega$, we obtain
\[
\begin{array}{cclcccc}
\frac{\langle\bigwedge^\top T_{\check\Delta_{\tau}} \rangle 
    \cap \bigwedge^{2}\check\Lambda_{v,\k}}
    {\langle\bigwedge^\top T_{\check\Delta_{\tau}} \rangle
    \cap \bigwedge^{2}\Delta_{\omega}^\perp}
    &\cong& \check\Lambda_{v,\k}/\Delta_{\omega}^\perp \cong\k &\qquad
\frac{\langle\bigwedge^\top T_{\check\Delta_{\tau}} \rangle 
    \cap \bigwedge^{3}\check\Lambda_{v,\k}}
    {\langle\bigwedge^\top T_{\check\Delta_{\tau}} \rangle
    \cap \bigwedge^{3}\Delta_{\omega}^\perp}\!&\cong&\!\k\\[2mm]

\frac{\langle\bigwedge^\top T_{\check\Delta_{\omega}} \rangle 
    \cap \bigwedge^{2}\check\Lambda_{v,\k}}
    {\langle\bigwedge^\top T_{\check\Delta_{\omega}} \rangle
    \cap \bigwedge^{2}\Delta_{\omega}^\perp}&\cong&
    \left\{
      \atop{\check\Lambda_{v,\k}/\Delta_{\omega}^\perp\cong\k^1
        \hbox{ for }\dim\check\Delta_{\omega}=1}{
	    0\qquad\qquad\,
        \hbox{ for }\dim\check\Delta_{\omega}=2}\right. & \qquad
\frac{\langle\bigwedge^\top T_{\check\Delta_{\omega}} \rangle 
    \cap \bigwedge^{3}\check\Lambda_{v,\k}}
    {\langle\bigwedge^\top T_{\check\Delta_{\omega}} \rangle
    \cap \bigwedge^{3}\Delta_{\omega}^\perp}\!&\cong&\!\k\\[2mm]

\frac{\langle\bigwedge^\top T_{\check\Delta_{\tau}} \rangle 
    \cap \bigwedge^{2}\check\Lambda_{v,\k}}
    {\langle\bigwedge^\top T_{\check\Delta_{\tau}} \rangle
    \cap \bigwedge^{2}\Delta_{\tau}^\perp}
    &\cong& \left\{
      \atop{\check\Lambda_{v,\k}/\Delta_{\tau}^\perp\cong\k^1
        \hbox{ for }\dim\Delta_{\tau}=1}{
	    \check\Lambda_{v,\k}/\Delta_{\tau}^\perp\cong\k^2
        \hbox{ for }\dim\Delta_{\tau}=2}\right. &\qquad
\frac{\langle\bigwedge^\top T_{\check\Delta_{\tau}} \rangle 
    \cap \bigwedge^{3}\check\Lambda_{v,\k}}
    {\langle\bigwedge^\top T_{\check\Delta_{\tau}} \rangle
    \cap \bigwedge^{3}\Delta_{\tau}^\perp}\!&\cong&\!\k.\\      
\end{array}
\]
We consider the differential $d_1$ on cohomology degree $q=1,2$ 
of the $E_1$-term of the hypercohomology spectral sequence 
of $\SC^\bullet(\Omega^p)$ for $p=1,2$.
\[
\newlength{\mylength}\settowidth{\mylength}{0}
\begin{array}{lll}
q=1,p=1&:& \hspace{-1.58mm}\left(\bigoplus_{\tau\in\P^{[2]}}R(Z_\tau)_1
     \otimes\check\Lambda_{v,\k}/ \Delta_{\tau}^\perp\right) \oplus 
     \left(\bigoplus_{\hspace{-1.5mm}
     \atop{\omega\in\P^{[1]}}{\dim\check\Delta_{\omega}=1}}R(Z_\omega)_1
     \otimes \check\Lambda_{v,\k}/\Delta_{\omega}^\perp\right)\\
    &&\hspace{\mylength}\lra \bigoplus_{\omega\ra\tau}R(Z_\tau)_1
     \otimes \check\Lambda_{v,\k}/\Delta_{\omega}^\perp\\
q=1,p=2&:& \hspace{-1.58mm}\left(\bigoplus_{\omega\in\P^{[1]}}R(Z_\omega)_1\right) 
    \oplus \left(\bigoplus_{\tau\in\P^{[2]}}R(Z_\tau)_1\right)
    \lra \bigoplus_{\omega\ra\tau}R(Z_\tau)_1\\
q=2,p=1&:& \bigoplus_{\omega\in\P^{[1]}}R(Z_\omega)_2    
    \lra 0\\   
q=2,p=2&:& 0 \lra 0
\end{array}
\]
where the sum on the right is only over edges $\omega\ra\tau$
which are contained in $\Delta$.
If we show that the first map is injective and the second surjective, we are done.

We show the surjectivity of the second map. We can rewrite the map as
$\bigoplus_\omega \Gamma^{\ocore{1}}(Z_\omega)\oplus 
\bigoplus_K V^0_K \ra \bigoplus_K V^1_K$
where $K$ runs over the connected components of $\Delta\backslash\Delta^0$.
We show that $V^0_K\ra V^1_K$ is surjective for each $K$. Note that
both spaces are a direct sum of spaces isomorphic to $R(Z_{\tau_K})_1$ for a suitable
$\tau_K$ in $K$. It is not hard to see that $V^0_K\ra V^1_K$ is isomorphic to
the \v{C}ech complex of a locally constant sheaf on $K$ with 
fibre $\k^{\dim R(Z_{\tau_K})_1}$. The contractibility of $K$ therefore implies the desired 
surjectivity.

A similar argument works for the injectivity of the first map by quasi-isomorphically projecting it to
\[
\bigoplus_{\atop{\omega\in\P^{[1]}}{\dim\check\Delta_{\omega}=1}}R(Z_\omega)_1
     \otimes \check\Lambda_{v,\k}/\Delta_{\omega}^\perp\lra \bigoplus_{\tau\in\P^{[2]}}R(Z_\tau)_1\otimes
     \coker\left(
     \check\Lambda_{v,\k}/ \Delta_{\tau}^\perp\hra
     \bigoplus_{\omega\ra\tau}
     \check\Lambda_{v,\k}/\Delta_{\omega}^\perp\right)\\
\]
and identifying this map with $\bigoplus_K W^0_K \ra \bigoplus_K W^1_K$ for suitable $W^0_K $, $W^1_K$, each of which is isomorphic to the dual of a \v{C}ech complex of a locally constant sheaf with fibre $\k^{\dim R(Z_{\tau_K})_1}$ on $K$. 
\end{proof}

Note that we only needed the weaker criterion of contractibility of those components $K$ of 
$\Delta\backslash\Delta^0$ where $\dim R(Z_{\tau_K})_1>0$. On the other hand, if this is not given for one $K$ and the locally constant sheaves constructed in the proof have global sections on $K$, we have $T^{1,1}_\llog(X)\neq 0 \neq T^{2,2}_\llog(X)$.

\begin{corollary} \label{twistedcomputed}
For the cases considered in Theorem~\ref{Thm_logdeRham}, 
at most the following log twisted sectors are non-trivial
\begin{itemize}
\item[a)] $T^{1,1}_\llog(X)\cong \bigoplus_{\omega\in\P^{[1]}} R(Z_\omega)_1$\\
\item[b)] 
$T^{1,2}_\llog(X) 
\cong \bigoplus_{\omega\in\P^{[1]}}R(Z_\omega)_2 
\oplus \bigoplus_K R(Z_{\tau_K})_1$ and\\
$T^{2,1}_\llog(X) 
\cong \bigoplus_{\omega\in\P^{[1]}} \Gamma^{\ocore{1}}(Z_\omega)
\oplus \bigoplus_K R(Z_{\tau_K})_1$\\
\item[c)] $T^{2,2}_\llog(X)\cong \bigoplus_{\tau\in\P^{[2]}} R(Z_\tau)_2$
\end{itemize}
\end{corollary}

Note that that in b) $R(Z_\omega)_2\cong \Gamma^{\ocore{1}}(Z_\omega)$. It is expected that the Picard-Lefschetz operator maps $T^{2,1}_\llog(X)$ isomorphically to $T^{1,2}_\llog(X)$.

\begin{proof}[Proof of Theorem~\ref{mirror_affinestringy}] 
Part a) is the combination of Cor.~\ref{basechange} and Theorem~\ref{Thm_logdeRham}.
To prove part b), note that the general fibre $X_t$ has isolated
singularities in these cases. Each singularity is described by a local model as referred to in Prop.~\ref{localmodels}. 
See also (\cite{grosie2}, Prop.~2.2). The degeneration is locally $\Spec \k[ K^\dual\cap (M_\tau\oplus\ZZ^2)]\ra \Spec \k[\NN]$
where $K$ is the cone over $(\tau\times \{e_1\})\cup (\Delta_\tau\times \{e_2\})$. Here, 
$\tau,\Delta_\tau\subset N_\tau\otimes\RR$, $N_\tau$ is a lattice of rank $\dim B-1$, $M_\tau=\Hom(N_\tau,\ZZ)$ and the generator of $\NN$ maps to $e_1^*$.
The general fibre is thus locally given by $\k[C(\Delta_\tau)^\dual\cap (M_\tau\oplus\ZZ)]$. So we have a singularity in
$X_t$ for each non-standard inner monodromy polytope $\Delta_\tau$. In case c), these
are non-standard elementary 3-simplices. 
In case a), these are intervals of length greater than one.\\
Borisov and Mavlyutov have identified a space whose dimension 
gives the difference $h^{p,q}_\st-h^{p,q}$ (see \cite{bormav},~Def.~8.1). For each singularity this is
$R_1( \omega_{\check\tau},C(\Delta_{\check\tau}))$ for some general $\check \omega_{\check\tau}$.
Under mirror symmetry, the K\"ahler parameter $\omega_{\check\tau}$ is supposed to become
the log moduli parameter $f_\tau$. 
Even though we cannot make this rigorous at the moment, we still have 
$\dim R_1(\omega_{\check\tau},C(\Delta_{\check\tau}))=\dim R_1(f_\tau,C(\check\Delta_\tau))$ because
an inner monodromy polytope of $(\check B,\check\P)$ is an outer monodromy polytope of
$(B,\P)$, i.e., $\Delta_{\check\tau}=\check\Delta_{\tau}$. Using Cor.~\ref{twistedcomputed}, Lemma~\ref{compareRs} and Lemma~\ref{elemthreesimplex},
we deduce the result.
\end{proof}

\section{Appendix}

\subsection{Barycentric complexes} \label{A_bct}
For convenience, we include here a 
slight modification of [\cite{grosie1}, A.1].
Let $\Xi$ be a $d$-dimensional polytope and $\underline{Pair}$ be the finite category with
$$
\begin{array}{ll}
\hbox{objects:}  & \{(\sigma_1,\sigma_2)\,|\, \sigma_1\subseteq\sigma_2\subseteq\Xi\hbox{ are faces}\} \\
\hbox{morphisms:} & (\tau_1,\tau_2)\ra (\sigma_1,\sigma_2) 
  \qquad\hbox{ for } \sigma_1\subseteq \tau_1,\tau_2\subseteq\sigma_2
\end{array}
$$
Let $\underline{Ab}$ denote the category of abelian groups. We assume to have a functor
$$\begin{array}{rcl}
\underline{Pair}&\ra& \underline{Ab}\\
           e=(\sigma_1,\sigma_2) &\mapsto& M_e.
\end{array}$$
Note that there is at most one morphism between any two objects $e_1,e_2$ in 
$\underline{Pair}$ whose image under this functor we denote by $\varphi_{e_1e_2}$.
Whenever the source is clear we will also write $\varphi_{e_2}$. 
The \emph{barycentric cochain complex}
$(C_\bct^\bullet, d_\bct^\bullet)$ associated with the image of this functor is the complex of abelian groups
$C^k=\bigoplus_{\sigma_0\subsetneq \sigma_1\subsetneq...\subsetneq\sigma_k} M_{(\sigma_0,\sigma_k)}$ 
with differentials
$$(d_{\bct}^k(f))_{\sigma_0\sigma_1...\sigma_{k+1}}
=\sum_{i=0}^{k+1} (-1)^i\varphi_{(\sigma_0,\sigma_{k+1})}(f_{\sigma_0...\breve{\sigma}_i...\sigma_{k+1}})$$
where $\breve{a}$ means the omission of $a$. It is easy to check that this is a complex, i.e., 
$d_\bct^{k+1}\circ d_\bct^k=0$. Assume we have some subset $U$ of the set of objects 
of $\underline{Pair}$. We call an element $(f_e)_e\in \bigoplus_{e\in U} M_e$ a \emph{compatible collection}
if, for each $e_1,e_2,\hat{e}\in U$ with morphisms $e_1\ra\hat{e},e_2\ra\hat{e}$, 
we have $\varphi_{\hat{e}}f_{e_1}=\varphi_{\hat{e}}f_{e_2}$.
We consider the following criterion
\begin{itemize} 
\item[(L)]
For each $\sigma_0\subseteq \sigma_{k-1}$, every compatible collection
$(f_e)_e\in \bigoplus_{\sigma_{k}\supsetneq\sigma_{k-1}} M_{(\sigma_0,\sigma_k)}$
lifts, i.e., there is some $g\in M_{(\sigma_0,\sigma_{k-1})}$ such that 
$$f_{(\sigma_0,\sigma_{k})}=\varphi_{(\sigma_0,\sigma_{k})}g\qquad\hbox{ for each }{(\sigma_0,\sigma_{k})}.$$
\end{itemize}

\begin{lemma} \label{baryacyclic}
If $(M_e)_e$ satisfies $\op{(L)}$ then the associated barycentric complex is acyclic.
\end{lemma}

\begin{proof} We wish to write a cocyle $(f_{\sigma_0...\sigma_k})_{\sigma_0...\sigma_k}$ as a coboundary of
a $(k-1)$-cochain $(g_{\sigma_0...\sigma_{k-1}})_{\sigma_0...\sigma_{k-1}}$. We construct
the $g_{\sigma_0...\sigma_{k-1}}$ by descending induction on $m=\dim \sigma_{k-1}=d+1,...,0$. The induction
hypothesis is that
$$f_{\sigma_0...\sigma_k}=\sum_{i=0}^k(-1)^i\varphi_{(\sigma_0,\sigma_k)}
(g_{\sigma_0...\breve{\sigma}_i...\sigma_k})$$
whenever $\dim\sigma_{k-1}\ge m$. The base case with $m=d+1$ is empty because $\dim\Xi=d$.
For the induction step consider some 
$\sigma_0\subsetneq...\sigma_{k-1}$ with $\dim\sigma_{k-1}=m-1$. 
We want to find $g_{{\sigma_0}...{\sigma_{k-1}}}$ such that for any $\sigma_k$
containing $\sigma_{k-1}$
$$(-1)^k\varphi_{(\sigma_0,\sigma_k)}
(g_{\sigma_0...\sigma_{k-1}})
=f_{\sigma_0...\sigma_k}-\sum_{i=0}^{k-1}(-1)^i 
\varphi_{(\sigma_0,\sigma_{k+1})}g_{\sigma_0...\breve{\sigma}_i...\sigma_k}.$$
All terms on the right hand side are known inductively. We view the right hand sides for
varying $\sigma_k$ as an element of 
$\bigoplus_{\sigma_{k-1}\subseteq\sigma_k} M_{(\sigma_0,\sigma_k)}$.
If we show that this constitutes a compatible collection, we get $g_{\sigma_0...\sigma_{k-1}}$ from
criterion $\op{(L)}$ and are done with the proof. So let us do this and assume 
we have some $\sigma_{k+1}$ containing $\sigma_{k-1}$. We need to show that
\begin{equation} \label{compat}
\varphi_{(\sigma_0,\sigma_{k+1})}
\left( f_{\sigma_0...\sigma_k}-\sum_{i=0}^{k-1}(-1)^i
\varphi_{(\sigma_0,\sigma_{k+1})}(g_{\sigma_0...\breve{\sigma}_i...\sigma_k})\right)
\end{equation}
is independent of $\sigma_k$ for $\sigma_{k-1}\subsetneq\sigma_k\subsetneq\sigma_{k+1}$. 
For $i\le k$ the induction hypothesis implies
$$
f_{\sigma_0\ldots\breve\sigma_i\ldots\sigma_{k+1}}
=\sum_{j=0}^{i-1} (-1)^j \varphi_{(\sigma_0,\sigma_{k+1})}(g_{\sigma_0\ldots \breve\sigma_j
\ldots\breve \sigma_i\ldots\sigma_{k+1}}) -\sum_{j=i+1}^{k+1} (-1)^j
\varphi_{(\sigma_0,\sigma_{k+1})} (g_{\sigma_0\ldots \breve\sigma_i
\ldots\breve \sigma_j\ldots\sigma_{k+1}}).
$$
Plugging this into the cocycle condition
$$\varphi_{(\sigma_0,\sigma_{k+1})}\big(f_{\sigma_0\ldots\sigma_k}\big)
=(-1)^k\sum_{i=0}^k (-1)^i
f_{\sigma_0\ldots\breve\sigma_i\ldots\sigma_{k+1}},
$$
the first term of (\ref{compat}) gives $f_{\sigma_0\ldots
\sigma_{k-1}\sigma_{k+1}}$ ($i=k$) plus a sum over
$\varphi_{(\sigma_0,\sigma_{k+1})}g_{\sigma_0\ldots\breve\sigma_i\ldots\breve\sigma_j\ldots
\sigma_{k+1}}$. For $0\le i<j<k$ the coefficient of
$\varphi_{(\sigma_0,\sigma_{k+1})}g_{\sigma_0\ldots \breve\sigma_i\ldots\breve\sigma_j\ldots
\sigma_{k+1}}$ is $(-1)^k$ times $(-1)^i(-(-1)^j)+(-1)^j(-1)^i=0$.
Contributions involving $\varphi_{(\sigma_0,\sigma_{k+1})} (g_{\sigma_0\ldots
\breve\sigma_i\ldots\sigma_k})$ come from the second term in
(\ref{compat}) and from $j=k+1$; they cancel as well. Thus
(\ref{compat}) equals
\[
f_{\sigma_0\ldots\sigma_{k-1} \sigma_{k+1}}
+(-1)^k\sum_{i=0}^{k-1} (-1)^i
(-1)^k(-\varphi_{(\sigma_0,\sigma_{k+1})}g_{\sigma_0\ldots\breve\sigma_i\ldots\breve\sigma_k
\sigma_{k+1}}).
\]
This shows the claimed independence of (\ref{compat}), and
hence the existence of $g_{\sigma_0\ldots\sigma_{k-1}}$.
\end{proof}




\end{document}